\documentclass[11pt]{amsart}

\usepackage[inner=2.4cm,outer=2.4cm,top=2.5cm,bottom=2.5cm]{geometry}
\usepackage[latin1]{inputenc}

\usepackage{color}
\usepackage{comment}
\usepackage{mathrsfs}
\usepackage{mathtools}
\usepackage{amsmath}
\usepackage{amssymb}
\usepackage{enumitem}
\usepackage{bbm}
\usepackage{esint}
\usepackage{nicefrac}
\numberwithin{equation}{section}
\usepackage[colorlinks,citecolor=green,linkcolor=red]{hyperref}

\usepackage[capitalise]{cleveref}

\usepackage{todonotes}
 
\newcommand{\N}{\mathbb{N}}

\newcommand{\R}
{\mathbb{R}}

\newcommand{\HH}{\mathcal{H}}
\newcommand{\Leb}{\mathcal{L}}

\newcommand{\ppi}{{\mbox{\boldmath$\pi$}}}

\newcommand{\sfd}{{\sf d}}

\newcommand{\frs}{{\mathfrak s}}

\newcommand{\frq}{{\mathfrak q}}
\renewcommand{\d}{{\mathrm d}}
\newcommand{\restr}[1]{\lower3pt\hbox{$|_{#1}$}}

\newcommand{\eps}{\varepsilon}  
\newcommand{\nchi}{{\raise.3ex\hbox{$\chi$}}}

\newcommand{\limi}{\varliminf}
\newcommand{\lims}{\varlimsup}

\newcommand{\fr}{\penalty-20\null\hfill$\blacksquare$}                      

\newcommand{\e}{{\rm{e}}}                         
 \newcommand{\X}{{\rm X}}

\newcommand{\mm}{\mathfrak m}

\newtheorem{theorem}{Theorem}[section]

\newtheorem{corollary}[theorem]{Corollary}
\newtheorem{lemma}[theorem]{Lemma}

\newtheorem{proposition}[theorem]{Proposition}
\newtheorem{definition}[theorem]{Definition}

\newtheorem{remark}[theorem]{Remark}
\newcounter{Counter}

\def\Xint#1{\mathchoice
{\XXint\displaystyle\textstyle{#1}}%
{\XXint\textstyle\scriptstyle{#1}}%
{\XXint\scriptstyle\scriptscriptstyle{#1}}%
{\XXint\scriptscriptstyle\scriptscriptstyle{#1}}%
\!\int}
\def\XXint#1#2#3{{\setbox0=\hbox{$#1{#2#3}{\int}$ }
\vcenter{\hbox{$#2#3$ }}\kern-.6\wd0}}

\def\dashint{\Xint-}
\newcommand{\mres}{\mathbin{\vrule height 1.6ex depth 0pt width 0.13ex\vrule height 0.13ex depth 0pt width 1.3ex}}

\title{Comparison estimates on nonsmooth spaces with\\ integrable Ricci lower bounds via localization}

\author[]{Emanuele Caputo}
\address{Mathematics Institute, Zeeman Building, University of Warwick, Coventry, CV4 7AL, UK}
\email{\url{emanuele.caputo@warwick.ac.uk}}

\author[]{Francesco Nobili} 
\address{Universit\'a di Pisa, Dipartimento di Matematica, Largo Bruno Pontecorvo 5,
56127 Pisa, IT}
\email{\url{francesco.nobili@dm.unipi.it}}

\author[]{Tommaso Rossi} 
\address{SISSA, Via Bonomea 265, 34136 Trieste, IT} 
\email{\url{trossi@sissa.it}}

\begin{document}

\begin{abstract}
We study comparison estimates on metric measure spaces admitting a synthetic variable Ricci curvature lower bound. We obtain geometric and functional inequalities assuming that the deficit of the lower bound from a given constant is sufficiently integrable. More precisely, we extend to the nonsmooth setting the Bishop-Gromov comparison, the Myers' diameter estimate and the Cheng's comparison principle for Dirichlet eigenvalues. Our analysis relies on the localization method and on one-dimensional comparison estimates for nonsmooth weighted intervals.
\end{abstract}

\maketitle

\tableofcontents

\section{Introduction}
In a smooth Riemannian manifold, a lower bound on the Ricci curvature tensor allows to deduce comparison estimates for many geometric and analytic objects, in terms of the corresponding quantities in a constant-curvature model space. Classical examples of comparison estimates include Bishop-Gromov volume comparison, Myers' maximal diameter theorem, and spectral estimates such as Lichnerowicz's spectral gap and Cheng's comparison for Dirichlet eigenvalues. We refer to the introductory books \cite{CheegerEbin08,Petersen2016} and references therein. Building on these foundational results, two major research directions have emerged:
\begin{itemize}
\item weakening the Ricci curvature lower bound from a uniform to an \emph{integral} bound;
\item extending comparison principles to \emph{nonsmooth} spaces allowing for singularities.
\end{itemize}

In the first direction, notable works \cite{PetersenSprouse98,PetersenWei97,PetersenWei01} study manifolds $(M^d,g)$ where the Ricci curvature lower bound is replaced by an \emph{integral curvature deficit}
\[ 
    \int_M |\min\{ {\rm Ric}_g^- -K,0\}|^p\,\d {\rm Vol}_g,
\]
with ${\rm Ric}_g^-(x)$ denoting the smallest eigenvalue of the Ricci tensor, $K\in\mathbb{R}$, and $p>1$.
When $p>d/2$ and the deficit is finite, or sufficiently small, comparison estimates hold up to error terms depending on this deficit \cite{Gallot88,PetersenWei97,PetersenWei01,Aubry2007,DaiWeiZhang18,WangWei19}.
Another related but different condition, which we will not employ here, is the Kato-type bound on the Ricci tensor \cite{Carron19,Rose19}.

The second direction has seen rapid growth.
A key observation is that the class of manifolds with a uniform Ricci lower bound and a uniform dimension upper bound is precompact in the pointed Gromov-Hausdorff topology \cite{Gromov07}. Limit structures are called Ricci limit spaces and were extensively studied starting from the celebrated work \cite{Cheeger-Colding96}. This analysis paved the way for studying constant curvature lower bounds from an intrinsic viewpoint. This is captured by the curvature-dimension condition (${\sf CD}$ for short), defined via optimal transport, introduced independently in \cite{Sturm06I,Sturm06II} and \cite{Lott-Villani09}. In this setting, comparison and rigidity results have been extensively studied, see e.g.\ \cite{Ketterer13,Ketterer15,DePhilippisGigli15,Rajala12,CavallettiMondino17-Inv,CavallettiMondino17,ErbarSturm21,Mondinosemola20,MR4701367,NobiliViolo22,NobiliViolo24,NobiliViolo25,Nobili24_overview}. 
We also refer to \cite{Villani2016,AmbICM,Caputo2025} for further details.

Lastly, these two lines of research naturally intersect. Indeed, the class of uniformly bounded $d$-manifolds with a sufficiently small uniform upper bound on the integral curvature deficit is precompact in the Gromov-Hausdorff topology, see \cite{PetersenWei97,PetersenWei01}. Hence nonsmooth limits arise in this setting as well. Analogous investigations were carried under Kato-type bounds \cite{CarronMondelloTewodrose24,CarronMondelloTewodrose22}, and related developments appear in metric measure settings \cite{Sturm15,Ket17,Ket21,ErbarRigoniSturmTamanini22,Sturm25}.

\medskip 

\noindent\textbf{Setting and main results}. In this manuscript, we focus on a metric measure space \((\X,\sfd,\mm)\), that is a  complete and separable metric space \((\X,\sfd)\) endowed with a boundedly finite, non-negative reference measure \(\mm\), which admits a \emph{variable} Ricci curvature lower bound. Concretely, we assume that \((\X,\sfd,\mm)\) satisfies the so-called \({\sf CD}(k,N)\) condition (see \cref{def:CDkN}) for some dimensional parameter \(N \in (1,\infty)\), and some locally bounded and lower semicontinuous function
\[
k \colon \X \to \R,
\]
called admissible. This condition has been introduced in \cite{Ket17}, using the language of optimal transport formulation, and generalizes the classical formulations \cite{Sturm06I,Sturm06II,Lott-Villani09} for constant curvature lower bounds. Moreover, as shown in~\cite{Ket17}, it is consistent with the smooth category: a manifold \((M^d,g)\) satisfies \({\sf CD}(k,N)\) if and only if \(d \le N\) and \({\rm Ric}_g \ge k g\) on \(M^d\) and a natural choice is \(k(x) = {\rm Ric}_g^-(x)\).

For \(p \ge 1\), we define the \emph{integral curvature deficit} from a constant \(K \in \mathbb{R}\) of an admissible function \(k\colon \X\to\R\) over a Borel set \(E \subset \X\) by
\begin{equation}\label{eq:defect}
    \rho^k_p(E,K) \coloneqq  \int_E \big| \min\{k-K,0\} \big|^p\,\d\mm .
\end{equation}
Our first main result is a quantitative version of the Bishop-Gromov inequality in the nonsmooth setting. Here \(v_{K,N}(r)\) denotes the volume of a geodesic ball of radius \(r\) in the \(N\)-dimensional model space of constant curvature \(K\) (see \eqref{eq:vKN}), star-shaped refers to sets that are radially generated from a point (see Definition~\ref{def:starshaped}), and \(\theta_N(\cdot)\) denotes the Bishop-Gromov density (see \eqref{eq:BDdensity}).
\begin{theorem}\label{thm:BG}
    For every $K\in \R,N \in (1,\infty), p>N/2$, there exists an explicit non-decreasing positive function $\left(0,\pi\sqrt{\tfrac{N-1}{K^+}}\right)\ni R\mapsto C_{K,N,p}(R)$ (see \eqref{eq:CpKNR}) such that the following holds. Let $(\X,\sfd,\mm)$ be an essentially non-branching ${\sf CD}(k,N)$ for some $k\colon \X\to\R$ admissible. Suppose that $x \in \X$ is such that $\mm(B_\varrho(x))=o(\varrho)$ as $\varrho\downarrow 0$ and let $T\subset \X$ be a star-shaped set at $x$. Then,  for $0<  r \le  R$, it holds
    \begin{equation}
    \label{eq:bishop_gromov_quantitative}
        \left( \frac{\mm(B_R(x)\cap T)}{v_{K,N}(R)} \right)^\frac{1}{2p-1} - \left( \frac{\mm(B_r(x)\cap T)}{v_{K,N}(r)} \right)^\frac{1}{2p-1} \le C_{K,N,p}(R) \rho^k_p(T,K)^\frac{1}{2p-1}. 
    \end{equation}
    Moreover, if $T = B_R(x)$, it holds
    \begin{equation}
    \label{eq:bishop_gromov_quantitative_theta}    
    \mm(B_R(x)) \le \left( \theta_N(x)+ C_{K,N,p}(R) \rho^k_p(B_R(x),K)^\frac{1}{2p-1}  \right)^{2p-1}v_{K,N}(R).
    \end{equation}
    Finally, the function $C_{K,N,p}(R) $ is globally bounded if $K<0$ and satisfies 
    \begin{equation}
        C_{0,N,p}(R)= \left(\frac{N-1}{(2p-1)(2p-N)}\right)^{\frac{p-1}{2p-1}} R^{\frac{2p-N}{2p-1}}.\label{eq:Cp_ZERO_NR}
    \end{equation} 
\end{theorem}
The above statement is a Bishop-Gromov type comparison result that quantifies, in terms of the integral curvature deficit $\rho_p^k(B_R(x),K)$, how close the function
\[
(0,R)\ni r \mapsto \frac{\mm(B_r(x))}{v_{K,N}(r)},
\]
is to be monotone (which is the case if $k\ge K$ and the integral curvature deficit vanishes, see \cite{Sturm06I,Sturm06II}). In particular, Theorem~\ref{thm:BG} extends the foundational manifold analysis of \cite{PetersenWei97,PetersenSprouse98} to the setting of ${\sf CD}$ spaces with variable Ricci curvature lower bounds.

The assumption $\mm(B_\varrho(x))=o(\varrho)$ is needed to deal with potentially too collapsed spaces. This is ultimately linked to the validity of one-dimensional mean curvature comparison estimate for nonsmooth densities, carried out in Proposition \ref{prop:mean vs deficit 1D} in the spirit of \cite{PetersenWei97,PetersenSprouse98,Aubry2007}. Moreover, our assumption clarifies previous literature for collapsed settings, see Remark~\ref{rem:Ket vs Noi} for a detailed discussion.

An analogous result has been obtained in various (smooth) settings in \cite{Wu19,LiWuZheng21,ChengFeng25} with integral Bakry-\'Emery Ricci bounds. We remark that, in all of these works, assumptions playing the same role as $\mm(B_\varrho(x))=o(\varrho)$ are required. Our condition is substantially weaker and automatically satisfied, for example, in the non-collapsed setting.

\medskip

We next present our second main result, namely a Myers' diameter theorem for nonsmooth spaces with sufficiently small integrable curvature deficit.
\begin{theorem}
\label{thm:diameter_estimate}
    For every $N\in [2,\infty),K>0$ and $p>N/2$ there exists a constant $C_{K,N,p}>0$ such that the following holds. Let $(\X,\sfd,\mm)$ be an essentially non-branching ${\sf CD}(k,N)$ for some $k\colon \X\to\R$ admissible. Assume that $\mm(B_\varrho(x))=o(\varrho)$ as $\varrho\downarrow 0$ at $\mm$-a.e.\ $x\in \X$.  If $\rho^k_p(\X,K)<+\infty$  then $\mm(\X)<\infty$. Moreover, if
    \begin{equation*}
      \rho^k_p(\X,K)\le \frac{\mm(\X)}{C_{K,N,p}},   
    \end{equation*}
    then $(\X,\sfd)$ is compact and it holds
    \begin{equation}
        {\rm diam}(\X) \le \pi\sqrt{\frac{N-1}{K}}\left(1 + C_{K,N,p} \left(\frac{\rho^k_p(\X,K)}{\mm(\X)}\right)^{\frac{1}{5}}\right).\label{eq:diameter estim}
    \end{equation}
\end{theorem}

The above statement can be regarded as a quantified version of the celebrated Myers' diameter theorem: if the averaged integral curvature deficit is sufficiently small, then the space is compact and the maximal diameter estimate~\eqref{eq:diameter estim} holds up to an error controlled by the deficit. Our result extends previous works \cite{Sprouse00,Aubry2007} on Riemannian manifolds to the nonsmooth setting; see also \cite{WuBY19,LiWuZheng21} for related generalizations in possibly weighted frameworks. Theorem \ref{thm:diameter_estimate} is proved as a byproduct of Theorem \ref{thm:diameter_estimate_body} where also the case $N\in(1,2)$ is treated. In this dimension range, the estimate \eqref{eq:diameter estim} trivializes as $N\downarrow 1$ as expected, see Remark \ref{rem:gamma ge 2} for more details. As before, the assumption $\mm(B_\varrho(x))=o(\varrho)$ is linked to the estimates developed in Proposition \ref{prop:mean vs deficit 1D}.

\medskip

We present our final main result. For $\Omega \subset \X$ open, the $p$-Dirichlet eigenvalue is defined as
\begin{equation} \label{eq:variational_eigenvalue}
    \lambda_p(\Omega) \coloneqq \inf_{u \in W^{1,p}_0(\Omega)}
    \frac{\int_{\Omega} |Du|_p^p\,\d \mm}{\int_{\Omega} |u|^p\,\d \mm},
\end{equation} 
where $W^{1,p}_0(\Omega) \coloneqq \overline{{\rm Lip}_{bs}(\Omega)}$ denotes the Sobolev space with homogeneous Dirichlet boundary conditions. Here, the closure is with respect the norm $\|u\|_{W^{1,p}(\Omega)}^p \coloneqq \|u\|_{L^p(\mm)}^p + \||Du|_p\|_{L^p(\mm)}^p$ and $|Du|_p$ is the minimal $p$-weak upper gradient (see \cite{AmbrosioGigliSavare11-3,Gigli12}). In the statement below, $\lambda_p(K,N,r)$ is the $p$-Dirichlet eigenvalue in the one-dimensional $(K,N)$-model, as defined in \eqref{eq:p eigenvalue}.
\begin{theorem}\label{thm:cheng}
    For every $K\in \R,N \in (1,\infty), p_0>N/2, p \in (1,\infty)$ and $r \in \left(0,\pi\sqrt{\tfrac{N-1}{K^+}}\right)$, there  exist $\eps \coloneqq\eps_{K,N,r,p,p_0}>0$ and $C_{K,N,r,p,p_0}>0$ such that the following holds. Let $(\X,\sfd,\mm)$ be an essentially non-branching ${\sf CD}(k,N)$ for some $k\colon \X\to\R$ admissible. Set $\bar p = \max\{p/2,p_0 \}$ and suppose, for $x\in\X$ with $\mm(B_\varrho(x))=o(\varrho)$ as $\varrho\downarrow 0$, that
    \begin{equation}
        \frac{\rho^k_{\bar p}(B_r(x),K)}{\mm(B_r(x))}  <\eps.\label{eq:p cheng small deficit}
    \end{equation}
    Then, it holds
    \[
    \lambda_p(B_r(x))  \le \lambda_p(K,N,r) + C_{K,N,r,p,p_0} \left(\frac{ \rho^k_{\bar p}(B_r(x),K)}{\mm(B_r(x))}\right)^{\frac{1}{2\bar p-1}}.
    \]
\end{theorem}
This theorem extends Cheng's eigenvalue comparison to the nonsmooth setting under a smallness condition on the averaged integral curvature deficit. The case $p=2$ was first established in \cite{PetersenSprouse98}, and later extended to all $p \in (1,\infty)$ in \cite{SetoWei2017}. In the nonsmooth setting, a Cheng's theorem for $2$-eigenvalues was recently obtained in \cite{DeLucaDePontiMondinoTomasiello25}. In contrast to Theorem \ref{thm:cheng}, their focus is on rigidity/stability statements for nonsmooth spaces admitting a constant Ricci lower bound.

\medskip 
\noindent\textbf{The localization method.} The localization method is a central tool in all our results. Originally developed in nonsmooth settings in \cite{CavallettiMondino17-Inv}, after earlier the work \cite{Klartag17} on manifolds, it has since been widely applied, most notably to show the L\'evy-Gromov isoperimetric inequality in the same work \cite{CavallettiMondino17-Inv} and to derive new results in geometric analysis, see e.g.\ \cite{CavallettiMondino17,CavallettiMaggiMondino19,FathiGentilSerres22,CavallettiMondinoSemola23}. More recently, it has been adapted to novel contexts, such as nonsmooth Lorentzian geometry \cite{CavallettiMondino24,BraunnMcCann24}.

The classical works \cite{PetersenWei97,PetersenSprouse98,Aubry2007} in the Riemannian setting reduce the problem to a one-dimensional analysis via polar coordinates. In low-regularity spaces, localization can play naturally an analogous role to nonsmooth polar coordinates. Indeed, to show Theorems \ref{thm:BG}, \ref{thm:diameter_estimate}, and \ref{thm:cheng}, it suffices to localize the curvature-dimension condition along transport rays associated with the distance from a point. This follows by a more general localization result, cf.\ Theorem \ref{thm:localization}, which holds for arbitrary $1$-Lipschitz functions.
Thus, our statement fully extends the analysis of \cite[Thm.\ 5.1]{CavallettiMondino17-Inv} and requires no integrability assumption on the variable Ricci bound $k\colon\X\to\R$.

It is relevant to point out a technical aspect of our applications that, differently from \cite{CavallettiMondino17-Inv}, we encounter here. The localization is capable of acting as a one-dimensional reduction method in the case of constant Ricci lower bounds; see for instance \cite{CavallettiMondino17} where many functional comparison inequalities are derived in this way. On the other hand, even if the spirit of the present paper is the same, this idea does not completely carry over to the variable Ricci curvature lower bound setting. Indeed, examining for instance Theorem \ref{thm:BG}, knowing that inequality \eqref{eq:bishop_gromov_quantitative} holds on weighted intervals does not directly imply that it holds on $(\X,\sfd,\mm)$ via disintegration of the reference measure $\mm$. This is because the reminder term containing the integral curvature deficit does not reintegrate to \eqref{eq:defect} due to a non-linear power $1/(2p-1)$ (indeed, it is concave, while convexity would be needed to reintegrate). Consequently, the one-dimensional analysis needs to be combined with reintegration procedures, and this poses technical obstacles due to the nonsmooth ambient space (cf.\ Proposition \ref{prop:mean vs deficit 1D} and Lemmas \ref{lem:average integral tend zero} and \ref{lem:2.2fatto_meglio}).

Finally, we believe that the localization method in Theorem \ref{thm:localization} is an ingredient that can be potentially applicable to study further comparison results requiring novel one-dimensional analysis where polar coordinates might not be the natural choice, e.g.\ Laplacian and isoperimetric comparisons and spectral inequalities. These will be the object of future investigations.

\section{Preliminaries}
We start setting up some notation. Given $a,b\in\R$, we set $a\wedge b \coloneqq \min\{a,b\}$ while $ a\vee b \coloneqq \max\{a,b\}$ and $a^+\coloneqq a\vee 0,a^-\coloneqq a\wedge 0$. A metric measure space is a triple  $(\X,\sfd,\mm)$ where $(\X,\sfd)$ is a complete and separable metric space and $\mm$ is a non-negative, non-zero and boundedly finite Borel measure. 

We denote by $C([0,1],\X)$ the space of continuous curves with values in $\X$, which is a complete and separable metric space when endowed with the supremum distance. We say that $\gamma \in C([0,1],\X)$ is a metric $\sfd$-geodesic provided $\sfd(\gamma_t,\gamma_s)=|t-s|\sfd(\gamma_0,\gamma_1)$ for all $t,s \in [0,1]$. We denote by ${\rm Geo}(\X)\subset C([0,1],\X)$ the set of metric $\sfd$-geodesics, which is a closed subset of $C([0,1],\X)$, hence complete. We denote by $\ell(\gamma) \in [0,\infty]$ the length of a curve $\gamma \in C([0,1],\X)$ and, if $\gamma \in {\rm Geo}(\X)$, we have $\ell(\gamma)=\sfd(\gamma_0,\gamma_1)$. Finally, given a function $f \colon \X \to \mathbb{R}$ and $\gamma \in {\rm Geo}(\X)$, we define ${f_{\gamma}^+} = f_\gamma \coloneqq  f \circ \hat\gamma$, where $\hat\gamma\colon [0,\ell(\gamma)]\to \X$ is the arc-length parametrization of $\gamma$. Similarly $f_{\gamma}^-\coloneqq f\circ \hat \gamma^-$ for $\hat\gamma^-_t\coloneqq \hat\gamma_{1-t}$.

By $C(\X)$, ${\rm Lip}(\X)$ and ${\rm Lip}_{bs}(\X)$, we denote respectively the collection of continuous, Lipschitz and boundedly supported Lipschitz functions on $\X$. For all $p \in (1,\infty)$, we denote by $L^p(\mm),L^p_{loc}(\mm)$ respectively the space of $p$-integrable functions and $p$-integrable functions on a neighborhood of every point (up to $\mm$-a.e.\ equality relation) on $\X$. By $\mathcal{P}(\X),\mathcal{P}_2(\X)$ we respectively denote the set of Borel probability measures and Borel probability measures with finite $2$-moment. The space $\mathcal{P}_2(\X)$ is a metric space when endowed with the Wasserstein $2$-distance $W_2$. We refer to \cite{AmbrosioBrueSemola24_Book} for a detailed treatment. In addition, given $N\ge 1$, we denote the Bishop-Gromov density at $x\in \X$ by
\begin{equation}\label{eq:BDdensity}
    \theta_N(x) \coloneqq  \lims_{r\downarrow 0}\frac{\mm(B_r(x))}{\omega_Nr^N} \in [0,\infty],\qquad\text{
where }\omega_N\coloneqq \pi^{N/2}\Gamma(N/2+1)^{-1}.
\end{equation}
Finally, throughout the paper, we write $\mm(B_\varrho(x))=o(\varrho)$, with the understanding that $\varrho\downarrow 0$.

\subsection{Metric measure spaces with variable Ricci bounds}
Let $\kappa \colon [a,b] \to \R$ be continuous. We define the generalized sine function $\frs_\kappa : [a,b]\to \R$ as the unique function $v$ that solves
\begin{equation}
\label{eq:sturm_liouville_ODE}
        v''+ \kappa v=0,
\end{equation}
with initial conditions $v(a)= 0$, $v'(a)= 1$. Recall that, for $K \in \R, N>1$, $a=0$, $b= \pi \sqrt{\frac{N-1}{K^+}} $ and $ \left[0, \pi \sqrt{\frac{N-1}{K^+}}\right] \ni t \mapsto \kappa(t)\equiv K/(N-1)$ constant, we have
\[
\frs_\kappa(t) = \sin_{K/(N-1)}(t) = \begin{cases}
    \, \sin\left(t\sqrt{\frac{K}{N-1}}\right) &\text{if }K>0 ,\\
    \, t &\text{if } K=0,\\
    \,  \sinh\left(t\sqrt{\frac{-K}{N-1}}\right)&\text{if } K<0.
\end{cases}
\]
\begin{definition}[Distortion coefficients]\label{distort}
Let $\kappa:[0,{L}]\rightarrow \mathbb{R}$ be continuous and let $\theta \in (0,L]$, $t\in [0,1]$. 
Then, the distortion coefficient is defined as
\begin{equation}
\sigma_{\kappa}^{{(t)}}(\theta)\coloneqq \begin{cases}
\frac{\frs_{{\kappa}}(t\theta)}{\frs_{{\kappa}}(\theta)}& \mbox{ if }\frs_{\kappa}|_{(0,\theta]}>c>0,\\
\infty & \mbox{ otherwise}.
\end{cases}
\end{equation} 
\end{definition}
Observe that, if $\sigma_{\kappa}^{{(t)}}(\theta)<\infty$ for some $t$ (hence, for all), then $[0,1]\ni t\mapsto\sigma_{\kappa}^{{(t)}}(\theta)$ is a solution of 
\begin{align}\label{klebeband}
u''(t)+\kappa(t\theta)\theta^2u(t)=0
\end{align}
satisfying $u(0)=0$ and $u(1)=1$. 
\begin{definition}[Admissible variable lower bound]
    Let $(\X,\sfd)$ be a metric space. We say that a function $k \colon \X \to \R$ is \emph{admissible} if it is locally bounded below and lower semi-continuous.  
\end{definition}
Given an admissible $k$, we define the increasing sequence of continuous functions $k_n \colon \X \to \R$ as 
\begin{equation}
\label{eq:k_n_continuous}
    k_n(x) \coloneqq \inf_{y \in \X} \left(k(y)+n\sfd(x,y) \right)\wedge n \le k(x),\qquad\forall\,x\in\X,\ n\in \N.
\end{equation}
Note that $k_n\uparrow k$ pointwise as $n\uparrow\infty$. If $\gamma\colon [0,1] \to \X$ is a geodesic of length $\theta\coloneqq \ell(\gamma)$, we can set
\begin{equation}\label{eq:sigma monotone def}
    \sigma_{k_\gamma}^{(t)}(\theta)\coloneqq \lim_{n\to \infty} \sigma_{(k_{n})_\gamma}^{(t)}(\theta) \in \mathbb{R} \cup \{+\infty\},\qquad\forall\,t\in [0,1],
\end{equation}
and, given $N> 1$, we set $\sigma_{k_{\gamma},N-1}^{(t)}(\theta)\coloneqq \sigma_{k_{\gamma}/(N-1)}^{(t)}(\theta)$. Note that the definition \eqref{eq:sigma monotone def} is well-posed, since the sequence is monotone non-decreasing. 
Moreover, $(t,\theta)\mapsto \sigma_{(k_{n})_\gamma}^{(t)}(\theta)$ is continuous, hence $(t,\theta)\mapsto \sigma_{k_{\gamma}}^{(t)}(\theta)$ is lower semicontinuous. In the case $k$ is already continuous, \eqref{eq:sigma monotone def} is consistent with Definition \ref{distort}. Finally, this definition is independent of the monotone sequence $\{k_n\}_{n\in \N}$ chosen in \eqref{eq:k_n_continuous} if $\{k'_n\}_{n\in\N}$ is another monotone sequence with the property that $\|k_n-k'_n\|_{\infty}\to 0$ as $n\to\infty$. We refer to \cite{Ket17} for further details. It is possible to show that this is also the case when $k\in L^\infty_{loc}(\mm)$.

\begin{definition}[Generalized distortion coefficients]
Let $(\X,\sfd,\mm)$ be a metric measure space, $N>1$ and $k:\X\to \R$ be admissible. Given a geodesic $\gamma:[0,1]\to \X$ of length $\theta\coloneqq \ell(\gamma)$, we define the \textit{generalized distortion coefficients with respect to $k$ and $N$ along $\gamma$} as
\begin{align*}
\tau_{k_{\gamma},N}^{(t)}(\theta)\coloneqq t^{\frac{1}{N}}\big[\sigma_{k_{\gamma},N-1}^{(t)}(\theta)\big]^{\frac{N-1}{N}},
\end{align*} 
for every $t\in [0,1]$, with  the conventions $r\cdot\infty=\infty$ for $r>0$ and $0\cdot\infty=0$.
\end{definition}
We assume the reader to be familiar with the basic concepts of optimal transport, such as Wasserstein distance and optimal geodesic plans (denoted as ${\rm OptGeo}(\mu_0,\mu_1)$ for given marginals $\mu_0,\mu_1\in\mathcal P_2(\X)$) referring, e.g., to \cite{AmbrosioBrueSemola24_Book}. Given $N>1$, define the $N$-\emph{R\'enyi-entropy functional}
\[ 
S_N: \mathcal{P}_2(\X) \to [-\infty,0],\qquad 
S_N(\mu) \coloneqq  -\int \rho^{1-\frac{1}{N}}\, \d \mm,
\]
where $\mu=\rho \mm + \mu^s$ and $\mu^s$ is singular with respect to $\mm$. The evaluation map at time $t\in [0,1]$ is defined as $\e_t:C([0,1],\X)\to \X$ such that $\e_t(\gamma)\coloneqq \gamma_t$. It is continuous, hence Borel. 

We recall the definition of curvature dimension condition with variable bound, introduced in \cite{Ket17}. 
\begin{definition}[${\sf CD}(k,N)$ space]
\label{def:CDkN}
Let $(\X,\sfd,\mm)$ be a metric measure space, let $k\colon \X\rightarrow \mathbb{R}$ be admissible and let $N> 1$. We say that $(\X,\sfd,\mm)$ satisfies the \textit{curvature dimension condition}
${\sf CD}(k,N)$ if for each pair $\mu_0=\rho_0\mm,\mu_1=\rho_1\mm\in \mathcal{P}_2(\X)$ with bounded support
there exists $\ppi \in {\rm OptGeo}(\mu_0,\mu_1)$ such that $\mu_t\coloneqq {(\e_t)}_\sharp\ppi$ is absolutely continuous with respect to $\mm$ and
\begin{equation}
\label{eq:cd_condition_variable}
S_{N'}(\mu_t)\leq{\textstyle -}\!\!\int \tau_{k^{-}_{\gamma},N'}^{(1-t)}(\sfd(\gamma_0,\gamma_1))\rho_0\left(\e_0(\gamma)\right)^{-\frac{1}{N'}}+\tau_{k^{+}_{\gamma},N'}^{(t)}(\sfd(\gamma_0,\gamma_1))\rho_1\left(\e_1(\gamma)\right)^{-\frac{1}{N'}}\d\ppi(\gamma),
\end{equation}
for all $t\in [0,1]$ and all $N'\geq N$.
\end{definition}
The definition of ${\sf CD}$ space for $N=1$ can be given as well, see \cite{Ket17}. We refrain to do it here, since all our results only hold for $N>1$. 
If $(\X,\sfd,\mm)$ is a ${\sf CD}(k,N)$ space and $k\colon \X\to\R$ admissible, then the Bishop-Gromov $N$-density in \eqref{eq:BDdensity} exists at every point $x \in\X$. Indeed, since $k$ is locally bounded from below, on a neighborhood of $x$ it holds $k \ge K_0$ for some $K_0 \in \R$. Thus $r\mapsto \mm(B_r(x))/v_{K_0,N}(r)$ is non-increasing on a right neighborhood of zero (cf.\ \cite[Thm.\ 5.9]{Ket17}) and $\lim_{r \to 0} \frac{v_{K_0,N}(r)}{\omega_N r^N}=1$.

We conclude this section showing that ${\sf CD}$ spaces with variable Ricci curvature lower bounds are qualitatively non-degenerate, cf. \cite{CavHue15,Kell17}. Given $A,B\subset \X$ and $t \in [0,1]$, the set of $t$-midpoints is
\begin{equation}
    M_t(A,B)\coloneqq \{ \e_t(\gamma) :\gamma \in {\rm Geo}(\X) \text{ such that }\gamma_0 \in A,\gamma_1 \in B\}.
\end{equation}
We also define, for $K\in\R$,
\begin{equation}
    \Theta_{K}(A,B)\coloneqq \begin{cases}
        \inf_{x \in A, y \in B} \sfd(x,y),& \text{if }K \ge 0,\\
        \sup_{x \in A, y \in B} \sfd(x,y),& \text{if }K < 0.
    \end{cases}
\end{equation}
The following proof is an adaptation of \cite[Cor.\ 3.6]{MPR22}.
\begin{proposition}
\label{prop:qualit_non_deg}
    Let $(\X,\sfd,\mm)$ be a ${\sf CD}(k,N)$ space for some $N>1$ and $k\colon \X\to \R$ admissible. For every ball $B_R(p) \subset \X$, set $K_0\coloneqq \min\{ k(x) \colon x \in \overline {B_{2R}(p)}\}$. Then, if $A \subset B_R(p)$ and $x\in B$, it holds 
    \begin{equation}
    \label{eq:tmidpoints_on_bounded_sets}
        \mm(M_t(A,x)) \ge \tau_{K_0,N}^{(1-t)}(\Theta_{K_0}(A,x))^N\,\mm(A),\qquad\forall\,t \in(0,1).
    \end{equation}
    In particular, $\mm$ is qualitatively non-degenerate in the sense of \cite{CavHue15}.
\end{proposition}
\begin{proof}
    Fix a ball $B := B_R(p)$. Consider a Borel set $A \subset B$ and $x \in B$. We can assume without loss of generality that $\mm(A) >0$  (as, if $\mm(A)=0$, there is nothing to prove). We would like to apply the Brunn-Minkowski inequality in ${\sf CD}(k,N)$ obtained in \cite[Thm.\ 5.1]{Ket17} with $A_0=A, A_1=\{x\}$, but $\mm(A_1)$ is required to be of positive measure. We thus perform a standard approximation step. Fix any $t \in (0,1)$. Let us assume $A_0\coloneqq A$ to be compact and $\epsilon_n\downarrow 0$, set $A_{1,n} \coloneqq  B_{\epsilon_n}(x)$ so that, for $n$ large enough, $A_{1,n}\subset B$. Then, by the argument of \cite[Thm.\ 5.1]{Ket17}, there is a $\ppi_n \in {\rm OptGeo}\big(\frac{\mm_{A_0}}{\mm(A_0)},\frac{\mm_{A_{1,n}}}{\mm(A_{1,n})}\big)$ satisfying the curvature dimension condition such that:
    \begin{equation}\label{eq:BM A_0 compact to x}
        \mm({\rm supp}\big( (\e_t)_\sharp\ppi_n) ^{\frac{1}{N}} \ge \inf_{\gamma \in {\rm Geo}(A_0,A_{1,n})} \tau^{(1-t)}_{k_{\gamma}^-,N}(\sfd(\gamma_0,\gamma_1))\,\mm(A_0)^{\frac{1}{N}} \ge   \tau^{(1-t)}_{K_0,N}(\Theta_{K_0}(A_0,A_{1,n}))\,\mm(A_0)^{\frac{1}{N}}, 
    \end{equation}
    for all $t \in (0,1)$, having used the monotonicity of $\tau$ coefficients \cite[Cor.\ 3.11]{Ket17} in the last inequality and the definition of $\Theta_{K_0}$. Set
    \[
    A_t \coloneqq  \bigcap_{n\in\N} \bigcup_{k\ge n} {\rm supp}\big((\e_t)_\sharp\ppi_k\big).
    \]
    We claim that 
    \[
    A_t \subset M_t(A,x).
    \]
    Let $y \in A_t$ and, by definition, for all $n \in\N$ there is a geodesic $\gamma^ {h_n} \in {\rm Geo}(A_0,A_{1,h_n})$ so that $\gamma_t^{h_n} =y$ with $h_n\uparrow \infty$. By Ascoli-Arzel\'a (all curves are constant-speed geodesics contained in $2B$ which has compact closure, since $(\X,\sfd)$ is proper, according to \cite[Thm.\ 5.3]{Ket17}), there is a geodesic $\gamma:[0,1]\to\X$ which is a uniform limit of $\gamma^{h_n}$. The geodesic $\gamma$ satisfies $\gamma_t = y, \gamma_0 \in A_0$ (since $A_0$ is assumed compact) and $\gamma_1=x$.  This proves the claim and, using that $\Theta_{K_0}(A_0,A_{1,n})\to \Theta_{K_0}(A,x)$, it holds
    \begin{equation}
        \mm(M_t(A,x)) ^{\frac{1}{N}} \ge   \tau^{(1-t)}_{K_0,N}(\Theta_{K_0}(A,x))\,\mm(A)^{\frac{1}{N}},  
    \label{eq:step A compact}
    \end{equation}
    provided $A$ is compact. However, \eqref{eq:step A compact} also holds for an arbitrary Borel set $A\subset B$ by approximation with compact sets $A_n\subset A$ with $\mm(A_n)\uparrow \mm(A)$ (by inner regularity) and using $M_t(A_n,x)\subset M_t(A,x)$ and $\tau^{(1-t)}_{K_0,N}(\Theta_{K_0}(A_n,x)) \ge \tau^{(1-t)}_{K_0,N}(\Theta_{K_0}(A,x))$ for each $n\in\N$.

    Finally, to conclude that $\mm$ is qualitatively non-degenerate, we observe that if $K_0>0$, then $\Theta_{K_0}(A,x)\geq 0$ and $\tau^{(1-t)}_{K_0,N}(\Theta_{K_0}(A,x))\geq (1-t)^N$, while if $K_0\leq 0$, then $\Theta_{K_0}(A,x)\leq 2R$ and we have $\tau^{(1-t)}_{K_0,N}(\Theta_{K_0}(A,x))\geq \tau^{(1-t)}_{K_0,N}(2R)^N$. Hence, we have that 
\begin{equation}
    \mm(M_t(A,x)) \ge f_{R,p}(t)\,\mm(A),\qquad\forall\,t\in(0,1),
\end{equation}
with $f_{R,p}(t) \coloneqq \tau^{(1-t)}_{K_0^+,N}(2R)^N$. The proof is therefore concluded recalling \cite[Assumption 1]{CavHue15}.
\end{proof}
\subsection{Non-branching and existence of transport maps}
We recall the definition of an essentially non-branching metric measure space.
\begin{definition}
    Let $(\X,\sfd,\mm)$ be a metric measure space. A set $G \subset { \rm Geo}(\X)$ is called \emph{non-branching} if for any $\gamma^1,\gamma^2 \in G$ the following holds: if $\gamma^1\restr{[0,t]}=\gamma^2\restr{[0,t]}$ for some $t\in (0,1)$, then $\gamma^1= \gamma^2$. We say that $(\X,\sfd,\mm)$ is called \emph{essentially non-branching} if for all $\mu_0,\mu_1 \in \mathcal P_2(\X)$, any $\nu \in {\rm OptGeo}(\mu_0,\mu_1)$ is concentrated on a non-branching Borel set $G \subset { \rm Geo}(\X)$.
\end{definition}
By a combination of Proposition \ref{prop:qualit_non_deg} with \cite[Thm.\ 5.8, Cor.\ 5.9]{Kell17}, we deduce the existence and uniqueness of optimal transport plans. 
\begin{proposition}
\label{prop:existence_uniqueness}
    Let $(\X,\sfd,\mm)$ be an essentially non-branching ${\sf CD}(k,N)$ space, for some $k:\X\to\R$ admissible and $N>1$. Then, for any $\mu_0,\mu_1 \in \mathcal{P}_2(\X)$ with $\mu_0 \ll \mm$ there is a unique optimal transport plan with respect to $W_2$ and it is induced by a map.
\end{proposition}
The above proposition implies the following equivalent characterization of the ${\sf CD}$ condition. The proof can be done as \cite[Prop.\ 3.1]{GigliRajalaSturm}, and \cite{Sturm06II} taking into account the necessary modifications.
\begin{proposition}
\label{prop:cd_curvewise}
Let $(\X,\sfd,\mm)$ be an essentially non-branching metric measure space and let $N>1$ and $k \colon \X \to \R$ be admissible. Then,  $(\X,\sfd,\mm)$ satisfies the ${\sf CD}(k,N)$ condition if and only for $\mu_0,\mu_1 \in \mathcal{P}_2(\X)$, there exists $\ppi \in {\rm OptGeo}(\mu_0,\mu_1)$ such that
\begin{equation}
    \label{eq:curvewise_characterization_CD_variable}
    \rho_t(\gamma_t)^{-\frac{1}{N}} \ge \tau_{k^{
-}_{\gamma},N'}^{(1-t)}(\sfd(\gamma_0,\gamma_1))\rho_0\left(\gamma_0\right)^{-\frac{1}{N'}}+\tau_{k^{+}_{\gamma},N'}^{(t)}(\sfd(\gamma_0,\gamma_1))\rho_1\left(\gamma_1\right)^{-\frac{1}{N'}}
\end{equation}
for all $t \in [0,1]$ for $\ppi$-a.e.\ $\gamma$, where $({\e_t})_\sharp\ppi=\rho_t\mm$.
\end{proposition}
\section{Localization of variable Ricci curvature bounds}
\label{sec:localization}
In this section, we show that a variable Ricci lower bound can be localized to the rays of the disintegration of $\mm$ relative to a 1-Lipschitz function. We adapt the strategy for constant curvature bound developed in \cite{Cav14,CavallettiMondino17-Inv} (for spaces of finite mass) and in \cite{CavallettiMondino20} (for spaces of infinite mass).

\subsection{Disintegration of the reference measure}\label{sec:disintegration} 
Let $u \colon \X \to \mathbb{R}$ be a $1$-Lipschitz function. We denote by $\pi_1 \colon \X \times \X \to \X$ as $\pi_1(x,y)\coloneqq x$.
We define $\Gamma_u = \{(x,y): u(x)-u(y)=\sfd(x,y)\}$, $\Gamma^{-1}_u\coloneqq \{(x,y): (y,x) \in \Gamma_u \}$ and $R_u\coloneqq \Gamma_u \cup \Gamma_u^{-1}$. Moreover, we denote by $\Gamma_u(x)\coloneqq \{y\in\X:(x,y)\in\Gamma_u\}$ and, similarly, $\Gamma^{-1}_u(x)\coloneqq \{y\in\X:(x,y)\in\Gamma^{-1}_u\}$. We call $R_u$ the transport relation. We recall the definition of transport set with endpoints $T_u \coloneqq  \pi_1 \big(R_u \setminus \{(x,y)\in \X\times \X:x=y\}\big)$.
We say that, given $x,y \in T_u$, $x \sim y$ if and only if $(x,y)\in R_u$. This relation is reflexive and symmetric but not transitive in general. Thus, we define the set of forward and backward branching points as 
\begin{equation}
    A^+ \coloneqq \{x \in T_u : \exists y,z \in \Gamma_u(x), (y,z)\notin R_u\},\quad A^- \coloneqq \{x \in T_u : \exists y,z \in \Gamma_u^{-1}(x), (y,z)\notin R_u\}.
\end{equation}
Consider, respectively, the non-branched transport set and the non-branched transport relation 
\begin{equation}
    T_u^{nb}\coloneqq  T_u \setminus (A^+ \cup A^-), \qquad R_u^{nb}\coloneqq  R_u \cap ( T_u^{nb} \times  T_u^{nb}).
\end{equation}
As shown in \cite{Cav14}, $R_u^{nb}$ is an equivalence relation on $T_u^{nb}$ and for every $x\in T_u^{nb}$, $R_u(x)\coloneqq  \{y\in \X : (x,y)\in R_u\}$ is isometric to a closed interval of $\R$. From the non-branched transport relation we obtain a partition of the non-branched transport set $T_u^{nb}$ into a disjoint family $\{\X_q\}_{q \in Q}$ of sets, where $Q$ is a set of indices. Moreover, $\bar \X_q$ is isometric to a closed interval of $\R$, for any $q\in Q$. We define the quotient map $\mathfrak{Q} \colon T_u^{nb} \to Q$ as 
\begin{equation}
    q=\mathfrak{Q}(x) \qquad\Longleftrightarrow\qquad x \in \X_q.
\end{equation}
We endow $Q$ with the pushforward $\sigma$-algebra, i.e.\ the finest $\sigma$-algebra on $Q$ for which $\mathfrak Q$ is measur\-able. We also define the quotient measure $\mathfrak{q
}\coloneqq \mathfrak{Q}_\sharp \mm$.
As proven in \cite{Cav14b, Cav14}, we have the following.

\begin{proposition}
\label{prop:disintegration}
    Let $(\X,\sfd,\mm)$ be a metric measure space and let $u \colon \X \to \R$ be a $1$-Lipschitz function. Define $T_u^{nb}$, $Q$, $\mathfrak{Q}$, $\mathfrak{q}$ as before.
    Then, there exists a disintegration of $\mm\restr{T_u^{nb}}$, namely a family $\{\mm_q\}_{q \in Q}$ such that, for every $A \in \mathscr{B}(\X)$, $Q \ni q \mapsto \mm_q(A) \in \R$ is $\mathfrak{q}$-measurable and
    \begin{equation*}
        \mm \restr{T_u^{nb}} =\int_Q \mm_q \,\d \mathfrak{q}.
    \end{equation*}
    Additionally, the disintegration is strongly consistent, i.e. $\mm_q(\mathfrak Q^{-1}(q))=1$, $\mathfrak{q}$-a.e.\ $q \in Q$. 
    
    The family $\{\mm_q\}_{q \in Q}$ is $\mathfrak q$-essentially unique, meaning that, if there is another family $\{\tilde \mm_q\}_{q\in Q}$ of probability measures satisfying the above properties, then $\tilde \mm_q = \mm_q$ for $\mathfrak q$-a.e.\ $q\in Q$.
\end{proposition}
In a given metric measure space, the non-branched transport set $T_u^{nb}$ can be smaller than $T_u$, i.e.\ $\mm(T_u \setminus T_u^{nb}) >0$. The next result gives sufficient conditions under which this behavior is excluded.
\begin{proposition}[{\cite[Prop.\ 4.5]{Cav17}}]
    Let $(\X,\sfd,\mm)$ be a metric measure space such that for any $\mu_0,\mu_1 \in \mathcal{P}_2(\X)$ with $\mu_0 \ll \mm$ any optimal transference plan for $W_2$ is induced by a map. Then $\mm(A_+)=\mm(A_-)=0$. In particular, $\mm(T_u \setminus T_u^{nb})=0$.
\end{proposition}
Since, for essentially non-branching ${\sf CD}(k,N)$ spaces, every optimal plan is induced by a map, cf. Proposition \ref{prop:existence_uniqueness}, we easily obtain the following corollary. 
\begin{corollary}
\label{coro:existence_optimal_maps}
Let $(\X,\sfd,\mm)$ be an essentially non-branching ${\sf CD}(k,N)$ space for some $k\colon \X\to\R$ admissible and $N>1$. Then, it holds $\mm(A^+)=\mm(A^-)=0$. In particular $\mm(T_u \setminus T_u^{nb})=0$.
\end{corollary}
\subsection{Regularity of the disintegration}
\label{sec:regularity_disintegration}
Next, we discuss the regularity of conditional measures $\mm_q$, namely their absolute continuity with respect to the one-dimensional Hausdorff measure restricted to $\X_q$. Firstly, consider the set $S\coloneqq \{ (q,t,x) \in Q \times [0,\infty) \times T_u^{nb} \colon (q,x) \in \Gamma_u,\, u(q)-u(x)=t \} \cup \{ (q,t,x) \in Q \times (-\infty,0] \times T_u^{nb} \colon (x,q) \in \Gamma_u,\, u(x)-u(q)=t \}$. Then, we define the ray map
\begin{equation}
    g \colon {\rm Dom}(g) \subset Q \times \mathbb{R} \to T_u^{nb},
\end{equation}
as the map whose graph in $Q\times\R\times T_u^{nb}$ is given by the set $S$, namely such that ${\rm graph}(g)=S$.
\begin{proposition}\label{prop:mq are AC}
    Let $(\X,\sfd,\mm)$ be an essentially non-branching ${\sf CD}(k,N)$ space for some $k\colon \X\to\R$ admissible and $N>1$. Given a $1$-Lispchitz function $u \colon \X\to\R$, consider the disintegration $\{\mm_q\}_{q\in Q}$ of $\mm\restr{T_u^{nb}}$ relative to $u$. Then, for $\frq$-a.e.\ $q\in Q$, $\mm_q\ll g(q,\cdot)_\sharp\Leb^1$.
\end{proposition}
We omit the proof of this result as it goes along the same lines of the case of constant lower curvature bounds. More precisely, one replicates the proof of \cite[Thm.\ 6.6]{Cav14}, taking into account Propositions \ref{prop:qualit_non_deg} and \ref{prop:existence_uniqueness}. 
\begin{remark}
    \label{rmk:ray_charles}\rm
    As observed in \cite[Prop.\ 4.12]{Cav17}, for every $q\in Q$, the map $t\mapsto g(q,t)$ is an isometry between ${\rm Dom}(g(q,\cdot))$ and $\X_q$. Moreover, $\overline{{\rm Dom}(g(q,\cdot))}=[0,r_q]$, where $r_q\coloneqq \ell(\X_q)\in [0,\infty]$. Since, by Proposition \ref{prop:mq are AC}, $\mm_q\ll g(q,\cdot)_\sharp\Leb^1$, for $\mathfrak q$ a.e.\ $q\in Q$, there exists a Borel function $h_q:[0,r_q]\to [0,\infty]$ such that $\mm_q=g(q,\cdot)_\sharp(h_q \Leb^1)$.\fr
\end{remark}
\subsection{Localization of variable curvature bounds}
We finally show that the rays of the disintegration $\{\mm_q\}_{q\in Q}$ of $\mm$ relative to a $1$-Lipschitz function inherit the variable Ricci lower bound. The proof follows the same lines of \cite[Thm.\ 4.2]{CavallettiMondino17-Inv} and we outline the main differences. In the statement below, recall the definition of $h_q$ and $r_q$ from Remark \ref{rmk:ray_charles}.
\begin{theorem}\label{thm:localization}
    Let $(\X,\sfd,\mm)$ be an essentially non-branching ${\sf CD}(k,N)$ space for some $k\colon \X\to\R$ admissible and $N>1 $. 
    Given a $1$-Lispchitz function $u \colon \X\to\R$, consider the disintegration $\{\mm_q\}_{q\in Q}$ of $\mm\restr{T_u^{nb}}$ relative to $u$. Then, for $\frq$-a.e.\ $q$, $([0,r_q],|\cdot|,h_q  \Leb^1)$ is a ${\sf CD}(k\circ g(q,\cdot),N)$ space.
\end{theorem}
\begin{proof}
    We subdivide the proof into different steps.
    We recall by \cref{prop:mq are AC} that $\mm_q \ll g(q,\cdot)_\sharp\Leb^1$ for $\mathfrak{q}$-a.e.\ $q$, so that $h_q$ is a well-defined density on  ${\rm Dom}(g(q,\cdot))$.
    
    \noindent\textsc{Step} 1. Let $N>1$ and by arguing as in the first part of the proof of \cite[Thm.\ 4.2]{CavallettiMondino17-Inv}, it is not restrictive to prove the statement for $\frq$-a.e.\ $q \in Q'$, where $Q' \subset Q$ is such that $\{ g(q,0)\colon q \in Q'\}\subset \{u =0\}$. Moreover, without loss of generality we can assume that there are $a_0,a_1\in\R$ so that $a_0<0<a_1$ and 
    \[
    (a_0,a_1)\subset {\rm Dom}(g(q,\cdot)),\qquad \forall\,q \in Q'.
    \]
    Consider any $a_0 < A_0<A_1<a_1$ and $L_0,L_1>0$ so that $A_0+L_0<A_1$ and $A_1+L_1<a_1$.
    Set
    \[
    A_s\coloneqq (1-s)A_0+sA_1,\qquad L_s\coloneqq (1-s)L_0 + sL_1,
    \]
    and observe that
    \[
    s\mapsto \mu_s \coloneqq \int_{Q'}\frac{1}{L_s}\HH^1\mres{\{ g(q,t)\colon t \in [A_s,A_s+L_s]\}}\,\d\frq,
    \]
    is the unique $W_2$-geodesics between its endpoints and, since $\mm_q \ll g(q,\cdot)_\sharp\Leb^1$ for $\frq$-a.e.\ $q \in Q$, it also satisfies $\mu_s\ll \mm$. Denoting by $\rho_s$ its Radon-Nikodym derivative, it holds
    \begin{equation}
        \rho_s(g(q,t)) = \frac{1}{L_s}h_q(t)^{-1},\qquad \forall\,t \in [A_s,A_s+L_s].
    \label{eq:rhos along ray}
    \end{equation}

    \noindent\textsc{Step} 2. Fix $A_0,A_1,L_0,L_1$ as in Step 1. Since, $(\X,\sfd,\mm)$ is an essentially non-branching ${\sf CD}(k,N)$ space, Proposition \ref{prop:existence_uniqueness} implies that $[0,1]\ni s\mapsto \mu_s$ is the unique $W_2$-geodesic between its endpoint, and we can apply \cref{prop:cd_curvewise} to $s\mapsto \mu_s$. Denoting by $\gamma_\cdot^{t_0} \coloneqq g(q,(1-\cdot)t_0+ \cdot \, t_1)$ (omitting the dependence $q$ for ease of notation), together with \eqref{eq:rhos along ray}, we obtain, for $\frq$-a.e.\ $q$,
    \begin{equation}
    \label{eq:loc_aux}
    L_s^{\frac 1N} h_q((1-s)t_0 + st_1)^{\frac 1N} \ge \tau^{(1-s)}_{{k^{-}_{\gamma^{t_0}}},N}(t_1-t_0)L_0^{\frac 1N}h_q(t_0)^\frac{1}{N} + \tau^{(s)}_{{k^{+}_{\gamma^{t_0}}},N}(t_1-t_0)L_1^{\frac 1N}h_q(t_1)^\frac{1}{N},   
    \end{equation}
    for all $s\in[0,1]$ and for $\Leb^1$-a.e.\ $t_0 \in [A_0,A_0+L_0]$, where $t_1$ is defined as the image of $t_0$ via monotone rearrangement map from $[A_0,A_0+L_0]$ to $[A_1,A_1+L_1]$. In particular, for any $\tau \in [0,1]$, if $t_0=A_0 + \tau L_0$, then $t_1 = A_1 + \tau L_1$. Thus, we can substitute in \eqref{eq:loc_aux}, obtaining, for $\frq$-a.e.\ $q$,
    \begin{align*}
       L_s^{\frac 1N} h_q(A_s + \tau L_s)^{\frac 1N} &\ge (1-s)^{\frac 1N} \sigma^{(1-s)}_{{k^{-}_{\gamma^{t_0}}},N}(A_1-A_0 + \tau(L_1-L_0))^{\frac{N-1}{N}}L_0^{\frac 1N}h_q(A_0 + \tau L_0)^\frac{1}{N} \\
       & \qquad + s^{\frac 1N}\sigma^{(s)}_{{k^{+}_{\gamma^{t_0}}},N}(A_1-A_0 + \tau(L_1-L_0))^{\frac{N-1}{N}}L_1^{\frac 1N}h_q(A_1 + \tau L_1)^\frac{1}{N} , 
    \end{align*}
    holds for $\Leb^1$-a.e.\ $\tau \in [0,1]$ and for all $s \in (0,1)$. We now choose, without relabeling it, a continuous representative of $s\mapsto h_q(s)$ whose existence is guaranteed by \eqref{eq:tmidpoints_on_bounded_sets} and the same argument of \cite[Prop.\ 7.5]{Cav14}. Then, also using the lower semi-continuity of $\sigma^{(s)}_{{k^{\pm}_{\gamma^{t_0}}},N}(\cdot)$ (recall that $\gamma^{t_0} _\cdot =g(q, (1-\cdot) (A_0+\tau L_0) + \cdot \, (A_1+\tau L_1))$ and $g(q,\cdot)$ is an isometry), we can send $\tau$ to zero along a suitable sequence (depending on $q$), to reach that, for $\frq$-a.e.\ $q$,
    \begin{equation}
       (L_s)^{\frac 1N} h_q(A_s)^{\frac 1N}\ge (1-s)^{\frac 1N}\sigma^{(1-s)}_{{k^{-}_{\gamma^{t_0}}},N}(A_1-A_0)^{\frac{N-1}{N}}L_0^{\frac 1N}h_q(A_0)^\frac{1}{N}+s^{\frac 1N}\sigma^{(s)}_{{k^{+}_{\gamma^{t_0}}},N}(A_1-A_0)^{\frac{N-1}{N}}L_1^{\frac 1N}h_q(A_1)^\frac{1}{N}. \label{eq:estim 1}
     \end{equation}
     \noindent\textsc{Step} 3. In the inequality \eqref{eq:estim 1}, the left-hand side depends continuously on $A_0,A_1,L_0,L_1$, while the right-hand side is lower semicontinuous as a function of $A_0,A_1,L_0,L_1$. Thus, there is a common exceptional set $N \subset Q'$ with $\frq(N)=0$, such that \eqref{eq:estim 1} holds true for every $q \in Q'\setminus N$, for every $s\in [0,1]$ and for all choices of $A_0,A_1,L_0,L_1$ as in Step 1. Therefore, for fixed $q\in Q'\setminus N$, we choose
    \begin{equation}
    \begin{aligned}
    \label{eq:choice_L0_L1}
    & L_0\coloneqq \frac{L}{1-s}\frac{ \sigma^{(1-s)}_{{k^{-}_{\gamma^{t_0}}},N}(A_1-A_0)\,h_q(A_0)^{\frac{1}{N-1}}}{\sigma^{(1-s)}_{{k^{-}_{\gamma^{t_0}}},N}(A_1-A_0)\,h_q(A_0)^{\frac{1}{N-1}} + \sigma^{(s)}_{{k^{+}_{\gamma^{t_0}}},N}(A_1-A_0)\,h_q(A_1)^{\frac{1}{N-1}}},\\
    &L_1\coloneqq \frac{L}{s}\frac{ \sigma^{(s)}_{{k^{+}_{\gamma^{t_0}}},N}(A_1-A_0) \,h_q(A_1)^{\frac{1}{N-1}}}{\sigma^{(1-s)}_{{k^{-}_{\gamma^{t_0}}},N}(A_1-A_0)\,h_q(A_0)^{\frac{1}{N-1}} + \sigma^{(s)}_{{k^{+}_{\gamma^{t_0}}},N}(A_1-A_0)\,h_q(A_1)^{\frac{1}{N-1}}},
    \end{aligned}
    \end{equation}
    for $L>0$ small enough. Notice that $L_s=(1-s)L_0 + s L_1 = L$ and, inserting \eqref{eq:choice_L0_L1} in \eqref{eq:estim 1}, we get
    \[
      h_q(A_s)^{\frac{1}{N-1}} \ge \sigma^{(1-s)}_{{k^{-}_{\gamma^{t_0}}},N}(A_1-A_0)h_q(A_0)^\frac{1}{N-1}+\sigma^{(s)}_{{k^{+}_{\gamma^{t_0}}},N}(A_1-A_0)h_q(A_0)^\frac{1}{N-1}
    \]
    for all $a_0 < A_0<A_1< a_1, s\in [0,1]$. By arbitrariness of $A_0,A_1$, this shows that $h_q$ is a ${\sf CD}(k\circ g(q,\cdot),N)$ density in the sense of Definition \ref{def:1D CD}, and proves the statement in light of Lemma \ref{lem:cd density implica cd space}.
\end{proof}
\begin{remark}
\label{rmk:densities}\rm
    From Lemma \ref{lem:cd density implica cd space}, the functions $(h_q)_{q\in Q}$ of the above theorem are ${\sf CD}(k\circ g(q,\cdot),N)$ densities, cf.\ Definition \ref{def:1D CD}. As such they admit locally Lipschitz representatives, which we will always consider without further notice.\fr
\end{remark}
\section{One-dimensional comparison estimates} \label{sec:mean curvature deficit 1d}
For $N \in (1,\infty)$ and $K \in \R$, we define the one-dimensional $(K,N)$-model space as
\[
\left(\left[0,\pi \sqrt{\frac{N-1}{K^+}}\, \right],|\cdot|,h_{K,N}\Leb^1\right), \qquad\text{where } h_{K,N} \coloneqq \sin_{K/(N-1)}^{N-1}.
\]
We denote the volume of the ball $B_r(0)$ in such a space by $v_{K,N}(r)$, namely
\begin{equation}
v_{K,N}(r) = \int_0^r h_{K,N}(t)\,\d t \qquad \text{for }r \in \left[0, \pi \sqrt{\frac{N-1}{K^+}} \, \right].
\label{eq:vKN}
\end{equation}
We define the mean curvature in the $(K,N)$-model space as 
\[
         H_{K,N} (r) \coloneqq \log(h_{K,N}(r))'=(N-1)\frac{\sin_{K/(N-1)}'(r)}{\sin_{K/(N-1)}(r)},\qquad \text{for }r \in \left(0, \pi \sqrt{\frac{N-1}{K^+}} \right),
\]
noting that this is well-defined since the denominator is strictly positive. For the next definition, we refer to Appendix \ref{apdx:1D density} for the notion of ${\sf CD}$ density. 

\begin{definition}[Mean curvature deficit]\label{def:deficit}
     Fix $K \in \R,N \in (1,\infty)$ and $D\in(0,\infty]$. Let $\kappa \colon [0,D] \to \mathbb{R}$ be admissible. Suppose that $h \colon [0,D]\to [0,\infty)$ is a ${\sf CD}(\kappa,N)$ density on $[0,D]$. We define the \emph{mean curvature deficit} of $h$ by
     \[
        \psi(r) \coloneqq  \left((\log h)'(r)-H_{K,N} (r)\right)\vee 0,\qquad\text{a.e.\ }r\in (0,D).
     \]
\end{definition}
\begin{remark}
\label{rmk:boundedness} \rm
    Note that $\psi(t)$ is a.e.\ well-defined since $h$ is locally Lipschitz (cf.\ Lemma \ref{lem:cd density implica cd space}). From this, we have $\psi \in L^p_{loc}([0,D))$, for every $p\geq 1$. Indeed, letting $I\subset [0,D)$ be a compact interval, there is $K_0\in \R$ such that $\kappa \ge K_0$ on $I$. Then, from \cite[Lemma A.9]{CavallettiMilman21}, we have 
    \[
        (\log h)'(r)\leq H_{K_0,N}(r),\qquad\text{a.e.\ }r\in I,
    \]
    so that $0\leq\psi\leq (H_{K_0,N}-H_{K,N}) \vee 0$ a.e.\ on $I$. On the other hand, the behavior of $H_{K,N}(r)$ as $r\downarrow 0$ is independent of $K$, hence $H_{K_0,N}-H_{K,N}$ is uniformly bounded close enough to zero. \fr
\end{remark}
We next extend \cite[Lemma 3.1]{Aubry2007} to ${\sf CD}$ densities on intervals (see also \cite{PetersenWei97,PetersenSprouse98}).
\begin{proposition}\label{prop:mean vs deficit 1D}
     Fix $K\in \R,N \in (1,\infty)$, $p>N/2$ and $D\in(0,\infty]$. Let $\kappa \colon [0,D] \to \mathbb{R}$ be admissible.  Suppose that $h \colon [0,D]\to [0,\infty)$ is a ${\sf CD}(\kappa,N)$ density on $[0,D]$ satisfying 
     \begin{equation}
     \label{eq:average_zero}
       \limi_{r\downarrow 0}\dashint_0^r h\,\d t=0.    
     \end{equation}
     Then, for a.e.\ $r \in \Big(0,  \frac{\pi}{2}\sqrt{\frac{N-1}{K^+}}\wedge D\Big)$ it holds
    \[
    \psi(r)^{2p-1}h(r) \le \alpha_{N,p}\int_0^r  | (\kappa(t) -K)\wedge 0 |^p h(t)\, \d t,
    \]
    where $\alpha_{N,p} \coloneqq (2p-1)^p \left(\frac{N-1}{2p-N}\right)^{p-1}$ and $\psi$ is as in Definition \ref{def:deficit}. Moreover, if $K>0$ and $D> \frac{\pi}{2}\sqrt{\frac{N-1}{K}}$, then for a.e.\ $r \in \left( \frac{\pi}{2}\sqrt{\frac{N-1}{K}}, \pi\sqrt{\frac{N-1}{K}}\wedge D\right)$, it holds
    \[
        \sin\left(\sqrt {\frac{K}{N-1}} r\right)^{4p-N-1}\psi^{2p-1}(r)h(r) \le \alpha_{N,p} \int_0^r  |(\kappa(t) -K)\wedge 0 |^p h(t)\, \d t.
    \]
\end{proposition}
\begin{proof}
    Without loss of generality, we can suppose $D<\infty$. Indeed, if not, for every $n\in\N$ the restriction of $h$ to the interval $[0,n]$ is a ${\sf CD}(\kappa,N)$ density. Then, the claim would follow by monotone approximation, up to discarding countably many negligible sets of $\Big(0, \pi\sqrt{\tfrac{N-1}{K^+}}\Big)$.  We set for brevity $\rho(t) \coloneqq  |(\kappa(t) -K)\wedge 0 |$ for $t \in (0,D)$ and subdivide the proof into different steps.

    \noindent\textsc{Step 1: Regularization of the densities}. For $\eps>0$ sufficiently small, we consider the density
    \begin{equation*}
        h_\eps \coloneqq  \exp \left( (\log h) \ast \eta_\eps \right)  \qquad\text{on } (\eps,D-\eps),
    \end{equation*}
    for $\{\eta_\eps\}_{\eps \in (0,D/2)}$ smooth mollifiers as in Proposition \ref{prop:convolution 1D density}. Thus, $h_\eps$ satisfies
     \begin{equation*}
         (\log h_\eps)''(t) + \frac{[(\log h_\eps)'(t)]^2}{N-1} \le - \kappa \ast \eta_\eps(t),\qquad\text{for a.e.\ }t \in (\eps,D-\eps).
     \end{equation*}
     Define $\kappa_\eps\coloneqq  \kappa \ast \eta_\eps$ and $\rho_\eps(t) \coloneqq | (\kappa_\eps (t) -K)\wedge 0 |$. We set $H_{K,N}^\eps(t)\coloneqq H_{K,N}(t-\eps)$, for every $t\in (\eps,D -\eps)$, and we write $\psi_\eps(t) \coloneqq  ((\log h_\eps)'(t)-H^\eps_{K,N}(t))\vee 0 $ for every $t\in (\eps,D-\eps)$. Note the key fact that $H_{K,N}(\cdot)$ is non-negative in the interval $\left(0,  \frac{\pi}{2}\sqrt{\frac{N-1}{K^+}}\wedge D\right)$. Since $h_\eps$ is $C^2(\eps,D-\eps)$ and $H_{K,N}$ is $C^1(0,D)$, we have that $(\eps,D-\eps)\ni t\mapsto \psi_\eps(t)$ is locally Lipschitz. Indeed, $\Phi(t)\coloneqq  0\vee t$ is 1-Lipschitz and $\psi_\eps(t)=\Phi\big((\log h_\eps)'(t)-H^\eps_{K,N}(t)\big)$, hence the composition is locally Lipschitz. A chain rule argument, together with  the computation (assume, without loss of generality, that $\psi_\eps(t)\neq 0$)
    \begin{align*}
        \left((\log h_\eps)'-H^\eps_{K,N}\right)' &\le -\kappa_\eps - \frac{((\log h_\eps)')^2}{N-1} + \frac{(H^\eps_{K,N})^2}{N-1} + K  \\
        &\le  -\kappa_\eps + K -2\frac{((\log h_\eps)'-H^\eps_{K,N})H^\eps_{K,N}}{N-1} - \frac{ ( (\log h_\eps)'- H^\eps_{K,N})^2}{N-1}\\
         &\le \rho_\eps  -2\frac{\psi_\eps H^\eps_{K,N}}{N-1} - \frac{(  (\log h_\eps)'- H^\eps_{K,N})^2}{N-1}.
    \end{align*}
    gives that
    \begin{equation}
    \psi_\eps'(t) + \frac{\psi_\eps(t)^2}{N-1} + 2 \frac{\psi_\eps H^\eps_{K,N}(t)}{N-1} \le \rho_\eps (t)\qquad \text{for a.e. }t\in(\eps,D-\eps). 
    \label{eq:key ODI}
    \end{equation}

    \noindent\textsc{Step 2: Estimate for $\eps>0$.} Consider any $R \in  \left(0, \frac{\pi}{2}\sqrt{\frac{N-1}{K^+}}\wedge D\right)$ and observe that there exists $\eps_0\coloneqq \eps_0(R)>0$ small enough so that $R< T -\eps_0(R)\le T -\eps$ for all $\eps \in (0,\eps_0)$, where $T\coloneqq  \frac{\pi}{2}\sqrt{\frac{N-1}{K^+}}\wedge D$. We claim that for any $\eps<\eps_0$, $r \in (\eps,R)$ and and for any function $\phi \in C^1(0,D)$ strictly positive and bounded in a neighborhood of $0$, it holds
    \begin{multline*}
        (\phi\psi_\eps^{2p-1}h_\eps)'(r) \le (2p-1)\rho_\eps(r) \phi(r) \psi_\eps^{2p-2}(r)h_\eps(r) - \left(\frac{2p-N}{N-1} \right) \phi (r)\psi_\eps^{2p}(r)h_\eps(r)\\
        + \left( \frac{4p-N-1}{N-1}H^\eps_{K,N}(r) - \frac{\phi'(r)}{\phi(r)}\right)^-\phi(r)\psi_\eps^{2p-1}(r)h_\eps(r).
    \end{multline*}
    Indeed, we observe that, by \eqref{eq:key ODI},
    \begin{align*}
        (\phi\psi_\eps^{2p-1}h_\eps)' &=  \phi'\psi_\eps^{2p-1}h_\eps +(2p-1)\phi\psi_\eps^{2p-2}h_\eps\psi_\eps' + \phi\psi_\eps^{2p-1}\left((\log h_\eps)' \pm H^\eps_{K,N}\right) h_\eps \\
        &\le\phi'\psi_\eps^{2p-1}h_\eps+(2p-1)\phi\psi_\eps^{2p-2}h_\eps\bigg(\rho_\eps-\frac{\psi_\eps^{2}}{N-1} -\frac{2\psi_\eps H^\eps_{K,N}}{N-1} \bigg)+\phi\psi_\eps^{2p}h_\eps + \phi\psi_\eps^{2p-1}H_{K,N}^\eps h_\eps \\
        & = (2p-1)\phi\rho_\eps\psi_\eps^{2p-2}h_\eps + \left(\frac{N-2p}{N-1}\right)\phi\psi_\eps^{2p}h_\eps  +\left( \frac{\phi'}{\phi}+\left(1-2\frac{2p-1}{N-1}\right)H_{K,N}^\eps \right)\phi\psi_\eps^{2p-1}h_\eps
    \end{align*}
    from which the claimed inequality holds. Integrating from $\eps$ to $r$, and by H\"older inequality, we get
    \begin{multline}
        \phi(r)\psi_\eps^{2p-1}(r)h_\eps(r) - \phi(\eps)\psi_\eps^{2p-1}(\eps)h(\eps) \le (2p-1)\left(\int_\eps^r \phi \rho_\eps^p h_\eps\,\d t\right)^{\frac 1p}\mathcal{I}_\eps^{1-\frac 1p} -\left(\frac{2p-N}{N-1}\right)\mathcal{I}_\eps \\
         + \left(\int_\eps^r \left(\left(\frac{4p-N-1}{N-1}\right) H^\eps_{K,N} - \frac{\phi'}{\phi}\right)^-\phi h_\eps\,\d t \right)^{\frac 1{2p}}\mathcal{I}_\eps^{1-\frac{1}{2p}}
    \label{eq:stima eps}
    \end{multline}
    where we denoted $\mathcal{I}_\eps \coloneqq \int_\eps^r\phi \psi_\eps^{2p}h_\eps\,\d t$ .

    \noindent{\sc Step 3: Limit as $\eps\downarrow 0$}. We pass to the limit in \eqref{eq:stima eps} as $\eps\downarrow 0$, treating each term separately. Recall that we fixed $R<T-\eps_0(R)$ while $\eps<\eps_0(R)$ as well as $r \in (\eps,R)$ are arbitrary. By construction and using item i) of \cref{prop:convolution 1D density}, as $\eps\downarrow 0$, it holds
    \[
        \phi(t)\psi_\eps^{2p-1}(t)h_\eps(t) \to\phi(t)\psi^{2p-1}(t)h(t),\qquad \text{for a.e.\ }t\in(\eps_0(R),R). 
    \] 
    We first claim that
    \[
        \phi(\eps)\psi_\eps^{2p-1}(\eps)h_\eps(\eps) \to 0, \qquad \text{as }\eps\downarrow 0.
    \]
    This follows directly by item iii) of Proposition \ref{prop:convolution 1D density} (which holds under the assumption \eqref{eq:average_zero}), provided we can show that $\phi(\eps)\psi_\eps(\eps)$ is uniformly bounded as $\eps\downarrow 0$. Let $K_0\in\R$ be such that $\kappa,\kappa_\eps\geq K_0$ on $[0,D]$ for every $\eps\in (0,\eps_0)$. Then, from \cref{rmk:boundedness}, after taking convolutions, we get
    \begin{equation*}
        (\log h_\eps)'(t)\leq H_{K_0,N}\ast\eta_\eps(t),\qquad\forall\,t\in [\eps,D-\eps]. 
    \end{equation*}
    Now, exploiting the fact that $H_{K_0,N}$ is a non-decreasing function in $[\eps,D-\eps]$, we deduce that
    \begin{equation*}
        (\log h_\eps)'(t)\leq H_{K_0,N}\ast\eta_\eps(t)= \int_{t-\eps}^{t+\eps}H_{K_0,N}(s)\eta_\eps(t-s)\d s\leq H_{K_0,N}(t-\eps)=H^\eps_{K_0,N}(t), 
    \end{equation*}
    for $t\in [\eps,D-\eps]$. Finally, $\psi_\eps(t)\leq (H^\eps_{K_0,N}(t)-H^\eps_{K,N}(t))\vee 0$, and, it holds that 
    \begin{equation}
    \label{eq:taylor_exp}
        H^\eps_{K,N}(t)=\frac{N-1}{t-\eps}(1+o(t-\eps)),\qquad\text{as }t\to\eps, \text{ independently of $K$.}
    \end{equation} 
    Thus $\psi_\eps(\eps)$ is uniformly bounded. Since $\phi$ is bounded close to zero, the claim follows. We show that
    \begin{equation}
       \mathcal{I}_\eps =  \int_\eps^r \phi \psi_\eps^{2p}h_\eps\,\d t  \to \int_0^r \phi  \psi^{2p}h\,\d t \eqqcolon \mathcal{I}, \qquad\text{as }\eps\downarrow 0. \label{eq:dominated claim}
    \end{equation} 
    Firstly, by item ii) of Proposition \ref{prop:convolution 1D density}, we note that, for a.e.\ $t \in (0,r)$, it holds
    \begin{equation*}
        \nchi_{(\eps,r)}\Phi\big((\log h_\eps)'(t) -H^\eps_{K,N}(t)\big)^{2p} h_\eps(t) \to \nchi_{(0,r)} \Phi((\log h)'(t) -H_{K,N}(t))^{2p}  h(t).
    \end{equation*}
    Secondly, $\psi_\eps(t)$ is uniformly bounded, due to the estimate $\psi_\eps(t)\leq (H^\eps_{K_0,N}(t)-H^\eps_{K,N}(t))\vee 0 $ deduced above and recalling the uniform bound in \cref{rmk:boundedness}. Moreover, by the definition of $h_\eps$ and by applying Jensen's inequality we have that $ \sup_{t \in [\eps,D-\eps]} h_\eps(t) \le  \sup_{t \in (0,D)} h(t) < \infty$, where the last inequality follows by \cite[Lemma A.8]{CavallettiMilman21}, recalling that $\kappa\ge K_0$ for some $K_0 \in \R$. Therefore, since $\phi$ is uniformly bounded on $(0,r)$, we apply the dominated convergence theorem to deduce \eqref{eq:dominated claim}.

    Thirdly, we claim that 
    \[
         \int_\eps^r \phi  \rho_\eps
        ^ph_\eps\,\d t \to \int_0^r\phi \rho
        ^ph\,\d t , \qquad\text{as }\eps\downarrow 0.
    \]
    Indeed, up to extracting a (not relabeled) subsequence, $\rho_\eps\to \rho$ a.e.\ on $(0,D)$. Then, since $h_\eps \to h$ locally uniformly on $(0,D)$, we deduce that $\nchi_{(\eps,r)}\rho_\eps^{p} h_\eps \to \nchi_{(0,r)}\rho^{p} h$ a.e.\ in $(0,r)$. Reasoning as before, we apply the dominated convergence theorem here as well to get the claim. 
    
    Finally, sending $\eps\downarrow 0$ in \eqref{eq:stima eps}, we deduce, for a.e.\ $r \in (\eps_0(R),R)$, 
    \begin{multline}
        \phi(r)\psi^{2p-1}(r)h(r) \le (2p-1)\left(\int_0^r \phi \rho^p h\,\d t\right)^{\frac 1p}\mathcal{I}^{1-\frac 1p} -\left(\frac{2p-N}{N-1}\right)\mathcal{I} \\
        + \left(\int_0^r \left(\frac{4p-N-1}{N-1} H_{K,N} - \frac{\phi'}{\phi}\right)^{-}-\phi h\,\d t \right)^{\frac 1{2p}}\mathcal{I}^{1-\frac{1}{2p}}.
    \label{eq:stima zero}
    \end{multline}
    Since $R$ is arbitrary and $\eps_0(R) \to 0$ as $R\to T$, the above holds for a.e.\ $r \in (0,T)$ as well. At last, since $ \phi(r)\psi^{2p-1}(r)h(r) \ge 0$, rearranging terms and dividing by  $\mathcal{I}^{1-\frac 1p}$ we get
    \begin{multline}
        \mathcal{I}^{\frac 1{2p}} \le \sqrt{\frac{(N-1)(2p-1)}{2p-N}}\left(\int_0^r\phi\psi^{2p}h\,\d t \right)^{\frac 1{2p}} \\
         + \left(\frac{N-1}{2p-N}\right) \left(\int_0^r \left(\left(\frac{4p-N-1}{N-1}\right) H_{K,N} - \frac{\phi'}{\phi}\right)^-\phi h\,\d t \right)^{\frac 1{2p}}
    \label{eq:stima X}
    \end{multline}

    \noindent\textsc{Step 4: Conclusion}. Taking $\phi \equiv 1$, we combine the estimate \eqref{eq:stima zero} with \eqref{eq:stima X} to deduce the first conclusion setting $\alpha_{N,p} \coloneqq  (2p-1)^p\left(\frac{N-1}{2p-N}\right)^{p-1}$. Instead, for the last conclusion, if $K>0$ we choose $\phi(r) = \sin\left(\sqrt{\frac{K}{N-1}} r\right)^{4p-N-1}$. In this case, for  a.e.\ $r \in \left( \frac{\pi}{2}\sqrt{\frac{N-1}{K}},\pi\sqrt{\frac{N-1}{K^+}}\right)$ we observe that the last integral in  \eqref{eq:stima X} vanish. As before, the conclusion follows by combining \eqref{eq:stima X} with \eqref{eq:stima zero}. 
\end{proof}
\section{Bishop-Gromov comparison}
\subsection{Technical lemmas}
In this section we shall prove our main result Theorem \ref{thm:BG} relying on Theorem \ref{thm:localization}. We start with some technical results around the disintegration relative to $\sfd_x$.  
\begin{definition}[Star-shaped set]\label{def:starshaped}
    Let $(\X,\sfd)$ be a metric space. We say that a Borel set $T\subset \X$ is star-shaped at $x\in T$, provided that, for every $y \in T$, there exists a geodesic $[0,1]\ni t \mapsto \gamma_t \in \X$ so that $\gamma_0=x,\gamma_1=y$ and $\gamma_t \in T$ for every $t \in [0,1]$. 
\end{definition}
Our first result is a polar coordinate formula via the localization associated to $u\sfd_x$. In this case, $T_u=\X$ and $\mm(\X\setminus T_u^{nb})=0$. In addition, for every $q\in Q$, the ray map $t\mapsto g(q,t)$ is an isometry between its domain (which always contain $0$) and $\X_q$. We refer to Section \ref{sec:localization} for details.   
\begin{lemma}\label{lem:polar_T}
Let $(\X,\sfd,\mm)$ be an essentially non-branching ${\sf CD}(k,N)$ space for some $N>1$ and $k\colon \X\to\R$ admissible. Let $T\subset\X$ be an open star-shaped set at some $x \in \X$, and consider the disintegration relative to $\sfd_x$. Then, for all $s < r$ and $\varphi \colon \X \to \R$ Borel and integrable, it holds
\begin{equation}
   \int_{(B_r(x)\setminus B_s(x))\cap T} \varphi\,\d \mm  =\int_s^r \int_{Q_T(t)} \varphi(g(q,t))h_q(t)\,\d {\mathfrak q}\,\d t,\label{eq:polar disintegration_T}
\end{equation}
where $h_q$, $r_q$ are defined in Remark \ref{rmk:ray_charles} and $Q_T(r)\coloneqq \{ q \in Q: r \in [0, r_q\wedge E_T^q )\}$. Here, $E_T^q$ is the first exit time of $t\mapsto g(q,t)$ from $T$, namely
\begin{equation*}
    E_T^q\coloneqq \sup\{t>0:g(q,t)\in T\}.
\end{equation*}
Moreover, $E_T^q>0$, whenever $r_q>0$.
\end{lemma}
\begin{proof}
    Firstly, we claim that, for $\mathcal{L}^1$-a.e.\ $r$, $Q_T(r)$ is $\mathfrak{q}$-measurable. Indeed, since $T\subset\X$ is a Borel set and the map $q\mapsto r_q$ is $\mathfrak{q}$-measurable, the map $q \mapsto r_q\wedge E_T^q $ is $\mathfrak{q}$-measurable as well. Thus
\begin{equation*}
    S\coloneqq \{ (q,s) \in Q \times \R:\, 0 < s <  r_q\wedge E_T^q \}
\end{equation*}
is $\mathfrak{q}\times \mathcal{L}^1$-measurable and, by Fubini's theorem, for $\mathcal{L}^1$-a.e.\ $r$, the set $\{ 
q \in Q: (q,r) \in S \}=Q_T(r)$ is $\mathfrak{q}$-measurable, proving the claim. In addition, this shows that the right-hand side of \eqref{eq:polar disintegration_T} is well-defined. Secondly, we show that $E_T^q>0$ whenever $r_q>0$. Indeed, let $q\in Q$ such that $r_q>0$. Then, since $T$ is open and $g(q,0)=x$, for $t\in (0,r_q)$ sufficiently small, $g(q,t)\in T$ and $g(q,\cdot)\restr{[0,t]}\subset T$, being $T$ star-shaped. Therefore, $E_T^q\geq t>0$.
The proof of the identity \eqref{eq:polar disintegration_T} now follows by adapting \cite[Prop. 5.3]{CaputoRossi2024}. We report it here for completeness. Fix $s<r$ and observe that it holds 
\[
    t<r_q,\quad g(q,t)\in (B_r(x)\setminus B_s(x))\cap T \qquad\Longleftrightarrow\qquad t\in[s,r), \quad q\in Q_T(t).
\]
Consequently, we have the following identity of characteristic functions: 
\[
    \nchi_{(B_r(x)\setminus B_s(x))\cap T}(g(q,t))=\nchi_{[s,r)}(t)\nchi_{Q_T(t)}(q),\qquad \mm_q\times \mathfrak q\text{-a.e..} 
\]
Therefore, we can write
\begin{align*}
    \int_{(B_r(x)\setminus B_s(x))\cap T}\varphi\,\d \mm &= \int_Q \int_\X  \nchi_{(B_r(x)\setminus B_s(x))\cap T}\cdot \varphi\,\d \mm_q\d \mathfrak q \\
    &= \int_Q \int_0^\infty  \nchi_{(B_r(x)\setminus B_s(x))\cap T}(g(q,t)) \varphi(g(q,t))h_q(t)\,\d t\d \mathfrak q \\
    &= \int_Q\int_0^\infty \nchi_{[s,r)}(t)\nchi_{Q_T(t)}(q) \varphi(g(q,t))h_q(t)\,\d t\d \mathfrak q \\
    &= \int_0^\infty \int_{Q} \nchi_{[s,r)}(t)\nchi_{Q_T(t)}(q)\varphi(g(q,t))h_q(t)\,\d {\mathfrak q}\,\d t \\
    &= \int_s^r \int_{Q_T(t)}\varphi(g(q,t))h_q(t)\,\d \mathfrak q \d t
\end{align*}
having used the disintegration theorem (cf.\ Proposition \ref{prop:disintegration}).
\end{proof}
\begin{lemma}\label{lem:average integral tend zero}
Let $(\X,\sfd,\mm)$ be an essentially non-branching ${\sf CD}(k,N)$ space for some $N >1$ and $k\colon \X\to\R$ admissible.  Suppose that $x \in \X$ is such that $\mm(B_\varrho(x))=o(\varrho)$. Consider the disintegration relative to the $1$-Lipschitz function $\sfd_x$. Then, we have 
\begin{equation}
\label{eq:average_integral}
  \limi_{r \to 0} \dashint_0^r h_q(t)\,\d t=0,\qquad\frq\text{-a.e.}.
\end{equation}
\end{lemma}
\begin{proof}
For every $q \in Q$, consider $h_q \colon [0,r_q] \to [0,\infty)$ as defined in Remark \ref{rmk:ray_charles} and extend to a not relabeled function $h_q \colon [0,\infty) \to [0,\infty)$ that is $0$ for every $t > r_q$. Now, since $\mm(B_\varrho(x))=o(\varrho)$, for $r$ sufficiently small, we have
\begin{equation*}
    \int_Q\int_0^r h_q(t)\,\d t \,\d \mathfrak{q}=\mm(B_r(x)) =o(r).
\end{equation*}
To conclude, we apply Fatou's lemma to the previous computation and get
\begin{equation*}
    \int_Q \limi_{r \to 0} \dashint_0^r h_q(t)\,\d t \,\d \mathfrak{q} \le \limi_{r \to 0} \int_Q\dashint_0^r h_q\,\d t \,\d \mathfrak{q} \le 0.
\end{equation*}
\end{proof} 
\begin{remark}\label{rem:necessity theta 1+c}\rm
    The conclusion  \eqref{eq:average_integral} is not always verified. Indeed, for $N>1$, an example is given by the metric measure space $([0,\infty),|\cdot|,\mm)$, where $\mm\coloneqq (r+1)^{N-1}\mathcal{L}^1$, which is ${\sf CD}(0,N)$. In particular, it is not true that $\mm([0,r])=o(r)$. \fr
\end{remark}
\begin{definition}[Spherical and volume integrals]
\label{def:spherical_integral}
     Let $(\X,\sfd,\mm)$ be an essentially non-branching ${\sf CD}(k,N)$ space for some $N> 1$ and $k\colon \X\to\R$ admissible.  Let $T\subset\X$ be an open star-shaped set at some $x\in T$ and consider the disintegration relative to $\sfd_x$. For every $r>0$, we set
     \[
        S_T(r) \coloneqq \int_{Q_T(r)}h_q(r)\,\d \mathfrak q,
     \]
    for $h_q(r)$ and $Q_T(r)$ as in Lemma \ref{lem:polar_T}. For every $r>0$, we denote by
    \[
        V_T(r) \coloneqq \int_0^rS_T(t)\,\d t=\mm(B_r(x)\cap T).
    \]
\end{definition}
The quantity $S_T$ is related to other notions of surface area in metric measure spaces, such as perimeter and Minkowski content (see e.g.\ \cite{CaputoCavallucci2025}).
The next lemma shows a monotonicity property of a key quantity related to $S_T$. It can be interpreted as a nonsmooth analogue of \cite[Lemma 2.2]{Aubry2007}. 
\begin{lemma}
\label{lem:2.2fatto_meglio}
    Let $(\X,\sfd,\mm)$ be an essentially non-branching ${\sf CD}(k,N)$ space for some $N> 1$ and $k\colon \X\to\R$ admissible.  Let $T\subset\X$ be an open star-shaped set at some $x\in T$ and consider the disintegration relative to $\sfd_x$. Then, the following properties hold:
    \begin{itemize}
        \item[${\rm i)}$] $S_T$ is right continuous and left lower semicontinuous on $(0,\infty)$;
        \item[${\rm ii)}$] $V_T$ is continuous and right-differentiable with derivative $S_T$ on $(0,\infty)$;
        \item[${\rm iii)}$] for every $\alpha \in (0,1],K\in\R$, the function
        \[
            f_\alpha(r) \coloneqq \left(\frac{S_T(r)}{h_{K,N}(r)}\right)^\alpha - \alpha\int_0^r\int_{Q_T(s)}\left(\frac{S_T(s)}{h_{K,N}(s)}\right)^{\alpha-1}\psi_q(s)\frac{h_q(s)}{h_{K,N}(s)}\,\d\mathfrak q\d s
        \]
        is non-increasing on  $\Big(0,\pi\sqrt{\tfrac{N-1}{K^+}}\Big)$, where, for every $q\in Q$, $\psi_q$ is the mean curvature deficit of $h_q$, cf.\ Definition \ref{def:deficit}, i.e.
        \begin{equation}
            \label{eq:mean_curvature_def_q}
            \psi_q(t)\coloneqq \left((\log h_q)'(t) -H_{K,N}(t)\right)\vee 0.
        \end{equation}
    \end{itemize}
\end{lemma}
Before giving the proof of the above, we recall a standard characterization of monotonicity. 
\begin{lemma}
\label{lem:monotonicity}
    Let $-\infty<a< b<+\infty$ and let $f:[a,b]\to\R$ be a function. Then, $f$ is non-increasing in $[a,b]$ if and only if $f$ is left lower semicontinuous in $(a,b]$ and, for every $r\in [a,b)$, 
    \begin{equation}
    \label{eq:right_der}
        \lims_{t\downarrow0}\frac{f(r+t)-f(r)}{t}\leq 0.
    \end{equation}
\end{lemma}
\begin{proof}[Proof of Lemma \ref{lem:2.2fatto_meglio}]
    Observe that ii) follows from i), thus we only prove i) and iii). Moreover, note that iii) is well-posed since $(q,t)\mapsto \psi_q(t)$ is $\frq \times \Leb^1$-measurable. We first prove i).  Fix $r>0$ and let $r_n \uparrow r$. Then, $Q_T(r)\subset Q_T(r_n)$ by definition and by Fatou's lemma we get
    \[
         S_T(r) \le \limi_{n\to\infty}  \int_{Q_T(r)}h_q(r_n)\,\d\mathfrak q \le \limi_{n\to\infty}  \int_{Q_T(r_n)}h_q(r_n)\,\d\mathfrak q = \limi_{n\to\infty}S_T(r_n).
    \]
    Thus, to conclude i), it remains to show that $S_T(\cdot)$ is right continuous. Hence, fix $r>0$ and let $r_n \downarrow r$. By definition, for every $n\in\N$, $Q_T(r_n)\subset Q_T(r_{n+1})$ and  $Q_T(r)=\bigcup_{n\in\N} Q_T(r_n)$. Consequently, up to extracting a subsequence, $\nchi_{Q_T(r_n)} (q)\to \nchi_{Q_T(r)}(q)$ $\mathfrak q$-a.e.. Claim i) then follows by applying the dominated convergence theorem.
    
    We now prove iii). Fix $\epsilon,R>0$ with $\epsilon<R<\pi\sqrt{\frac{N-1}{K^+}}$. Let $r\in [\epsilon,R]$ and $\delta>0$, and define
    \[
        S_T^\delta(r) \coloneqq \int_{Q_T^\delta(r)}h_q(r)\,\d \mathfrak q,
    \]
    where $Q_T^\delta(r)\coloneqq  Q_T\left(\frac{\delta r}{\delta-1}\right)$ and 
    \begin{equation}
    \label{eq:aux_fun_fun}
        f_\alpha^\delta(r)\coloneqq \left(\frac{S_T^\delta(r)}{h_{K,N}(r)}\right)^\alpha-\alpha\int_0^{r}\int_{Q_T^\delta(s)}\left(\frac{S_T^\delta(s)}{h_{K,N}(s)}\right)^{\alpha-1}\psi_q(s)\frac{h_q(s)}{h_{K,N}(s)}\,\d\mathfrak q\d s.
    \end{equation}
    Observe that, if $r\in[\epsilon,R]$ and $q\in Q_T^\delta(r)$, then $r_q>\epsilon$ and $r\in [\epsilon,r_q\wedge R]$. Fix $r\in [\epsilon,R]$. By construction $Q_T^\delta(r+t)\subset Q_T^\delta(r)$, hence, for sufficiently small $t>0$, it holds
    \[
        S_T^\delta(r+t)= \int_{Q_T^\delta(r+t)}h_q(r+t)\,\d \mathfrak q \le \int_{Q_T^\delta(r)}h_q(r+t)\,\d \mathfrak q.
    \]
    Thus, we may estimate 
    \begin{equation}
    \label{eq:S_T-uff}
        \frac{S_T^\delta(r+t)-S_T^\delta(r)}{t} \le \int_{Q_T^\delta(r)}\frac{h_q(r+t)-h_q(r)}{t}\,\d\mathfrak q.   
    \end{equation}
    From \eqref{eq:S_T-uff}, using the fundamental theorem of calculus and the definition of $\psi_q$, we get
    \begin{align*}
        \frac{S_T^\delta(r+t)-S_T^\delta(r)}{t} &\le  \int_{Q_T^\delta(r)}\dashint_{r}^{r+t} (\log h_q(s) )'h_q(s)\,\d s\d\mathfrak q\le \int_{Q_T^\delta(r)}\dashint_{r}^{r+t}( \psi_q(s)+H_{K,N}(s))h_q(s) \,\d s\d\mathfrak q.
    \end{align*}
    All in all, we can estimate the following difference quotient
    \begin{multline}
    \label{eq:nomi_finiti2}
    \frac1t\left(\frac{S_T^{\delta}(r+t)}{h_{K,N}(r+t)} - \frac{S_T^\delta(r)}{h_{K,N}(r)}\right)  = \frac{ S_T^\delta(r+t)-S_T^\delta(r)}{t h_{K,N}(r) } + \frac{ S_T^\delta(r+t)}{t} \left( \frac{1}{h_{K,N}(r+t)} - \frac{1}{h_{K,N}(r)}\right) \\
        \leq \int_{Q_T^\delta(r)}\dashint_{r}^{r+t}( \psi_q(s)+H_{K,N}(s))\frac{h_q(s)}{h_{K,N}(r)} \,\d s\d\mathfrak q - S_T^\delta(r+t) \dashint_{r}^{r+t}  \frac{h'_{K,N}(s)}{h_{K,N}(s)^2}\,\d s.
    \end{multline}
    Now, the function $[r,r+t]\ni s\mapsto H_{K,N}(s)h_q(s)$ is continuous and uniformly bounded with respect to $q\in Q_T^\delta(r)$ (cf.\ \cite[Lemma 2.15]{CavallettiMondino20}) 
    therefore, by dominated convergence theorem, we get 
    \begin{multline}
    \label{eq:limit_is_zero}
       \lim_{t\downarrow 0} \left(\int_{Q_T^\delta(r)}\dashint_{r}^{r+t}H_{K,N}(s)\frac{h_q(s)}{h_{K,N}(r)} \,\d s\d\mathfrak q- S_T^\delta(r+t) \dashint_{r}^{r+t}  \frac{h'_{K,N}(s)}{h_{K,N}(s)^2}\,\d s\right)\\ 
       = S_T^\delta(r)  \frac{H_{K,N}(r)}{h_{K,N}(r)}- S_T^\delta(r)\frac{h'_{K,N}(r)}{h_{K,N}(r)^2}\,\d s= 0,
    \end{multline}
    where we used the right-continuity of $S_T^\delta$.
    
    We now claim property \eqref{eq:right_der} for $f_1^\delta$ as defined in \eqref{eq:aux_fun_fun} with $\alpha=1$ for every $r \in [\eps,R]$. For brevity, let us denote $L\coloneqq  \lims_{t\downarrow 0} \frac{f_1^\delta(r+t)-f_1^\delta(r)}{t}$ so that, equivalently, we want to show that $L\le 0$. Using \eqref{eq:nomi_finiti2} and \eqref{eq:limit_is_zero}, we estimate
    \begin{align*}
         L & = \lims_{t\downarrow 0}\left(
        \frac1t\left(\frac{S_T^\delta(r+t)}{h_{K,N}(r+t)} - \frac{S_T^\delta(r)}{h_{K,N}(r)}\right) - \dashint_{r}^{r+t}\int_{Q_T^\delta(s)}\psi_q(s)\frac{h_q(s)}{h_{K,N}(s)} \,\d \mathfrak q \d s\right)\\
        & \leq \lims_{t\downarrow 0} \left(\int_{Q_T^\delta(r)}\dashint_{r}^{r+t}\psi_q(s)\frac{h_q(s)}{h_{K,N}(r)} \,\d s\d\mathfrak q - \dashint_{r}^{r+t}\int_{Q_T^\delta(s)}\psi_q(s)\frac{h_q(s)}{h_{K,N}(s)} \,\d \mathfrak q \d s\right)\\
        &= \lims_{t\downarrow 0} \left(\int_{Q_T^\delta(r)}\dashint_{r}^{r+t}\psi_q(s)\frac{h_q(s)}{h_{K,N}(r)} \,\d s\d\mathfrak q - \int_{Q_T^\delta(r+t)}\dashint_{r}^{r+t}\psi_q(s)\frac{h_q(s)}{h_{K,N}(s)} \,\d \mathfrak q \d s\right),
    \end{align*}
    where in the last step, we used the inclusion $Q_T^\delta(r+t)\subset Q_T^\delta(s)$ for every $s\in[r,r+t]$, together with Fubini's theorem. Since $h_{K,N}$ is smooth and positive on $[r,r+t]$, in the limit above we can replace $h_{K,N}(s)$ with $h_{K,N}(r)$, and obtain
    \begin{align*}
        L  &\leq \frac{1}{h_{K,N}(r)} \lims_{t\downarrow 0} \left(\int_{Q_T^\delta(r)}\dashint_{r}^{r+t}\psi_q(s)h_q(s) \d s\d\mathfrak q - \int_{Q_T^\delta(r+t)}\dashint_{r}^{r+t}\psi_q(s)h_q(s) \d \mathfrak q \d s\right)\\
        &= \frac{1}{h_{K,N}(r)} \lims_{t\downarrow 0} \int_Q\dashint_{r}^{r+t}\psi_q(s)h_q(s) \,\d s\left(\chi_{Q_T^\delta(r)}(q)-\chi_{Q_T^\delta(r+t)}(q)\right)\d\mathfrak q.
    \end{align*}
    Note that $(\log h_q)'$ is uniformly bounded on $[r,r+t]$, with respect to $q\in Q_T^\delta(r)$, cf.\ \cite[Eq.\ 2.13]{CavallettiMondino20}. Hence, the function $[r,r+t]\ni s\mapsto\psi_q(s)h_q(s)$ is uniformly bounded above with respect to $q\in Q_T^\delta(r)$. Thus, we apply Fatou's lemma to conclude that $L\le0$, thus proving \eqref{eq:right_der} for $f_1^\delta$ with $\alpha=1$. 

    We claim that \eqref{eq:right_der} holds for $f_\alpha^\delta$, for all $\alpha\in (0,1)$.
    Indeed, on the one hand, as a consequence of \eqref{eq:nomi_finiti2} and \eqref{eq:limit_is_zero}, given $\eta>0$, there exists $t_\eta>0$ such that for all $t\in(0,t_\eta)$, we have 
    \begin{equation*}
        \frac{S_T^\delta(r+t)}{h_{K,N}(r+t)} \leq \frac{S_T^\delta(r)}{h_{K,N}(r)}+t(A+\eta),
    \end{equation*}
    having set 
    \begin{equation*}
        A\coloneqq \int_{Q_T^\delta(r)} \dashint_r^{r+t}\psi_q(s)\frac{1}{h_{K,N}(r)}h_q(s) \,\d s\d\mathfrak q<\infty.    
    \end{equation*}
     On the other hand, by concavity of $\varphi(x)\coloneqq x^\alpha$ it holds $\varphi(x+y)\leq\varphi(x)+\varphi'(x)y$ and this implies
    \begin{equation*}
        \left(\frac{S_T^\delta(r)}{h_{K,N}(r)}+t(A+\eta)\right)^\alpha - \left(\frac{S_T^\delta(r)}{h_{K,N}(r)}\right)^\alpha \leq \alpha\left(\frac{S_T^\delta(r)}{h_{K,N}(r)}\right)^{\alpha-1}t(A+\eta).
    \end{equation*}
    It follows that 
    \begin{align}
    \label{eq:difference_quotient_f_alpha}
    \frac{f_\alpha^\delta(r+t)-f_\alpha^\delta(r)}{t} &\leq\alpha\left(\frac{S_T^\delta(r)}{h_{K,N}(r)}\right)^{\alpha-1}(A+\eta)-\alpha \dashint_r^{r+t}\int_{Q_T^\delta(s)}\left(\frac{S_T^\delta(s)}{h_{K,N}(s)}\right)^{\alpha-1}\psi_q(s)\frac{h_q(s)}{h_{K,N}(s)}\d\mathfrak q\d s. 
    \end{align}
    Hence, we conclude the proof of \eqref{eq:right_der} for $f_\alpha^\delta$ if we show that 
    \begin{equation}
    \label{eq:limsup_claim}
        \lims_{t\downarrow 0} \left(\left(\frac{S_T^\delta(r)}{h_{K,N}(r)}\right)^{\alpha-1}A- \dashint_r^{r+t}\int_{Q_T^\delta(s)}\left(\frac{S_T^\delta(s)}{h_{K,N}(s)}\right)^{\alpha-1}\psi_q(s)\frac{h_q(s)}{h_{K,N}(s)}\d\mathfrak q\d s\right) \leq 0,
    \end{equation}
    as, together with \eqref{eq:difference_quotient_f_alpha}, it would imply
    \begin{equation*}
        \lims_{t\downarrow 0}\frac{f_\alpha^\delta(r+t)-f_\alpha^\delta(r)}{t} \leq \eta\alpha \left(\frac{S_T^\delta(r)}{h_{K,N}(r)}\right)^{\alpha-1},
    \end{equation*}
    and this gives \eqref{eq:right_der} for $f_\alpha^\delta$ by the arbitrariness of $\eta>0$. For proving \eqref{eq:limsup_claim}, we use the expression of $A$, the inclusion $Q_T^\delta(r+t)\subset Q_T^\delta(s)$ for every $s\in [r,r+t]$ and Fubini's theorem to estimate 
    {\allowdisplaybreaks
    \begin{align}
    \lims_{t\downarrow 0}& \left(\left(\frac{S_T^\delta(r)}{h_{K,N}(r)}\right)^{\alpha-1}A- \dashint_r^{r+t}\int_{Q_T^\delta(s)}\left(\frac{S_T^\delta(s)}{h_{K,N}(s)}\right)^{\alpha-1}\psi_q(s)\frac{h_q(s)}{h_{K,N}(s)}\d\mathfrak q\d s\right) \nonumber \\
    &=\lims_{t\downarrow 0}\int_{Q_T^\delta(r)} \dashint_r^{r+t}\psi_q(s)\frac{h_q(s)}{h_{K,N}(r)}\left(\left(\frac{S_T^\delta(r)}{h_{K,N}(r)}\right)^{\alpha-1} -\chi_{Q_T^\delta(r+t)}\left(\frac{S_T^\delta(s)}{h_{K,N}(s)}\right)^{\alpha-1}\right)\d s\d\mathfrak q \nonumber \\
    & \leq C\lims_{t\downarrow 0}\int_{Q_T^\delta(r)} \dashint_r^{r+t}\left|\left(\frac{S_T^\delta(r)}{h_{K,N}(r)}\right)^{\alpha-1} -\chi_{Q_T^\delta(r+t)}\left(\frac{S_T^\delta(s)}{h_{K,N}(s)}\right)^{\alpha-1}\right|\d s\d\mathfrak q,\nonumber \\
    & \leq C\lims_{t\downarrow 0}\int_{Q_T^{M,\delta}(r)} \dashint_r^{r+t}\left|\left(\frac{S_T^{M,\delta}(r)}{h_{K,N}(r)}\right)^{\alpha-1} -\left(\frac{S_T^{M,\delta}(s)}{h_{K,N}(s)}\right)^{\alpha-1}\right|\d s\,\d \mathfrak q\label{eq:term1}\\
    &\qquad\qquad\qquad\qquad\qquad\qquad\qquad+C \lim_{t \downarrow 0} \dashint_r^{r+t}\left(\frac{S_T^{M,\delta}(s)}{h_{K,N}(s)}\right)^{\alpha-1}\d s \,  {\mathfrak q}(Q_T^{M,\delta}(r)\setminus Q_T^{M,\delta}(r+t))\label{eq:term2}.
    \end{align}
    }
    \!\!where $C>0$ is an upper bound for the function $ \psi_q(s)\frac{h_q(s)}{h_{K,N}(r)}$ on $[r,r+t]\times Q_T^\delta(r)$. Now, the term \eqref{eq:term2} converges to $0$, since ${\mathfrak q}(Q_T^{M,\delta}(r)\setminus Q_T^{M,\delta}(r+t))\to 0$ as $t\to 0^+$ and the ratio $\frac{S_T^{M,\delta}(s)}{h_{K,N}(s)}$ is uniformly bounded on $[r,r+t]$. For the the term \eqref{eq:term1}, we observe that the integrand function 
    \begin{equation*}
        g(s)\coloneqq \left|\left(\frac{S_T^{M,\delta}(r)}{h_{K,N}(r)}\right)^{\alpha-1} -\left(\frac{S_T^{M,\delta}(s)}{h_{K,N}(s)}\right)^{\alpha-1}\right|
    \end{equation*}
    is right-continuous and bounded on $[r,r+t]$. Therefore, by the fundamental theorem of calculus, the right-derivative of $x\mapsto\int_0^x g(s)\d s$ at $r$ is $g(r)=0$. Hence, by Fatou's lemma, also \eqref{eq:term1} converges to $0$. This shows \eqref{eq:limsup_claim}, and thus establishes \eqref{eq:right_der} for $f_\alpha^\delta$.

    We are in position to conclude the proof of iii). Since $f^\delta_\alpha$ is left lower semicontinuous and satisfies \eqref{eq:right_der} on $[\epsilon,R]$, it is non-increasing on $[\epsilon,R]$ as a consequence of Lemma \ref{lem:monotonicity}. The conclusion of the proof now follows since $f^\delta_\alpha$ converges pointwise to $f_\alpha$ as $\delta\uparrow\infty$ (note that $\frac{\delta r}{\delta-1}\downarrow r$, hence $Q_T^\delta(r)\subset Q_T^{\delta'}(r)$ for $\delta'\geq\delta$ and $Q_T(r)=\bigcup_{\delta>0}Q_T^\delta(r)$, meaning that $\nchi_{Q_T^\delta(r)}\to \nchi_{Q_T(r)}$ in $L^1(\mathfrak q)$ as $\delta\to\infty$) and by arbitrariness of $0<\epsilon<R<\pi\sqrt{\frac{N-1}{K^+}}$.
\end{proof}
\subsection{Proof of Bishop-Gromov comparison}
We prove our first main result. Observe that, under the standing assumptions, Lemma \ref{lem:average integral tend zero} holds, hence we are in position to apply Proposition \ref{prop:mean vs deficit 1D}.

\begin{proof}[Proof of Theorem \ref{thm:BG}]
    Let us localize the problem using the localization method developed in \cref{thm:localization} with the $1$-Lipschitz function $u= \sfd_x$. More precisely, we have the disintegration
    \[
    \mm= \int \mm_q \, \d \mathfrak q,
    \]
    and, for $\mathfrak q$-a.e.\ $q\in Q$, $\mm_q = g(q,\cdot)_\sharp(h_q \Leb^1)$, where $h_q:[0,r_q]\to [0,\infty)$ is a ${\sf CD}( k\circ g(q,\cdot) ,N)$ density on ${\rm Dom}(g(q,\cdot))$. Denote by $\psi_q$ the mean curvature deficit of $h_q$ as in \eqref{eq:mean_curvature_def_q} and $V_T(r)=\mm(B_r(x)\cap T)$. By Lemma \ref{lem:polar_T}, we know that $V_T$ is locally absolutely continuous, hence differentiable for a.e.\ $r$ and
    \begin{equation}
            \frac{\d}{\d r}\frac{V_T(r)}{v_{K,N}(r)} = 
    \frac{V_T'(r)\,v_{K,N}(r)-V_T(r)\,v_{K,N}'(r)}{(v_{K,N}(r))^2} \eqqcolon \frac{N(r)}{D(r)},\qquad\text{a.e. }r\in (0,\infty).
    \label{eq:ODE BG}
    \end{equation}
    Using Lemma \ref{lem:2.2fatto_meglio}, we can estimate the numerator as following
    \begin{equation}\label{eq:numerator_BG}
    \begin{split}
        N(r)&=\int_0^r S_T(r)h_{K,N}(t)-S_T(t)h_{K,N}(r)\d t =  \int_0^r  \left(\frac{S_T(r)}{h_{K,N}(r)} - \frac{S_T(t)}{h_{K,N}(t)}\right) h_{K,N}(r)h_{K,N}(t)\,\d t \\
        &\le\int_0^r h_{K,N}(r)h_{K,N}(t)\int_t^r \int_{Q_T(s)} \psi_q(s)\frac{h_q(s)}{h_{K,N}(s)}\,\d \mathfrak q\d s \d t\\
        &\le \int_0^r h_{K,N}(r)h_{K,N}(t)\underbrace{\int_t^r \left( \int_{Q_T(s)} \psi_q(s)^{2p-1}h_q(s)\,\d \mathfrak q\right)^{\frac{1}{2p-1}} S_T(s)^{1-\frac{1 }{2p-1}} \frac1{h_{K,N}(s)}\,\d s}_{\eqqcolon{\mathcal I}(t,r)}\d t,  
    \end{split}
    \end{equation}
    where in the last inequality we used H\"{o}lder's inequality.
    From here, we distinguish two cases.

    \noindent\textsc{Case $r\le \frac{\pi}{2}\sqrt{\frac{N-1}{K^+}}$}. 
    We estimate here ${\mathcal I}(t,r)$ for $t\in (0,r)$ with \cref{prop:mean vs deficit 1D} as follows
    \begin{align*}
        {\mathcal I}(t,r) &\le  \int_t^r
        \left( \int_{Q_T(s)}  \int_0^s  | (k\circ g(q,\tau) -K)\wedge  0 |^p h_q(\tau)\, \d \tau\,\d \mathfrak q\right)^{\frac{1}{2p-1}} S_T(s)^{1-\frac{1 }{2p-1}} h_{K,N}(s)^{-1}\,\d s \\ 
        & \leq  \alpha_{N,p}^{\frac{1}{2p-1}} \rho_p^k(T,K)^{\frac{1}{2p-1}}  \int_t^rS_T(s)^{1-\frac{1 }{2p-1}} h_{K,N}(s)^{-1}\,\d s,
    \end{align*}
    having also used \cref{lem:polar_T}, where $\alpha_{N,p}$ is the constant appearing in \cref{prop:mean vs deficit 1D}.  Notice that, since we have $t\le s \le r\le \frac{\pi}{2}\sqrt{\frac{N-1}{K^+}}$, for any $K\in\R$ we have $h_{K,N}(s)^{-1} \le h_{K,N}(t)^{-1}$. In addition, applying Jensen's inequality with the concave function $\varphi(x)=x^{1-\frac1{2p-1}}$, we obtain:
    \begin{equation}
    \label{eq:jensen}
        \int_t^rS_T(s)^{1-\frac1{2p-1}}\d s\leq (r-t)^{\frac1{2p-1}}\left(V_T(r)-V_T(t)\right)^{1-\frac1{2p-1}}.
    \end{equation}
    Thus, denoting by $\alpha'_{N,p}\coloneqq \alpha_{N,p}^{\frac{1}{2p-1}}$, we estimate ${\mathcal I}(t,r)$ as following
    \begin{equation}\label{eq:case 1 end}
    \begin{aligned}
        {\mathcal I}(t,r) &\le  \alpha'_{N,p} \rho_p^k(T,K)^{\frac{1}{2p-1}} h_{K,N}(t)^{-1}  \int_t^r S_T(s)^{1-\frac{1 }{2p-1}}\,\d s \\
        &\le \alpha'_{N,p} \rho_p^k(T,K)^{\frac{1}{2p-1}}  h_{K,N}(t)^{-1} (r-t)^{\frac1{2p-1}}\left(V_T(r)-V_T(t)\right)^{1-\frac1{2p-1}}\\
        &\le \alpha'_{N,p} \rho_p^k(T,K)^{\frac{1}{2p-1}}  h_{K,N}(t)^{-1} r^{\frac1{2p-1}}V_T(r)^{1-\frac1{2p-1}},
    \end{aligned}
    \end{equation}
    where in the last inequality we used $r-t\leq r$ and $V_T(r)-V_T(t)\leq V_T(r)$.
    
    \noindent\textsc{Case $\frac{\pi}{2}\sqrt{\frac{N-1}{K}}< r\le R \le \pi \sqrt{\frac{N-1}{K}} $ and $K>0$}. 
    We carry on estimating ${\mathcal I}(t,r)$ in this case by splitting the integral as follows
    \[
        {\mathcal I}(t,r)= {\mathcal I}\left(\frac{\pi}{2}\sqrt{\tfrac{N-1}{K}}\wedge t,\frac{\pi}{2}\sqrt{\tfrac{N-1}{K}}\right) + {\mathcal I}\left( \frac{\pi}{2}\sqrt{\tfrac{N-1}{K}}\vee t,r \right).
    \]
    The first addendum can be handled by the analysis of the previous case using \eqref{eq:case 1 end}. Hence, we focus on the second addendum: we multiply and divide in the integral by $h_{K,N}(s)^{\tfrac{4p-N-1}{N-1}}$ and estimate
    \begin{equation*}
    \begin{aligned}
        {\mathcal I}\bigg( \frac{\pi}{2} &\sqrt{\tfrac{N-1}{K}}\vee t,r \bigg) \\
        &= \int_{\tfrac{\pi}{2}\sqrt{\frac{N-1}{K}}\vee t}^r \left( \int_{Q_T(s)} h_{K,N}(s)^{\tfrac{4p-N-1}{N-1}}\psi_q(s)^{2p-1}h_q(s)\,\d \mathfrak q\right)^{\frac{1}{2p-1}} \frac{ S_T(s)^{1-\frac{1 }{2p-1}}}{h_{K,N}(s)^{1+\tfrac{4p-N-1}{(N-1)(2p-1)}}}\,\d s\\
        & \leq \frac{ \alpha'_{N,p}}{h_{K,N}(r)^{1+\tfrac{4p-N-1}{(N-1)(2p-1)}}} \rho^k_p(T,K)^{\frac{1}{2p-1}}   \int_{\tfrac{\pi}{2}\sqrt{\frac{N-1}{K}}\vee t}^rS_T(s)^{1-\frac{1 }{2p-1}}\,\d s \\
        &\overset{\eqref{eq:jensen}}{\leq} \frac{ \alpha'_{N,p}}{h_{K,N}(r)^{1+\tfrac{4p-N-1}{(N-1)(2p-1)}}} \rho^k_p(T,K)^{\frac{1}{2p-1}}\left(\tfrac{\pi}{2}\sqrt{\tfrac{N-1}{K}}\right)^{\frac1{2p-1}}V_T(r)^{1-\frac1{2p-1}}.
    \end{aligned}
    \end{equation*}
    where, after the application of Jensen's inequality \eqref{eq:jensen}, we used $r-\left(\tfrac{\pi}{2}\sqrt{\frac{N-1}{K}}\vee t\right)\leq \tfrac{\pi}{2}\sqrt{\frac{N-1}{K}}$. Therefore, using \eqref{eq:case 1 end}, the above and the fact that $h_{K,N}$ has a maximum at $\tfrac{\pi}{2}\sqrt{\tfrac{N-1}{K}}$, we obtain
    \begin{align*}
        {\mathcal I}(t,r) &\le  \alpha'_{N,p} \rho_p^k(T,K)^{\frac{1}{2p-1}}   \left(\tfrac{\pi}{2}\sqrt{\tfrac{N-1}{K}}\right)^{\frac1{2p-1}}h_{K,N}\left(\tfrac{\pi}{2}\sqrt{\tfrac{N-1}{K}}\wedge t\right)^{-1}V_T\left(\tfrac{\pi}{2}\sqrt{\tfrac{N-1}{K}}\right)^{1-\frac1{2p-1}} \\
        &\qquad\qquad\qquad\qquad +\alpha'_{N,p} h_{K,N}(r)^{-1-\tfrac{4p-N-1}{(N-1)(2p-1)}} \rho^k_p(T,K)^{\frac{1}{2p-1}}\left(\tfrac{\pi}{2}\sqrt{\tfrac{N-1}{K}}\right)^{\frac1{2p-1}}V_T(r)^{1-\frac1{2p-1}}\\
        &\le \left( \alpha_{N,p}\tfrac{\pi}{2}\sqrt{\tfrac{N-1}{K}}\rho_p^k(T,K) \right)^{\frac1{2p-1}}V_T(r)^{1-\frac1{2p-1}}\Bigg(  h_{K,N}(t)^{-1} +h_{K,N}(r)^{-1-\tfrac{4p-N-1}{(N-1)(2p-1)}}\Bigg),
    \end{align*}
    having also used that $V_T$ is monotone non-decreasing.

    \noindent\textsc{Conclusion}. All in all, by integrating the identity
    \[
    \frac{\d }{\d t} \left(\frac{V_T(t)}{v_{K,N}(t)}\right)^{\frac1{2p-1}}= \frac{1}{2p-1}\left(\frac{V_T(t)}{v_{K,N}(t)}\right)^{\frac1{2p-1}-1} \frac{N(t)}{D(t)},\qquad \text{a.e.\ }t
    \]
    from $r$ to $R$, and estimating $N(\cdot)$ by combining \eqref{eq:numerator_BG} and the estimates for ${\mathcal I}(t,r)$ from the two cases,  we finally get  \eqref{eq:bishop_gromov_quantitative} with the constant
    \begin{equation}\label{eq:CpKNR}
    \begin{aligned}
        C_{K,N,p}(R) &\coloneqq  \left(\frac{N-1}{(2p-1)(2p-N)}\right)^{\frac{p-1}{2p-1}}  \Bigg[ \int_0^{\tfrac{\pi}{2}\sqrt{\frac{N-1}{K^+}}\wedge R}  h_{K,N}(t)\left( \frac{t}{v_{K,N}(t)}\right)^{1+\frac{1}{2p-1}}\, \d t \\
        & +  \left(\tfrac{\pi}{2}\sqrt{\tfrac{N-1}{K^+}} \right)^{\frac1{2p-1}}\int_{\tfrac{\pi}{2}\sqrt{\frac{N-1}{K^+}}}^{\frac{\pi}{2}\sqrt{\frac{N-1}{K^+}}\vee R} \left(\frac{1}{v_{K,N}(t)}\right)^{\frac1{2p-1}+1}\Big(th_{K,N}(t)+ \frac{v_{K,N}(t)}{h_{K,N}(t)^{\tfrac{4p-N-1}{(N-1)(2p-1)}}}\Big) \, \d t  \Bigg],
    \end{aligned}
    \end{equation}
    where the second addendum is set to zero if $K\le 0$. Notice that, recalling the assumption $p>N/2$ and taking into account the asymptotics $v_{K,N}(t) \approx t^N$, $h_{K,N}(t)\approx t^{N-1}$ for $t$ small, it holds
    \[
      \int_0^{\frac{\pi}{2}\sqrt{\tfrac{N-1}{K^+}}\wedge R}  h_{K,N}(t)\left( \frac{t}{v_{K,N}(t)}\right)^{1+\frac{1}{2p-1}}\, \d t <\infty,
    \]
    whence $  C_{K,N,p}(R)<\infty$. We also observe that, if $K<0$ and $t>1$, then $\frac 14 \exp\left( t(N-1)\sqrt{\tfrac{K}{N-1}}\right) \le h_{K,N}(t) \le   \exp\left( t(N-1)\sqrt{\tfrac{K}{N-1}}\right)$, whence $(0,\infty)\ni R\mapsto C_{K,N,p}(R)$ is globally bounded. Finally, since $\theta_N(x) = \lim_{r \downarrow 0} v(r)/v_{K,N}(r)$ (it is independent of $K$), we obtain \eqref{eq:bishop_gromov_quantitative_theta} and the proof is concluded.
\end{proof}
\begin{remark}\label{rem:Ket vs Noi} \rm
    Theorem \ref{thm:BG} has been partially addressed in \cite[Cor. 3.4, 3.5]{Ket21} (limited to $K\le 0$ and balls), without imposing the assumption $\mm(B_\varrho(x))=o(\varrho)$. To derive this estimate, \cite{Ket21} relies on a one-dimensional analysis built upon mean curvature comparison estimates, previously developed for \emph{smooth} polar coordinates in \cite{PetersenSprouse98, PetersenWei97} and later sharpened in \cite{Aubry2007}. The latter is reported in \cite[Prop. 4.1]{Ket21} for the smooth weight $\omega(r) = \sin_{\kappa/(N-1)}^{N-1}(r)$, but applied in \cite[Step 3]{Ket21} to \emph{nonsmooth} weights arising from the disintegration of $\mm$. In contrast, Proposition \ref{prop:mean vs deficit 1D} shows via a regularization argument that such comparison principles remain valid for nonsmooth weights under an additional boundary condition. Our extra assumption $\mm(B_\varrho(x))=o(\varrho)$, together with Lemma \ref{lem:average integral tend zero}, guarantees that this condition is met, thereby enabling the proof of Theorem \ref{thm:BG}. This assumption rules out settings that are too collapsed (cf. Remark \ref{rem:necessity theta 1+c}). \fr
\end{remark}
\subsection{Consequences of Bishop-Gromov comparison}
We deduce uniform local doubling properties of the reference measure $\mm$ provided the \emph{scaling invariant integral curvature deficit} is small. For a bounded star-shaped set $T$ at $x\in\X$ such that $T\subset B_R(x)$ for some $R>0$, the latter is defined as
\begin{equation*}
    R^2\left(\frac{\rho^k_p(T,K)}{\mm(T)}\right)^{\frac 1p}.
\end{equation*} 
The above quantity is indeed invariant for rescalings of both the distance $\sfd$ and the measure $\mm$ in the following sense: if $(\X',\sfd',\mm') \coloneqq (\X,\alpha\cdot \sfd,\beta\cdot \mm)$ for $\alpha,\beta>0$, then it holds that $T\subset B'_{R'}(x) $ where $B'_r(x) = \{y \in \X:\,\alpha \sfd_x(y)<r\}$ and $R' \coloneqq \alpha R$. In particular, we have
    \begin{equation}
        \label{eq:scaling}
        R^2\left(\frac{\rho^k_p(T,K)}{\mm(T)}\right)^{\frac 1p} = (R')^2 \left( \frac{\int_{T} | \alpha^{-2}(k-K)\wedge 0 |^p\,\d \mm'}{\mm'(T)}\right)^\frac{1}{p} =   (R')^2\left(\frac{\rho^{k/\alpha^2}_p(T,K/\alpha^2)}{\mm'(T)}\right)^{\frac 1p},
    \end{equation}
    where the latter integral deficit is computed with $\mm'$.

\begin{proposition}\label{prop:doubling uniform at x}
   For every $N>1,K\in\R,p>N/2$ and $0< R\le \pi\sqrt{\tfrac{N-1}{K^+}}$ there exist $A_{K,N,p,R}>0$ with $A_{0,N,p}\coloneqq A_{0,N,p,R}$ independent of $R$ if $K=0$ such that the following holds. Let $(\X,\sfd,\mm)$ be an essentially non-branching ${\sf CD}(k,N)$ space for some $k\colon \X\to\R$ admissible and suppose that $x \in \X$ is such that $\mm(B_\varrho(x))=o(\varrho)$. If $T$ is a star-shaped set at $x$ with $T\subset B_R(x)$ and
   \[
        \eps\coloneqq R^2\left(\frac{\rho^k_p(T,K)}{\mm(T)}\right)^{\frac 1p} \le   A_{K,N,p,R},
    \]
    then, it holds
   \begin{equation}\label{eq:uniform doubling starshaped}
        \left(1-2 C_{K,N,p}(R){v_{K,N}(R)^{\frac{1}{2p-1}}}  R^{-\frac{2p}{2p-1}} \eps^{\frac{p}{2p-1}}\right)^{2p-1}\frac{\mm(B_r(x)\cap T)}{\mm(B_t(x)\cap T)}\le \frac{v_{K,N}(r)}{v_{K,N}(t)},\qquad\forall\, 0< t\le r\le R,  
    \end{equation}
    where $ C_{K,N,p}(R)$ is given by Theorem \ref{thm:BG}. Moreover, choosing $T=B_R(x)$, it holds
    \[  
        \frac{\mm(B_r(x))}{\mm(B_t(x))}\le 2 \frac{v_{K,N}(r)}{v_{K,N}(t)},\qquad\forall\,0< t\leq r\le R,
    \]
\end{proposition}
\begin{proof}
    The last conclusion is obvious once we show \eqref{eq:uniform doubling starshaped}, possibly decreasing the constant $A_{K,N,p,R}$. To this aim, rearranging in \eqref{eq:bishop_gromov_quantitative}, we get
    \begin{align}
    \label{eq:guess_what_itsBG}
         \mm(B_r(x)\cap T)^{\frac{1}{2p-1}} &\le \left(\frac{v_{K,N}(r)}{v_{K,N}(t)} \mm(B_t(x)\cap T)\right)^{\frac{1}{2p-1}} +  C_{K,N,p}(R) v_{K,N}(r)^{\frac{1}{2p-1}} R^{-\frac{2p}{2p-1}} \mm(T)^{\frac{1}{2p-1}} \eps^{\frac{p}{2p-1}},
    \end{align}
    for all $t\le r\le R$, where we used that $T\subset B_R(x)$. By \eqref{eq:guess_what_itsBG} with $(r,R)$ in place of $(t,r)$, we get
    \begin{align*}
         \mm(T)^{\frac{1}{2p-1}} &\le \left(\frac{v_{K,N}(R)}{v_{K,N}(r)} \mm(B_r(x)\cap T)\right)^{\frac{1}{2p-1}} +  C_{K,N,p}(R) v_{K,N}(R)^{\frac{1}{2p-1}} R^{-\frac{2p}{2p-1}} \mm(T)^{\frac{1}{2p-1}} \eps^{\frac{p}{2p-1}}\\
        &\le \left(\frac{v_{K,N}(R)}{v_{K,N}(r)} \mm(B_r(x)\cap T)\right)^{\frac{1}{2p-1}} + \frac12\mm(T)^{\frac{1}{2p-1}},
    \end{align*}
    provided that $A_{K,N,p,R}>0$ is chosen small enough so that 
    \begin{equation}
    \label{eq:choice_B}
       C_{K,N,p}(R) v_{K,N}(R)^{\frac{1}{2p-1}} R^{-\frac{2p}{2p-1}}A_{K,N,p,R}^\frac{p}{2p-1}<\frac12.
    \end{equation}
    This implies $\mm(T)^{\frac{1}{2p-1}} \leq 2\left(\frac{v_{K,N}(R)}{v_{K,N}(r)} \mm(B_r(x)\cap T)\right)^{\frac{1}{2p-1}} $ and plugging this in \eqref{eq:guess_what_itsBG} we obtain,
    \begin{multline*}
        \mm(B_r(x)\cap T)^{\frac{1}{2p-1}} 
         \le \left(\frac{v_{K,N}(r)}{v_{K,N}(t)} \mm(B_t(x)\cap T)\right)^{\frac{1}{2p-1}} \\
          + 2 C_{K,N,p}(R) R^{-\frac{2p}{2p-1}}\left(v_{K,N}(R) \mm(B_r(x)\cap T)\right)^{\frac{1}{2p-1}} \eps^{\frac{p}{2p-1}},
    \end{multline*}
    for all $t\le r\le R$.  Dividing by $\mm(B_t(x)\cap T)^{\frac1{2p-1}}$, and by the choice of $A_{K,N,p,R}$, we finally get \eqref{eq:uniform doubling starshaped}.

    To conclude the proof, note that, if $K=0$, then $A_{0,N,p,R}$ can be taken small independently of $R$. Indeed, by the explicit formula of $ C_{K,N,p}(R)$ when $K=0$ (cf. \eqref{eq:Cp_ZERO_NR}), the requirement \eqref{eq:choice_B} reads
    \[
        \frac 12 > \left(\frac{N-1}{(2p-1)(2p-N)}\right)^{\frac{p-1}{2p-1}}  R^{\frac{2p-N}{2p-1}} R^{-\frac{2p}{2p-1}} R^{\frac{N}{2p-1}}A_{0,N,p,R}^\frac{p}{2p-1} =   \left(\frac{N-1}{(2p-1)(2p-N)}\right)^{\frac{p-1}{2p-1}}  A_{0,N,p,R}^\frac{p}{2p-1}.
    \]
\end{proof}
The second consequence is a doubling property for the measure of balls centered at different points. For this proof, we restrict to the case $K=0$, being enough for our purposes (cf. \cite[Lemma 5.1]{Aubry2007}).
\begin{lemma}\label{prop:volume_growth}
     For every $N >1$ and $p>N/2$, there exist $\beta_{N,p},B_{N,p}>0$ such that the following holds. Let $(\X,\sfd,\mm)$ be an essentially non-branching ${\sf CD}(k,N)$ space for some $k\colon \X\to\R$ admissible. Assume that $\mm(B_\varrho(z))=o(\varrho)$ at $\mm$-a.e.\ $z\in\X$. If $x\in\X$ with $\mm(B_\varrho(x))=o(\varrho)$ satisfies
     \[
         \eps \coloneqq R^2 \left(\frac{\rho_p^k(B_{R}(x),0)}{\mm(B_{R}(x))}\right)^\frac{1}{p} \le B_{N,p},
     \]
    then, for all $y \in B_{R}(x),r>0$ with $\sfd(x,y) + r \le R$, it holds
    \begin{equation}
        \label{eq:ball_chaining_2}    \left(\frac{\mm(B_{r}(y))}{\mm(B_R(x))} \right)^\frac{1}{2p-1}\geq \left(\frac{r}{R}\right)^{\frac{N}{2p-1}}\left(\left(\frac{r}{R}\right)^{\frac{2N}{2p-1}}\left( 1- \beta_{N,p}\eps^{\frac{p}{2p-1}}\right)-\beta_{N,p}  \eps^\frac{p}{2p-1}\right).
    \end{equation}
\end{lemma}
\begin{proof}
    Let $z\in B_R(x)$ be such that $\mm(B_\varrho(z))=o(\varrho)$. Let $0<r_1\leq r_2< R-\sfd(x,z)$, and denote by $\beta'_{N,p}\coloneqq  \left(\frac{N-1}{(2p-1)(2p-N)}\right)^{\frac{p-1}{2p-1}} $. Then, by inequality \eqref{eq:bishop_gromov_quantitative} for $K=0$, we have 
    \begin{align*}
        \left( \frac{\mm(B_{r_2}(z))}{r_2^N} \right)^\frac{1}{2p-1} - \left( \frac{\mm(B_{r_1}(z))}{r_1^N} \right)^\frac{1}{2p-1} &\le \beta'_{N,p}\,{r_2}^{\frac{2p-N}{2p-1}} \rho^k_p(B_{r_2}(z),0)^\frac{1}{2p-1}\\
        &\le \beta'_{N,p} \,{r_2}^{\frac{2p-N}{2p-1}} \rho^k_p(B_{R}(x),0)^\frac{1}{2p-1}\\
        &= \beta'_{N,p}\,{r_2}^{\frac{2p-N}{2p-1}}{R}^{\frac{2p}{2p-1}}\mm(B_R(x))^{\frac1{2p-1}}\eps^\frac{p}{2p-1},
    \end{align*}
    using the inclusion $B_{r_2}(z)\subset B_R(x)$. Rearranging the terms in the above inequality, we obtain 
    \begin{equation}
    \label{eq:bg_massaggiato}
        \left( \frac{\mm(B_{r_2}(z))}{\mm(B_R(x))} \right)^\frac{1}{2p-1}\leq \left( \frac{\mm(B_{r_1}(z))}{\mm(B_R(x))} \right)^\frac{1}{2p-1}\left(\frac{r_2}{r_1}\right)^{\frac{N}{2p-1}} + \beta'_{N,p}\left(\frac{r_2}{R}\right)^{\frac{2p}{2p-1}}  \eps^\frac{p}{2p-1}.
    \end{equation}
    We claim that inequality \eqref{eq:bg_massaggiato} holds for every $z\in B_R(x)$. Indeed, let $z\in B_R(x)$, then, there exists a sequence $\{z_n\}_{n\in \N}$ converging to $z$ and such that $\mm(B_\varrho(z_n)) =o(\varrho)$ for all $n$. In particular, we can apply \eqref{eq:bg_massaggiato} to the point $z_n$, so that for big enough $n \in \N$ it holds 
    \begin{equation}
    \label{eq:bg_massaggiato_ennesimo}
        \left( \frac{\mm(B_{r_2}(z_n))}{\mm(B_R(x))} \right)^\frac{1}{2p-1}\leq \left( \frac{\mm(B_{r_1}(z_n))}{\mm(B_R(x))} \right)^\frac{1}{2p-1}\left(\frac{r_2}{r_1}\right)^{\frac{N}{2p-1}} + \beta'_{N,p}\left(\frac{r_2}{R}\right)^{\frac{2p}{2p-1}}  \eps^\frac{p}{2p-1}.
    \end{equation}
    Moreover, for a fixed $\rho >0$, we have that
    \begin{equation*}
        \lims_{y \to z} \left| \mm(B_\rho(y)) - \mm(B_\rho(z))  \right| \le \lims_{y \to x} \left|  \int \chi_{B_\rho(y)}-\chi_{B_\rho(z)}\,\d \mm  \right|= \mm(\{ w \in \X:\, \sfd(w,z) =\rho \})=0,
    \end{equation*}
    where the last equality follows from \cite[Lemma 3.7]{KorteLahti2014} (they assume that $\mm$ is doubling, nonetheless \cite[Thm.\ 5.9]{Ket17} is enough for our purposes). Thus,
    the function $z \mapsto \mm(B_\rho(z))$ is continuous.
    By taking the limit in \eqref{eq:bg_massaggiato_ennesimo} we obtain \eqref{eq:bg_massaggiato}, thus proving the claim.
    
    We now prove \eqref{eq:ball_chaining_2} for every $y\in B_R(x)$ and $r>0$ with $\sfd(x,y)+r\leq R$, via a ball-chaining argument. Let $\eta\in (0,1)$ to be chosen later. Set $L\coloneqq \sfd(x,y)$ and let $\gamma \colon [0,L]\to\X$ be a geodesic parametrized by arc-length such that $\gamma_0=y$, $\gamma_L=x$, and set $y_1\coloneqq y$, $R_1\coloneqq r$. Then, we define 
    \begin{equation*}
        y_i\coloneqq \gamma \left(\gamma^{-1}(y_{i-1})+(1-\eta) R_{i-1}\right),\qquad r_i\coloneqq \eta R_{i-1},\qquad R_i\coloneqq \left(2-\eta\right)^{i-1} r.
    \end{equation*}
     for all $2\le i\le M-1$, where $M\coloneqq \Big\lfloor \frac{\log\left( 1+ \frac{\sfd(x,y)}{r}\right)}{\log(2-\eta)} +1 \Big\rfloor$. The iterative choices of $y_i,r_i,R_i$ for $i=2,\dots,M-1$ are done in such a way that the following three properties hold:
    \begin{itemize}
        \item[i)] the point $y_{i}$ is chosen along the geodesic $\gamma$, such that $\sfd(y_{i},y_{i-1})=(1-\eta) R_{i-1}$; 
        \item[ii)] $R_i$ is chosen to be the largest number such that $B_{R_i}(y_i)\subset B_{R}(x)$;
        \item[iii)] $r_{i}$ is chosen to be the largest number such that $B_{r_i}(y_i) \subset B_{R_{i-1}}(y_{i-1})$.
   \end{itemize}
    Lastly, we define $y_M\coloneqq x$, $r_M\coloneqq \eta R_{M-1}$, $R_M\coloneqq \left(2-\eta\right)^{M-1} r$. For every $i=2,\ldots,M$, by \eqref{eq:bg_massaggiato} applied with $z\coloneqq y_i$, $(r_1,r_2)\coloneqq (r_i,R_i)$, we obtain  
    \begin{equation}
    \label{eq:bg_massaggiato_step_i}
    \begin{split}
        \left( \frac{\mm(B_{R_i}(y_i))}{\mm(B_R(x))} \right)^\frac{1}{2p-1}&\leq \left( \frac{\mm(B_{r_i}(y_i))}{\mm(B_R(x))} \right)^\frac{1}{2p-1}\left(\frac{2-\eta}{\eta}\right)^{\frac{N}{2p-1}} + \beta'_{N,p}\left(\frac{R_i}{R}\right)^{\frac{2p}{2p-1}}  \eps^\frac{p}{2p-1},\\
        &\leq \left( \frac{\mm(B_{R_{i-1}}(y_{i-1}))}{\mm(B_R(x))} \right)^\frac{1}{2p-1}\left(\frac{2-\eta}{\eta}\right)^{\frac{N}{2p-1}} + \beta'_{N,p}\left(2-\eta\right)^{\frac{2p(i-1)}{2p-1}}\left(\frac{r}{R}\right)^{\frac{2p}{2p-1}}  \eps^\frac{p}{2p-1},
    \end{split}
    \end{equation}
    having used that $R_i=\frac{2-\eta}{\eta} r_i$ for every $2 \le i \le M$, and the fact that $B_{r_i}(y_i)\subset B_{R_{i-1}}(y_{i-1})$, by our choice of radii. For the sake of notation, denote by 
    \begin{equation*}
        \xi\coloneqq \left(\frac{2-\eta}{\eta}\right)^{\frac{N}{2p-1}}\qquad\text{and}\qquad \zeta\coloneqq (2-\eta)^{\frac{2p}{2p-1}}.
    \end{equation*}
    Iterating \eqref{eq:bg_massaggiato_step_i} starting from $i=M-1$, we deduce that
    \begin{equation}
    \label{eq:bg_iterato}
        \left( \frac{\mm(B_{R_{M-1}}(y_{M-1}))}{\mm(B_R(x))} \right)^\frac{1}{2p-1}\leq\xi^{M-2}\left(\left( \frac{\mm(B_{r}(y))}{\mm(B_R(x))} \right)^\frac{1}{2p-1}+\beta'_{N,p}\left(\frac{r}{R}\right)^{\frac{2p}{2p-1}}  \eps^\frac{p}{2p-1} \sum_{j=0}^{M-3} \left(\frac{\zeta}{\xi}\right)^{M-j-2}\right).
    \end{equation}
    To estimate the sum appearing in \eqref{eq:bg_iterato}, we distinguish two cases, according to the magnitude of $p$. 
    
    \noindent{\sc Case $p\in\big(\frac{N}2,N\big)$.} In this range, we may choose $\eta$ sufficiently close to $1$ to ensure that
    \begin{equation}
    \label{eq:assumption_parameter_eta}
        \frac{\zeta}{\xi}=
        \eta^{\frac{N}{2p-1}}(2-\eta)^{\frac{2p-N}{2p-1}}< 1.
    \end{equation}
    Therefore, the finite series appearing in \eqref{eq:bg_iterato} can be estimated as follows 
    \begin{equation*}
        \sum_{j=0}^{M-3} \left(\frac\zeta\xi\right)^{M-j-2}=\frac{1-\left(\frac\zeta\xi\right)^{M-1}}{1-\frac\zeta\xi}-1\stackrel{\eqref{eq:assumption_parameter_eta}}{\leq} \frac{1}{1-\frac\zeta\xi}.
    \end{equation*}
    In conclusion, from \eqref{eq:bg_iterato}, we get the inequality 
    \begin{multline}
    \label{eq:bg_massaggiato_iterato1}
      \left( \frac{\mm(B_{R_{M-1}}(y_{M-1}))}{\mm(B_R(x))} \right)^\frac{1}{2p-1} \leq\left(\frac{2-\eta}{\eta}\right)^{\frac{N(M-1)}{2p-1}}\Bigg(\left(\frac{\mm(B_{r}(y))}{\mm(B_R(x))} \right)^\frac{1}{2p-1}\\+\beta'_{N,p}\left(\frac{r}{R}\right)^{\frac{2p}{2p-1}}  \eps^\frac{p}{2p-1} \frac{1}{1-\eta^{\frac{N}{2p-1}}\left(2-\eta\right)^{\frac{2p-N}{2p-1}}}\Bigg)       
    \end{multline}

    \noindent{\sc Case $p\geq N$}. In this range, for the $\eta$ fixed above, we have $\zeta\geq\xi$. Thus, we estimate the sum as 
    \begin{align*}
        \sum_{j=0}^{M-3} \left(\frac\zeta\xi\right)^{M-j-2}&=\frac{\left(\frac\zeta\xi\right)^{M-1}-1}{\frac\zeta\xi-1}-1\leq \frac{\left(1+\frac{\sfd(x,y)}{r}\right)^{\frac{N\log_{2-\eta}\eta+2p-N}{2p-1}}-1}{\frac\zeta\xi-1}\leq \frac{\left(\frac{R}{r}\right)^{\frac{N\log_{2-\eta}\eta+2p-N}{2p-1}}}{\frac\zeta\xi-1}\\
        & \leq \frac{\left(\frac{R}{r}\right)^{\frac{2p-N}{2p-1}}}{\frac\zeta\xi-1},
    \end{align*}
    where, in the last step, we used that, since $\eta\in(0,1)$, $\log_{2-\eta}\eta<0$. Therefore, from \eqref{eq:bg_iterato}, we get
    \begin{multline}
    \label{eq:bg_massaggiato_iterato2}
      \left( \frac{\mm(B_{R_{M-1}}(y_{M-1}))}{\mm(B_R(x))} \right)^\frac{1}{2p-1} \leq\left(\frac{2-\eta}{\eta}\right)^{\frac{N(M-2)}{2p-1}}\Bigg(\left(\frac{\mm(B_{r}(y))}{\mm(B_R(x))} \right)^\frac{1}{2p-1}\\+\beta'_{N,p}\left(\frac{r}{R}\right)^{\frac{N}{2p-1}}  \eps^\frac{p}{2p-1} \frac{1}{\eta^{\frac{N}{2p-1}}\left(2-\eta\right)^{\frac{2p-N}{2p-1}}-1}\Bigg)       
    \end{multline}
    
    \noindent {\sc Conclusion}. By \eqref{eq:bg_massaggiato_iterato1} and \eqref{eq:bg_massaggiato_iterato2}, for every $p>\frac{N}{2}$, we deduce 
    \begin{multline}
    \label{eq:bg_massaggiato_iterato}
      \left( \frac{\mm(B_{R_{M-1}}(y_{M-1}))}{\mm(B_R(x))} \right)^\frac{1}{2p-1} \leq\left(\frac{2-\eta}{\eta}\right)^{\frac{N(M-2)}{2p-1}}\Bigg(\left(\frac{\mm(B_{r}(y))}{\mm(B_R(x))} \right)^\frac{1}{2p-1}+\beta'_{N,p}\beta''_{N,p}\left(\frac{r}{R}\right)^{\frac{N}{2p-1}}  \eps^\frac{p}{2p-1} \Bigg),
    \end{multline}
    where $\beta''_{N,p}\coloneqq \left|1-\eta^{\frac{N}{2p-1}}\left(2-\eta\right)^{\frac{2p-N}{2p-1}}\right|^{-1}$.    
    On the other hand, by Proposition \ref{prop:doubling uniform at x}, choosing $B_{N,p}\leq A_{0,N,p}$ (recall that $B_{r_M}(x)\subset B_{R_{M-1}}(y_{M-1})$, $y_M=x$ and $r_M=\eta\left(2-\eta\right)^{M-2}r$), we estimate
    \begin{equation}
    \label{eq:doubling_massaggiato}
         \left( \frac{\mm(B_{R_{M-1}}(y_{M-1}))}{\mm(B_R(x))} \right)^\frac{1}{2p-1} \geq \eta^{\frac{N}{2p-1}}\left(2-\eta\right)^{\frac{N(M-2)}{2p-1}}\left(\frac{r}{R}\right)^{\frac{N}{2p-1}}\left( 1- 2\beta'_{N,p}\eps^{\frac{p}{2p-1}}\right).
    \end{equation}
    Combining \eqref{eq:bg_massaggiato_iterato} and \eqref{eq:doubling_massaggiato}, we finally obtain the following bound 
    \begin{align}
    \label{eq:final_massaggiata}
        \left(\frac{r}{R}\right)^{\frac{N}{2p-1}}\left( 1- 2\beta'_{N,p}\eps^{\frac{p}{2p-1}}\right)\leq\left(\frac{1}{\eta}\right)^{\frac{N(M-1)}{2p-1}}\Bigg(\left(\frac{\mm(B_{r}(y))}{\mm(B_R(x))} \right)^\frac{1}{2p-1}+\beta'_{N,p}\beta''_{N,p}\left(\frac{r}{R}\right)^{\frac{N}{2p-1}}  \eps^\frac{p}{2p-1} \Bigg) .
    \end{align}
    We now take $\eta$ sufficiently close to $1$ so that it satisfies $\log_{2-\eta}\eta\geq - 2$ as well. Hence, we deduce
    \begin{equation*}
        \eta^{M-1}=(2-\eta)^{(M-1)\log_{2-\eta}\eta}\geq (2-\eta)^{-2(M-1)}\geq \left(1+\frac{\sfd(x,y)}{r}\right)^{-2}\geq\left(\frac{r}{R}\right)^2.
    \end{equation*}
    Therefore, we can continue from \eqref{eq:final_massaggiata} to obtain
    \begin{equation*}
        \left(\frac{\mm(B_{r}(y))}{\mm(B_R(x))} \right)^\frac{1}{2p-1}\geq \left(\frac{r}{R}\right)^{\frac{3N}{2p-1}}\left( 1- 2\beta'_{N,p}\eps^{\frac{p}{2p-1}}\right)-\beta'_{N,p}\beta''_{N,p}\left(\frac{r}{R}\right)^{\frac{2N}{2p-1}}  \eps^\frac{p}{2p-1}.
    \end{equation*}
    Choosing $\beta_{N,p}= \left(\beta'_{N,p}\beta''_{N,p}\right) \vee 2\beta'_{N,p}$, inequality \eqref{eq:ball_chaining_2} follows.
\end{proof}
\begin{remark}\rm
    Note that the inequality \eqref{eq:ball_chaining_2} holds whenever the scaling invariant curvature deficit is smaller than $B_{N,p}$. However, the ratio $\frac{r}R$ can be extremely small (if, for example, $y$ is close to the boundary of $B_R(x)$) and the right-hand side of \eqref{eq:ball_chaining_2} could be negative, making the statement trivial. We will use this in Proposition \ref{prop:Lemma_1.4_aubry}, showing that the right-hand side is indeed strictly positive.\fr
\end{remark}
\section{Myers' diameter estimate}\label{sec:diameter}
In this section, we show Theorem \ref{thm:diameter_estimate} as a byproduct of the more general statement Theorem \ref{thm:diameter_estimate_body}. 
\subsection{Star-shaped partition}
The aim of this section is to show the existence in \cref{prop:partition_with_voronoi} of a countable partition, up to a negligible set, of a geodesic metric measure space into star-shaped sets. 

We start by setting up some notations and proving two preliminary lemmas. Given a metric space $(\X,\sfd)$ and a countable collection $\{x_i\}_{i\in I}$, for every $i,j \in I$ and $\delta \in \R$, we define
\begin{equation}
\label{eq:Tij_notation}
    U_{i,j}^\delta\coloneqq \{ x \in \X:\,\sfd(x,x_i)-\sfd(x,x_j)<\delta \},\qquad\text{and}\qquad T_i^\delta=\bigcap_{j >i} U_{i,j}^\delta \cap \bigcap_{j<i} U_{i,j}^{-\delta}.
\end{equation}
\begin{lemma}
\label{lem:alimentiamo l'ego di Emanuele}
    Let $(\X,\sfd)$ be a metric space and let $\{x_i\}_{i\in I}$ be a countable collection of distinct points in $\X$. For every $\delta \in\R$, it holds
    \begin{equation}
        \X\setminus \bigcup_{i\in I} T_i^\delta \subset \bigcup_{i,j\in I} \left\{x\in\X : \sfd(x,x_i)-\sfd(x,x_j)=\delta \right\}.
    \end{equation}
\end{lemma}
\begin{proof}
    We proceed by induction on the cardinality of $I$. If $\# I =2$, the claim follows directly from the definition of $T_i^\delta$. For $n\in\N$, let us assume that the statement holds true for $\# I=n-1$ and we prove it for a collection of $n$ points. We apply the induction hypothesis on $\{x_i\}_{i=1}^{n-1}$. Denoting by 
    \begin{equation*}
        T_{i}^{\delta,n}\coloneqq  \bigcap_{\substack{j >i\\ j\neq n}} U_{i,j}^\delta \cap \bigcap_{\substack{j<i \\ j\neq n}} U_{i,j}^{-\delta}= \bigcap_{\substack{j >i\\ j\neq n}} U_{i,j}^\delta \cap \bigcap_{j<i} U_{i,j}^{-\delta},\qquad \forall\, i=1,\ldots,n-1,
    \end{equation*}
    we have that
    \begin{equation*}
        \X\setminus \bigcup_{i=1}^{n-1} T_{i}^{\delta,n}\subset\bigcup_{i,j=1}^{n-1} \left\{x\in\X : \sfd(x,x_i)-\sfd(x,x_j)=\delta \right\}.
    \end{equation*}
    Moreover, observe that the following relation holds: 
    \begin{equation*}
        T_i^\delta=T_{i}^{\delta,n}\cap U_{i,n}^\delta, \qquad \forall\, i=1,\ldots,n-1.
    \end{equation*}
    Thus, by standard set operations, we deduce that
    \begin{align*}
        \X\setminus \bigcup_{i=1}^n T_i^\delta 
        &= \left(\X\setminus \bigcup_{i=1}^{n-1} T_i^\delta\right) \cup (\X\setminus T_n^\delta) 
        = \left(\X\setminus \bigcup_{i=1}^{n-1} (T_{i}^{\delta,n}\cap U_{i,n}^\delta)\right) \cup (\X\setminus T_n^\delta)\\
        &= \left(\X\setminus \bigcup_{i=1}^{n-1} T_{i}^{\delta,n}\right) \cup \left(\bigcap_{i=1}^{n-1}(U_{i,n}^\delta)^c\cap (T_n^\delta)^c\right) \\ 
        &\subset \bigcup_{i,j=1}^{n-1} \left\{x\in\X : \sfd(x,x_i)-\sfd(x,x_j)=\delta \right\} \cup \left(\bigcap_{i=1}^{n-1}(U_{i,n}^\delta)^c\cap (T_n^\delta)^c\right).
    \end{align*}
    Now, the conclusion of the proof follows if we prove the inclusion
    \begin{equation}\label{eq:final claim voronoi}
        \bigcap_{i=1}^{n-1}(U_{i,n}^\delta)^c\cap (T_n^\delta)^c\subset \bigcup_{i=1}^{n} \left\{x\in\X : |\sfd(x,x_i)-\sfd(x,x_n)|=\delta \right\}.
    \end{equation}
    Let $p\in \bigcap_{i=1}^{n-1}(U_{i,n}^\delta)^c\cap (T_n^\delta)^c$. On the one hand, since $p\in (U_{i,n}^\delta)^c$,  we must have
    \begin{equation*}
        \sfd(p,x_i)-\sfd(p,x_n)\geq \delta,\qquad\forall\,i=1,\ldots,n-1.
    \end{equation*}
    On the other hand, note that $T_n^\delta=\bigcap_{j=1}^{n-1} U_{n,j}^{-\delta}$, hence $p\in (T_n^\delta)^c$ implies that there exists $j_0\in \{1,\ldots,n-1\}$ such that 
     \begin{equation*}
        \sfd(p,x_n)-\sfd(p,x_{j_0})\geq -\delta.
    \end{equation*}
    This means that $p\in\{x \in \X \colon  |\sfd(x,x_{j_0})-\sfd(x,x_n)|=\delta \}$, and thus \eqref{eq:final claim voronoi} is verified.
\end{proof}
\begin{lemma}
\label{lemma:perturbation_is_starshaped}
Let $(\X,\sfd,\mm)$ be a geodesic metric measure space and let $\{x_i\}_{i\in I}$ be a countable collection of distinct points in $\X$. For any $i,j \in I$ and $\delta > -\sfd(x_i,x_j)$, we have that $U_{i,j}^\delta$ defined in \eqref{eq:Tij_notation} is star-shaped at $x_i$. Moreover, there exists a countable set $S_{i,j} \subset (-\sfd(x_i,x_j),\infty)$, such that 
\begin{equation}
\label{eq:negligible_measure}
\mm\left(\{x\in\X : \sfd(x,x_i)-\sfd(x,x_j)=\delta \}\right)=0,\qquad\forall\delta \in (-\sfd(x_i,x_j),\infty)\setminus S_{i,j}.
\end{equation} 
\end{lemma}
\begin{proof}
    First, we prove that $U_{i,j}^\delta$ is star-shaped at $x_i$, for fixed $i,j \in I$.
    Let $w \in U_{i,j}^\delta$. We consider an arbitrary geodesic $\gamma\colon [0,1] \to \X$ such that $\gamma_0 = x_i$ and $\gamma_1 = w$. We define the continuous function $g(t)\coloneqq \sfd(\gamma_t,x_i)-\sfd(\gamma_t,x_j)-\delta$. We have that $g(0)=-\sfd(x_i,x_j)-\delta <0$ and $g(1)<0$ by definition of $U_{i,j}^\delta$.
    We need to prove that $g(t) <0$ for every $t \in [0,1]$. Since $g$ is continuous, we can assume by contradiction that there exists $t \in [0,1]$ such that $g(t)=0$. This means that there exists $p \in \X$ on the geodesic $\gamma$ so that $\sfd(p,x_i)=\sfd(p,x_j)+\delta$. We have that
    \begin{equation*}
        \sfd(x_j,w)\le \sfd(p,w)+\sfd(p,x_j) = \sfd(p,w)+\sfd(p,x_i) - \delta = \sfd(x_i,w)-\delta,
    \end{equation*}
    where we used that $p \in {\rm Im} \gamma$. This contradicts that $w \in U_{i,j}^\delta$, which is then star-shaped at $x_i$, as claimed. Finally, \eqref{eq:negligible_measure} is a straightforward consequence of the continuity of the distance. 
\end{proof}
\begin{proposition}[Star-shaped essential partition]
\label{prop:partition_with_voronoi}
    Let $(\X,\sfd,\mm)$ be a geodesic metric measure space and consider a $4\pi$-separated set of points $\{x_i\}_{i\in I}$. Then, there is $S\subset\R$ countable and, for all $\delta \in (0,2\pi)\setminus S$, there is a countable collection of open sets $\{T_i^\delta\}_{i \in I}$ such that:
    \begin{itemize}
        \item[{\rm i)}] $T_i^\delta\cap T_j^\delta=\emptyset$ for every $i,j\in I$ with $i\neq j$, and $\mm(\X\setminus \cup_i T_i^\delta)=0$;
        \item[{\rm ii)}] $T_i^\delta$ is star-shaped at $x_i$ for all $i \in I$;
        \item[{\rm iii)}] $B_{2\pi-\delta}(x_i) \subset T_i^\delta$ for all $i \in I$.
    \end{itemize}
\end{proposition}
\begin{proof}
For every $\delta\in \R$, let $\{U_{i,j}^\delta\}_{i,j\in I}$ and $\{T_i^\delta\}_{i\in I}$ the collections of sets associated to $\{x_i\}_{i\in I}$, as defined in \eqref{eq:Tij_notation} (note that $I$ is countable, since $(\X,\sfd)$ is separable). For every $i< j$, let $S_{i,j}$ be the countable set identified by Lemma \ref{lemma:perturbation_is_starshaped} and define the countable set $S\coloneqq  \cup_{i,j} S_{i,j}$. We will prove that $\{T_i^\delta\}$ satisfies i), ii), iii) for $\delta \in (0,\pi)\setminus S$. For ease of notation, we set $T_i\coloneqq T_i^\delta$.

First, we observe that ii) holds thanks to Lemma \ref{lemma:perturbation_is_starshaped}, which says that $T_i$ is the intersection of star-shaped domains at $x_i$. Secondly, for item i), observe that by construction, $\{T_i\}_i$ are pairwise disjoint. Moreover, since $\{x_i\}_{i \in I}$ are $4\pi$-separated, by \eqref{eq:negligible_measure}, combined with Lemma \ref{lem:alimentiamo l'ego di Emanuele}, we also deuce that $\mm(\X\setminus \cup_i T_i)=0$. We are left to prove iii), and we start by proving the first inclusion. If this were not true, there would exist a point $z \in \X$ such that $\sfd(z,x_i) < 2 \pi -\delta$ and an index $j \in I$ such that, either $j >i$ and $z \notin U_{i,j}^\delta$ or $j <i$ and $z \notin U_{i,j}^{-\delta}$. In the first case, it holds
\begin{equation*}
    \delta + \sfd(z,x_j) \le \sfd(z,x_i) < 2\pi -\delta.
\end{equation*}
This implies that $\sfd(z,x_j) < 2\pi-2\delta < 2\pi$, which contradicts the fact that $\{x_i\}_i$ are $4\pi$-separated. In the second case, it holds
\begin{equation*}
    -\delta + \sfd(z,x_j) \le \sfd(z,x_i) < 2\pi -\delta,
\end{equation*}
and for similar reasons we contradict that $\{x_i\}_{i \in I}$ are $4\pi$-separated. 
\end{proof}
\subsection{Spherical integral estimate}
In the upcoming parts, except for our main results (Theorems \ref{thm:diameter_estimate_body}, \ref{thm:diameter_estimate}), we treat the case $K=N-1$. Clearly, all the intermediate technical results admit a straightforward generalization for general $K>0$ by a scaling argument (cf.\ \eqref{eq:scaling}).

We prove next that large balls must have a small spherical integral, provided the scaling invariant integral curvature deficit is small enough. Here, by large, we mean that the radius is exceeding the maximal one in the comparison model. We closely follow \cite[Lemma 4.1]{Aubry2007}.
\begin{lemma}
\label{lemma:small_perimeter_of_large_balls}
    For every $N>1$ and $p>N/2$ there exists a constant $D_{N,p}>0$ such that the following holds.  Let $(\X,\sfd,\mm)$ be an essentially non-branching ${\sf CD}(k,N)$ space for some $k\colon \X\to\R$ admissible and with $\mm(B_\varrho(x))=o(\varrho)$ for some $x\in\X$. Let  $T\subset B_{R}(x)$, for some $R>0$, be a bounded star-shaped set at $x$ satisfying
    \[
        \eps \coloneqq R^2 \left(\frac{\rho_p^k(T,N-1)}{\mm(T)}\right)^\frac{1}{p} \le \left(\frac{\pi}{6}\right)^{2-\frac 1p }.
    \]
    Then, for all $r \in \left(\pi,R\right)$, it holds
    \[
        S_T(r) \le D_{N,p}\eps^{\frac{p(N-1)}{2p-1}}\frac{\mm(T)}{r}. 
    \]
\end{lemma}
\begin{proof}
    Let us apply the monotonicity of item iii) of Lemma \ref{lem:2.2fatto_meglio}, for $0<\lambda\le N-1$ to be chosen later, $\alpha = 1/(2p-1)$ and $0<t<r \le \pi \sqrt{\frac{N-1}{\lambda}}$, to deduce that
    \[
        \left( \frac{S_T(r)}{h_{\lambda ,N}(r)}\right)^{\frac1{2p-1}} - \left( \frac{S_T(t)}{h_{\lambda ,N}(t)}\right)^{\frac1{2p-1}} \le \frac{1}{2p-1}\int_t^r\left( \frac{S_T(s)}{h_{\lambda ,N}(s)}\right)^{\frac{2-2p}{2p-1}}\int_{Q_T(s)} \psi_q(s)\frac{h_q(s)}{h_{\lambda,N}(s)}\,\d \mathfrak q\d s \eqqcolon \frac{{\mathcal J}}{2p-1}.
    \]
    To estimate the term ${\mathcal J}$, we firstly use H\"older inequality to obtain 
    \begin{equation*}
    \begin{aligned}
        {\mathcal J} &= \int_t^r S_T(s)^{\frac{2-2p}{2p-1}}\int_{Q_T(s)} \sin\left(\sqrt{\tfrac{\lambda}{N-1}}s\right)^{\frac{1-N}{2p-1}}\psi_q(s)h_q(s)\,\d \mathfrak q\d s\\
        &\le\int_t^r S_T(s)^{\frac{2-2p}{2p-1}}\left(\int_{Q_T(s)} h_q(s)\,\d \mathfrak q\right)^{\frac{2p-2}{2p-1}}\left(\int_{Q_T(s)} \sin\left(\sqrt{\tfrac{\lambda}{N-1}}s\right)^{1-N}\psi_q(s)^{2p-1}h_q(s)\,\d \mathfrak q\right)^{\frac1{2p-1}}\d s\\
        &= \int_t^r \sin\left(\sqrt{\tfrac{\lambda}{N-1}}s\right)^{-2}\left(\int_{Q_T(s)}\sin\left(\sqrt{\tfrac{\lambda}{N-1}}s\right)^{4p-N-1}\psi_q^{2p-1}(s)h_q(s)\,\d\mathfrak q\right)^{\frac{1}{2p-1}}\,\d s,
    \end{aligned}
    \end{equation*}
    where, in the last equality, we used the definition of $S_T$.
    Using then Proposition \ref{prop:mean vs deficit 1D} (which we can apply by Lemma \ref{lem:average integral tend zero}), we get
    \begin{equation}
    \begin{aligned}
    \label{eq:estimate_of_I}
        {\mathcal J} &\le \alpha_{N,p}^{\frac1{2p-1}} \rho_p^k(T,\lambda)^{\tfrac{1}{2p-1}} \int_t^r  \sin\left(\sqrt{\tfrac{\lambda}{N-1}}s\right)^{-2}\,\d s.
    \end{aligned}
    \end{equation}
    Note that in the application of Proposition \ref{prop:mean vs deficit 1D} we are using both conclusions at the same time. (Whenever $t<s<\frac{\pi}{2}\sqrt{\tfrac{N-1}{\lambda}}\le r$, we simply use $\sin(s)^{4p-N-1} \le 1$). 
    
    We now choose $\lambda$. Define $\eps'\coloneqq \eps^{\frac{p}{2p-1}}$ and let $\lambda =\lambda(\eps,r)$ be defined by $\frac{\lambda}{N-1}=\frac{(\pi-\eps')^2}{r^2}$ (so that, in particular, $r<\frac{\pi}{2}\sqrt{\frac{N-1}{\lambda}}$ also holds). In the sequel, we shall consider $t\in [\frac{\pi}{2(\pi-\eps')}r,r]$. Since $\sin\left(\sqrt{\tfrac{\lambda}{N-1}} s\right) \ge \frac{2}{\pi} \left( \pi -\sqrt{\tfrac{\lambda}{ N-1}} s \right)$ (recall $0\le s \le \frac{\pi}{2}\sqrt{\frac{N-1}{\lambda}}$), we have
    \begin{equation*}
    \begin{aligned}
        \int_t^r \sin\left(\sqrt{\tfrac{\lambda}{N-1}}s\right)^{-2}\d s &\le \frac{\pi^2}{4} \int_t^r\left(\pi-\sqrt{\tfrac{\lambda}{N-1}} s\right)^{-2}\d s = \frac{\pi^2}{4} \frac{r-t}{\left( \pi-\sqrt{\tfrac{\lambda}{N-1}} t \right) \left( \pi-\sqrt{\tfrac{\lambda}{N-1}} r \right)}
        &\le \frac{\pi}{2}\frac{r}{\eps'}.
    \end{aligned}
    \end{equation*}
    As $\lambda \le N -1$, it holds $\rho_p^k(T,\lambda) \le \rho_p^k(T,N-1)$ so that \eqref{eq:estimate_of_I} and the above yield
    \begin{equation*}
        {\mathcal J} \le \alpha_{N,p}^{\frac1{2p-1}} \rho_p^k(T,N-1)^{\frac{1}{2p-1}} \frac{\pi}{2}\frac{r}{\eps'} =  \frac{\pi}{2} \alpha_{N,p}^{\frac1{2p-1}} \mm(T)^{\frac{1}{2p-1}}R^{-\frac{1}{2p-1}}.
    \end{equation*}
    All in all, after multiplication using the identity $ h_{K,\lambda}(r) = \sin\left( \sqrt{\frac{\lambda}{N-1}} r\right)^{\frac{N-1}{2p-1}}$, we obtain
    \begin{align*}
         S_T(r)^{\frac{1}{2p-1}} &\le \left( S_T(t)\frac{ h_{K,\lambda}(r)}{h_{\lambda ,N}(t)}\right)^{\frac1{2p-1}} +   \frac{\pi}{4p-2} \alpha_{N,p}^{\frac1{2p-1}} \mm(T)^{\frac{1}{2p-1}}R^{-\frac{1}{2p-1}}  h_{K,\lambda}(r)^{\frac{p}{2p-1}} \\
         &= S_T(t)^{\frac1{2p-1}} \left(\frac{\eps'}{\sin\left( \frac{t(\pi-\eps')}{r}\right)} \right)^{\frac{N-1}{2p-1}} +   \frac{\pi}{4p-2} \alpha_{N,p}^{\frac1{2p-1}} \mm(T)^{\frac{1}{2p-1}} R^{-\frac{1}{2p-1}}(\eps')^{\frac{N-1}{2p-1}},
    \end{align*}
    having used, in the last inequality, that $ \sin\left( \sqrt{\frac{\lambda}{N-1}} r\right)\le \eps' $. From here, the conclusion of the proof follows by repeating verbatim the last part of the proof of \cite[Lemma 4.1]{Aubry2007}.
\end{proof}
\subsection{Proof of diameter estimate}
We next show that a small scale invariant integral curvature deficit on a sufficiently large star-shaped set is enough to guarantee a diameter estimate. This follows from the combination of the upper bound given by Lemma \ref{lemma:small_perimeter_of_large_balls} and the volume growth given by \cref{prop:volume_growth}. Compare the following with \cite[Lemma 1.4]{Aubry2007}.
\begin{proposition}\label{prop:Lemma_1.4_aubry}
Fix $\overline N>1$. For every $N \in [\overline N,\infty)$ and $p>N/2$ there exist constants $E_{N,p}>0$, $\gamma=\gamma_{\overline N,p}\in \big(0,\frac14\big]$ such that the following holds. Let $(\X,\sfd,\mm)$ be an essentially non-branching ${\sf CD}(k,N)$ space for some $k\colon \X\to\R$ admissible. Assume that $\mm(B_\varrho(z))=o(\varrho)$ at $\mm$-a.e.\ $z\in\X$. Let $T \subset \X$ be star-shaped at $x\in\X$ with $\mm(B_\varrho(x))=o(\varrho)$ and suppose that there exist $R_0,R$ such that $6\pi>R \ge R_0 > \pi$ and
$B_{R_0}(x) \subset T \subset B_{R}(x)$. Furthermore, suppose that
\[
\eps\coloneqq R^2 \left(\frac{\rho_p^k(T,N-1)}{\mm(T)}\right)^\frac{1}{p} \le E_{N,p}\left(1-\frac{\pi}{R_0}\right)^\frac{1}{\gamma}.
\]
Then, it holds $\X \subset B_{\pi + \frac{R_0}2 \eps^{\gamma}}(x)$ and, in particular, $(\X,\sfd)$ is compact.     
\end{proposition}
\begin{proof}
     Fix $\delta\coloneqq \frac{R_0}{4}\eps^\gamma$, for some $\gamma>0$ to be chosen in the sequel. The conclusion $\X \subset B_{\pi + \frac{R_0}2 \eps^\gamma}(x)$ follows if we show that 
    \begin{equation}
    \label{eq:claim_empty_annulus}
        U\coloneqq B_{\pi + \frac{R_0}2 \eps^\gamma+ \delta}(x)\setminus \overline{B_{\pi+\frac{R_0}2\eps^\gamma}(x)}=\emptyset.
    \end{equation}
    Indeed, if by contradiction there is $z \in \X$ with $\sfd(z,x) > \pi +\frac{R_0}2 \eps^\gamma$, then, since $(\X,\sfd)$ is geodesic, there exists an interior point of a geodesic connecting $x$ to $z$ and belonging to $U$. This contradicts the fact that the set is empty. We now prove claim \eqref{eq:claim_empty_annulus}. By contradiction, assume that there exists $y \in U$. 
    The choice of $\delta$ gives that $B_\delta(y) \subset B_{\pi + R_0 \eps^\gamma}(x)\setminus \overline{B_{\pi+\frac{R_0}4 \eps^\gamma}(x)}$. In addition, by hypothesis (taking $E_{N,p}\le 1$) we have $R_0\eps^\gamma\leq R_0-\pi$, $B_{\pi + R_0 \eps^\gamma}(x)\subset T$. Thus,
    By Lemma \ref{lemma:small_perimeter_of_large_balls}, we have 
    \begin{equation}
    \begin{aligned}
    \label{eq:upper_bound_on_the_volume_of_ball}
        \mm(B_\delta(y)) &\le \mm(B_{\pi + R_0 \eps^\gamma}(x)\setminus \overline{B_{\pi+\frac{R_0}4 \eps^\gamma}(x)}) \le \int_{\pi + \frac{R_0}4 \eps^\gamma}^{\pi + R_0 \eps^\gamma} S_T(r)\,\d r \\
        &\le D_{N,p} \mm(T) \eps^{\frac{p(N-1)}{2p-1}} \int_{\pi + \frac{R_0}4 \eps^\gamma}^{\pi + R_0 \eps^\gamma} \frac{1}{r}\,\d r
        \le \frac{3 }{4\pi}D_{N,p} R_0\mm(T)\eps^{\gamma + \frac{p(N-1)}{2p-1}}.
    \end{aligned}
    \end{equation}
    provided that $E_{N,p}\leq \left(\frac{\pi}{6}\right)^{2-1/p}$. In the last inequality, we used that $\int_{\pi + \frac{R_0}4 \eps^\gamma}^{\pi + R_0 \eps^\gamma} \frac{1}{r}\,\d r \le \frac{3}{4\pi}R_0\eps^\gamma$, since $x \mapsto x^{-1}$ is decreasing. Moreover, by monotonicity of $K\mapsto \rho_p^k(T,K)$, we deduce that $R^2\left(\frac{\rho_p^k(T,0)}{\mm(T)}\right)^\frac{1}{p} \le \eps$, therefore, up to taking $E_{N,p}\le A_{0,N,p}$, we can apply \eqref{eq:uniform doubling starshaped} to obtain
    \begin{equation}
    \label{eq:comparison_measure_of_the_ball_and_measure_of_T}
        \mm(B_{R_0}(x)) \ge \mm(T) (1-\beta_{N,p}\eps^{\frac{p}{2p-1}}) \left(\frac{R_0}{R}\right)^N  \ge \frac{\mm(T)}{2} \left(\frac{R_0}{R}\right)^N.
    \end{equation}
    Combining \eqref{eq:upper_bound_on_the_volume_of_ball} and \eqref{eq:comparison_measure_of_the_ball_and_measure_of_T}, and using that $\pi<R_0\leq R<6\pi$, we obtain
    \begin{equation}
    \label{eq:combination}
    \begin{aligned}
    \mm(B_\delta(y)) &
        \le \frac{3 }{2\pi}D_{N,p}R_0\eps^{\gamma + \frac{p(N-1)}{2p-1}}
        \mm(B_{R_0}(x)) \left(\frac{R}{R_0}\right)^N \leq 6^{N+2} D_{N,p}\eps^{\gamma + \frac{p(N-1)}{2p-1}}
        \mm(B_{R_0}(x)).
    \end{aligned}
    \end{equation}
    We now look for a lower bound for $\mm(B_\delta(y))$. Firstly, since $B_{R_0}(x)\subset T$ then $\rho_p^k(B_{R_0}(x),0)\leq \rho_p^k(B_{R_0}(x),N-1)\leq \rho_p^k(T,N-1)$. Hence, by \eqref{eq:comparison_measure_of_the_ball_and_measure_of_T}, we have
    \begin{equation}
    \label{eq:estimate_deficit_T_ball}
    \begin{split}
        R_0^2 \left(\frac{\rho_p^k(B_{R_0}(x),0)}{\mm(B_{R_0}(x))}\right)^\frac{1}{p} & \leq  2^{\frac{1}{p}} \left(\frac{R_0}{R}\right)^{2-\frac{N}{p}} R^2 \left(\frac{\rho_p^k(T,N-1)}{\mm(T)}\right)^\frac{1}{p} \le 2^{\frac{1}{p}} \eps \le E_{N,p},
    \end{split}
    \end{equation}
    where we used that $p>N/2$ and $R_0<R$. Therefore, up to taking $E_{N,p}\leq B_{N,p}$, we are in position to apply \cref{prop:volume_growth} to $B_{R_0}(x)$ to get
		\begin{equation}
		\label{eq:finito_i_nomi}
        \left(\frac{\mm(B_{\delta}(y))}{\mm(B_{R_0}(x))} \right)^\frac{1}{2p-1}\geq \left(\frac{\delta}{R_0}\right)^{\frac{N}{2p-1}}\left(\left(\frac{\delta}{R_0}\right)^{\frac{2N}{2p-1}}\left( 1- \beta_{N,p}\eps^{\frac{p}{2p-1}}\right)-\beta_{N,p}  \eps^\frac{p}{2p-1}\right).
		\end{equation}
		Secondly, by our choice of $\delta$, the right-hand side of \eqref{eq:finito_i_nomi} is strictly positive and such that 
        \begin{equation*}
            \left(\frac{\delta}{R_0}\right)^{\frac{N}{2p-1}}\left(\left(\frac{\delta}{R_0}\right)^{\frac{2N}{2p-1}}\left( 1- \beta_{N,p}\eps^{\frac{p}{2p-1}}\right)-\beta_{N,p}  \eps^\frac{p}{2p-1}\right) \geq \,\eps^{\frac{3N\gamma}{2p-1}},
        \end{equation*}
        provided that $\gamma\leq\frac14$ and $E_{N,p}$ is chosen sufficiently small. Thus, \eqref{eq:finito_i_nomi} and the estimate above yield
    \begin{equation}
    \label{eq:volume_growth_for_small_eps}
        \frac{\mm(B_\delta(y))}{\mm(B_{R_0}(x))} \ge   \eps^{3N\gamma}.
    \end{equation}
    	In conclusion, the combination of \eqref{eq:combination} and  \eqref{eq:volume_growth_for_small_eps} gives that 
    \begin{equation*}
        \eps^{3N\gamma} \le \frac{\mm(B_\delta(y))}{\mm(B_{R_0}(x))} \le 6^{N+2} D_{N,p}\eps^{\gamma + \frac{p(N-1)}{2p-1}}.
    \end{equation*}

    This gives a contradiction if $\gamma + \frac{p(N-1)}{2p-1} > 3 N\gamma$. Since $N\geq \overline N$, we can choose $\gamma=\gamma_{\overline N,p}$ such that we reach the desired contradiction for every $N\geq \overline N$.
\end{proof}
Note that an easy consequence of the previous proposition is that ${\rm diam}(\X)\le 2\pi+R_0\eps^\gamma$. We aim at improving this diameter bound, and we can do so by iterating the previous argument.
\begin{corollary}\label{cor:diameter_estimate}
Fix $\overline N>1$. For every $N \in [ \overline N,\infty)$ and $p>N/2$ there exist constants $F_{N,p}>0$, $\gamma=\gamma_{\overline N,p}\in \big(0,\frac18\big]$ such that the following holds. Let $(\X,\sfd,\mm)$ be an essentially non-branching ${\sf CD}(k,N)$ space for some $k\colon \X\to\R$ admissible. Assume that $\mm(B_\varrho(z))=o(\varrho)$ at $\mm$-a.e.\ $z\in\X$. Let $T \subset \X$ be star-shaped at $x\in\X$ with $\mm(B_\varrho(x))=o(\varrho)$ and suppose that there exist $R_0,R$ such that $6\pi>R \ge R_0 > \pi$ and
$B_{R_0}(x) \subset T \subset B_{R}(x)$. Furthermore, suppose that
\[
\eps\coloneqq R^2 \left(\frac{\rho_p^k(T,N-1)}{\mm(T)}\right)^\frac{1}{p} \le F_{N,p}\left(1-\frac{\pi}{R_0}\right)^\frac{1}{\gamma}.
\]
Then, it holds ${\rm diam}(\X) \le \pi (1+3\pi\eps^{2\gamma})$.     
\end{corollary}

\begin{proof}
We fix $\gamma=\frac{\gamma_{\overline N,p}}2$ from Proposition \ref{prop:Lemma_1.4_aubry}. By contradiction, we assume that there exist two points $z_1$, $z_2 \in \X$ such that $\sfd(z_1,z_2)> \pi +\frac{R_0}2 \eps^{2\gamma}$. Firstly, by Proposition \ref{prop:Lemma_1.4_aubry}, choosing $F_{N,p}\le E_{N,p}$, we immediately deduce that $\X\subset B_{\pi+\frac{R_0}2\eps^\gamma}(x)$. In particular, $\sfd(x,z_1) < \pi + \frac{R_0}2 \eps^\gamma$ and thus $B_{R_0-\pi-\frac{R_0}2 \eps^\gamma}(x) \subset B_{R_0}(z_1)$. Secondly, we claim that 
		\begin{equation}
		\label{eq:defect_bound}
        R_0^2 \left(\frac{\rho^k_p(B_{R_0}(z_1),N-1)}{\mm(B_{R_0}(z_1))}\right)^{\frac{1}{p}} \leq 2^{\frac1p}F_{N,p}\left(1-\frac{\pi}{R_0}\right)^{\frac{1}{2\gamma}}.
    \end{equation}
    Indeed, on the one hand, since $B_{R_0-\pi-\frac{R_0}2 \eps^\gamma}(x) \subset B_{R_0}(z_1)$, one has $\mm(B_{R_0}(z_1))\ge \mm(B_{R_0-\pi-\frac{R_0}2 \eps^\gamma}(x))$ and, applying \cref{prop:doubling uniform at x}, we then obtain
    \begin{equation}
		\label{eq:basta}
        \frac{\mm(B_{R_0}(z_1))}{\mm(B_{R_0}(x))} \ge \frac{\mm\left(B_{R_0-\pi-\frac{R_0}2 \eps^\gamma}(x)\right)}{\mm(B_{R_0}(x))} \ge \left(\frac{R_0- \pi -\frac{R_0}2 \eps^\gamma}{R_0}\right)^N \ge \frac{1}{2}\left( \frac{R_0-\pi}{R_0} \right)^N.
    \end{equation}
		Note that here the assumptions of \cref{prop:doubling uniform at x} are verified since $R_0\le R$ and $\eps\le F_{N,p}\le A_{p,0,N}$. On the other hand, by Proposition \ref{prop:Lemma_1.4_aubry} and since $R_0\eps^\gamma\le R_0-\pi$, we also have that $B_{R_0}(z_1)\subset B_{\pi+\frac{R_0}2\eps^\gamma}(x)\subset B_{R_0}(x)$, which implies that $\rho_p^k(B_{R_0}(z_1),N-1) \le \rho_p^k(B_{R_0}(x),N-1)$, by monotonicity $ \rho_p^k(\cdot,N-1)$ for set inclusions. Combining this observation with \eqref{eq:basta}, we infer 
    \begin{equation*}
    \begin{aligned}
        R_0^2 \left(\frac{\rho_p^k(B_{R_0}(z_1),N-1)}{\mm(B_{R_0}(z_1))}\right)^\frac{1}{p} &\le R_0^2 \left(\frac{\mm(B_{R_0}(x))}{\mm(B_{R_0}(z_1))}\right)^{\frac{1}{p}}\left(\frac{\rho_p^k(B_{R_0}(x),N-1)}{\mm(B_{R_0}(x))}\right)^\frac{1}{p}  \le 2^{\frac1p}\left(  \frac{R_0}{R_0-\pi}\right)^{\frac{N}{p}} \eps,
    \end{aligned}
    \end{equation*}
    where, in the last inequality, we used \eqref{eq:estimate_deficit_T_ball} (which is available since we chose $F_{N,p}\leq E_{N,p}$). Now, by the assumption on $\eps$, we have that 
    \begin{equation*}
        \left(  \frac{R_0}{R_0-\pi}\right)^{\frac{N}{p}} \eps\leq F_{N,p}\left(1-\frac{\pi}{R_0}\right)^{\frac{1}{\gamma}-\frac{N}{p}}\leq F_{N,p}\left(1-\frac{\pi}{R_0}\right)^{\frac{1}{2\gamma}},
    \end{equation*}
    since $\frac{1}{\gamma}-\frac{N}p\geq\frac{1}{2\gamma}$ by the constraint on $p$ and the choice of $\gamma\leq\frac14$. This proves claim \eqref{eq:defect_bound}. 
    
    To conclude the proof, note that, thanks to the estimate \eqref{eq:defect_bound}, we can apply Proposition \ref{prop:Lemma_1.4_aubry} with $T=B_{R_0}(z_1)$ star-shaped set at $z_1$, up to taking a smaller $F_{N,p}$. This shows that $\X\subset B_{\pi+\frac{R_0}2\eps^{2\gamma}}(z_1)$, which is in contradiction with the fact that $\sfd(z_1,z_2)>\pi+\frac{R_0}2\eps^{2\gamma}$.
\end{proof}
\begin{remark}\label{rem:gamma ge 2}\rm
    In the previous statements, observe that if $\overline N=2$, $\gamma$ can be chosen equal to $1/10$ independently of $p$. Additionally, by the proof of Proposition \ref{prop:Lemma_1.4_aubry}, we must have
    \begin{equation*}
        \gamma <\frac{p(N-1)}{(3N-1)(2p-1)},
    \end{equation*}
    hence if $N\to 1$, then $\gamma\to 0$. In this limiting case, we cannot deduce a diameter estimate. Indeed, for $N\geq 1$, the metric measure space $([0,2\pi],|\cdot|,r^{N-1}\Leb^1)$ is ${\sf CD}(0,N)$. The integral curvature deficit from $N-1$ can be taken arbitrarily small as $N\downarrow 1$, but the diameter is constantly equal to $2\pi$.
\fr
\end{remark}
Finally, we prove the Myers' diameter estimate, cf.\ Theorem \ref{thm:diameter_estimate}. The proof makes use of the essential partition of the metric measure space into a family of star-shaped set (cf.\ \cref{prop:partition_with_voronoi}). The key point is that, provided $\rho_p^k(\X,K)$ is small enough, Corollary \ref{cor:diameter_estimate} can be applied to at least one star-shaped set, giving in turn the desired diameter estimate.
\begin{theorem}
\label{thm:diameter_estimate_body}
    Fix $\overline N>1$. For every $N\in [\overline N,\infty),K>0$ and $p>N/2$ there exist constants $C_{K,N,\overline N,p},\gamma_{\overline N,p}>0$ such that the following holds. Let $(\X,\sfd,\mm)$ be an essentially non-branching ${\sf CD}(k,N)$ for some $k\colon \X\to\R$ admissible. Assume that $\mm(B_\varrho(x))=o(\varrho)$ at $\mm$-a.e.\ $x\in \X$.  If $\rho^k_p(\X,K)< \infty$  then $\mm(\X)<\infty$. Moreover, if 
    \[
        \rho^k_p(\X,K)\le \frac{\mm(\X)}{C_{K,N,\overline N,p}},
    \]
    then $(\X,\sfd)$ is compact and it holds
    \begin{equation}
        {\rm diam}(\X) \le \pi\sqrt{\frac{N-1}{K}}\left(1 + C_{K,N,\overline N,p} \left(\frac{\rho^k_p(\X,K)}{\mm(\X)}\right)^{2\gamma_{\overline N,p}}\right).\label{eq:diameter estim body}
    \end{equation}
\end{theorem}
\begin{proof}
    We fix $\gamma=\gamma_{\overline N,p}$ from Corollary \ref{cor:diameter_estimate}. We shall prove the statement for $K=N-1$ and subdivide the proof into different steps. The statement and the estimate \eqref{eq:diameter estim body} would then follow for general $K>0$ by a scaling argument (cf. \eqref{eq:scaling}) with the constant $C_{K,N,\overline N,p} = C_{N-1,N,\overline N,p} \left(\frac{N-1}{K}\right)^{p\gamma}$.

    \noindent\textsc{Step 1: Construction of a star-shaped partition}. Consider a countable maximal collection $\{z_i\}_{i \in I}$ that is $4\pi$-separated in $(\X,\sfd)$ and fix $\eta\in (0,\pi/64)$. Then, by the assumptions, for every $i\in I$ there exists $x_i\in B_\eta(z_i)$ such that $\mm(B_\varrho(x_i))=o(\varrho)$. The new collection of points $\{x_i\}_{i\in I}$ is $(4\pi-2\eta)$-separated however it may be non-maximal. Hence, choosing $\delta$ in a full-measure subset of $(\pi/16,\pi/8)$, and denoting by $\{T_i^{\delta-2\eta}\}_{i\in I}$ the sets defined in \eqref{eq:Tij_notation} using the collection $\{x_i\}_{i\in I}$, Proposition \ref{prop:partition_with_voronoi} holds. In particular, for every $i\in I$, $B_{2\pi-\delta+\eta}(x_i)\subset T_i^{\delta-2\eta}$.
    
    We now show that also the inclusion $T_i^{\delta-2\eta}\subset B_{4\pi+\delta+\eta}(x_i)$ holds. Suppose by contradiction that this is not the case, so there exists a point $x\in T_i^{\delta-2\eta}\setminus B_{4\pi+\delta+\eta}(x_i)$. On the one hand, this implies 
    \begin{equation}
    \label{eq:contradiction_net}
        \sfd(x,z_i)\geq \sfd(x,x_i)-\sfd(x_i,z_i) \geq 4\pi +\delta >4\pi.
    \end{equation}
    On the other hand, since $x\in T_i^{\delta-2\eta}=\bigcap_{j>i}\tilde U_{i,j}^{\delta-2\eta}\cap \bigcap_{j<i}\tilde U_{i,j}^{-\delta+2\eta}$, using the definition of the sets $\tilde U_{i,j}^{\delta-2\eta}$, the fact that $x_i\in B_\eta(z_i)$, and the triangle inequality, we have, for every $j<i$,
    \begin{equation*}
        \sfd(x,z_j) \geq \sfd(x,x_j)-\sfd(z_j,x_j) > \sfd(x,x_i)-\delta+\eta \geq \sfd(x,z_i)-\sfd(z_i,x_i) -\delta +\eta> \sfd(x,z_i)-\delta \geq 4\pi,
    \end{equation*}
    where, in the last inequality, we used \eqref{eq:contradiction_net}. Analogously, for every $j<i$, we have
    \begin{equation*}
        \sfd(x,z_j)\geq \sfd(x,x_j) -\eta > \sfd(x,x_i)+\delta-3\eta \geq  \sfd(x,z_i)+\delta -4\eta> 4\pi.
    \end{equation*}
    This is in contradiction with the fact that the original family $\{z_i\}_{i\in I}$ was $4\pi$-maximally separated.
    
    \noindent\textsc{Step 2: $\mm(\X)<\infty$}. We consider the family $\{x_i\}_{i\in I}$ from the previous step. Note that  $\mm(B_\varrho(x_i))=o(\varrho)$ for every $i\in I$. In addition, denoting by $\{T_i\}_{i \in I}$ the sets defined in \eqref{eq:Tij_notation} for a suitable parameter $\delta\in (\pi/32,\pi/4)$, they satisfy items i) and ii) of Proposition \ref{prop:partition_with_voronoi} and it holds
    \begin{equation}
        B_{2\pi-2\delta}(x_i)\subset T_i \subset B_{4\pi+2\delta}(x_i).\label{eq:inclusione old iii}
    \end{equation}
    We also define
    \[\alpha\coloneqq \inf_{i \in I} \left( \frac{\rho_p^k(T_i,N-1)}{\mm(T_i)} \right)^{\frac{1}{p}}.
    \]
    We first consider the case $\alpha=0$. There exists $T_j$ such that Corollary \ref{cor:diameter_estimate} can be applied with $R_0\coloneqq 2\pi-2\delta\geq \frac32\pi>\pi$, $R\coloneqq 4\pi+2\delta<6\pi$, thus in particular proving that ${\rm diam}(\X)<\infty$, so also $\mm(\X)<\infty$. We now consider the case $\alpha >0$. In this case,
    \begin{equation}
    \label{eq:local_to_global_deficit}
        \infty > \rho_p^k(\X,N-1) = \sum_{i=1}^\infty \rho_p^k(T_i,N-1) \ge \alpha^p \sum_{i=1}^\infty \mm(T_i) =\alpha^p \mm(\X),
    \end{equation}
    where we repeatedly used item i) of Proposition \ref{prop:partition_with_voronoi}. Thus, the conclusion $\mm(\X)<\infty$ holds.
    
    \noindent\textsc{Step 3: Diameter estimate}. Let $F_{N,p}>0$ be the constant given by Corollary \ref{cor:diameter_estimate}. We shall prove \eqref{eq:diameter estim body} with the choice $C_{N-1,N,\overline N,p}\coloneqq 2 \left( \frac{4^{1/\gamma} 36 \pi^2 }{F_{N,p}} \right)^p$. We first claim that $\alpha < \frac{F_{N,p}}{4^{1/\gamma} 36 \pi^2 }$. If not, observe that $\alpha \ge \frac{F_{N,p}}{4^{1/\gamma} 36 \pi^2 }$ and \eqref{eq:local_to_global_deficit} yields 
    \begin{equation*}
        \mm(\X) \le \frac{1}{2} C_{N-1,N,\overline N,p} \,\rho_p^k(\X,N-1) < C_{N-1,N,\overline N,p} \,\rho_p^k(\X,N-1).
    \end{equation*} 
    This is not possible by the assumptions, and the claim is proved. Now, by definition of $\alpha$, we have that there exists $j \in I$ such that $T_{j}$ satisfies 
    \begin{equation*}
        \left(\frac{\rho_p^k(T_{j},N-1)}{\mm(T_{j})} \right)^{\frac{1}{p}} < \frac{F_{N,p}}{4^{1/\gamma} 36 \pi^2 }.    
    \end{equation*}
    Recalling \eqref{eq:inclusione old iii}, we are in position to apply Corollary \ref{cor:diameter_estimate} to the set $T_{j}$ with $R\coloneqq 4\pi+2\delta<6\pi$ and $R_0\coloneqq 2\pi-2\delta > \pi$. Indeed, note that $1-\frac{\pi}{R_0}\geq \frac14$ by the choice of $R_0$, and then
 	 \begin{equation*}
        R^2 \left(\frac{\rho_p^k(T_{j},N-1)}{\mm(T_{j})} \right)^{\frac{1}{p}} \le \frac{F_{N,p}}{4^{1/\gamma} 36 \pi^2} (4\pi+\delta)^2 \le  \frac{F_{N,p}}{4^{1/\gamma}}  \le F_{N,p} \left( 1-\frac{\pi}{R_0} \right)^{1/\gamma}.
    \end{equation*}
    Thus, by Corollary \ref{cor:diameter_estimate}, $\X$ is compact and $\#I$ is finite. Hence, $\exists\,j_0\in I$ such that $T_{j_0}$ is a minimizer for $\alpha$ and for which Corollary \ref{cor:diameter_estimate} holds. In particular, this, together with \eqref{eq:local_to_global_deficit}, gives the following 
    \begin{equation*}
    \begin{aligned}
      {\rm diam}(\X) & \le  \pi  \left(1+3\pi\Bigg(R^2 \Bigg( \frac{\rho_p^k(T_{j_0},N-1)}{\mm(T_{j_0})}\Bigg)^{\frac{1}{p}}
    \Bigg)^{2\gamma} \right)\le \pi \Bigg(1+ \frac{(6\pi)^{2\gamma+1}}p \left( \frac{\rho_p^k(\X,N-1)}{\mm(\X)}  \right)^{2\gamma} \Bigg).
    \end{aligned}
    \end{equation*}
    Finally, recall that if $\overline N=2$, $\gamma$ can be chosen equal to $\frac1{10}$, cf.\ Remark \ref{rem:gamma ge 2}.
\end{proof}
\section{Cheng's comparison principle}
\subsection{Dirichlet eigenvalues in model intervals} 
Given $K \in \R$, $N \in \N$, $p\in(1,\infty)$ and $r \in \left( 0, \pi\sqrt{\frac{N-1}{K^+}}\right)$, we denote by $\lambda_p(K,N,r)$  the first non-zero Dirichlet eigenvalue for the $p$-Laplacian on a geodesic ball of radius $r$ in the model space of dimension $N$ and constant (sectional) curvature $K$. Equivalently, since eigenfunctions on geodesic balls of the model spaces are radial, we can compute $\lambda_p(K,N,r)$ by optimizing over one-dimensional $(K,N)$-model intervals. This allows for the generalization of the definition of $\lambda_p(K,N,r)$ for any $N\in (1,\infty)$. More precisely, we set
\begin{equation}
\lambda_p(K,N,r) := \inf \left\{ \frac{\int_0^r |\phi'|^p  h_{K,N}\,\d t}{\int_0^r|\phi|^ph_{K,N}\,\d t } \colon \phi \in {\sf AC}_{loc}(0,r) \text{ so that} \begin{array}{l}
    \int_0^r|\phi|^p h_{K,N}\,\d t <\infty  \\
   \lim_{t\to r^-}\phi(t)=0
\end{array}
\right\}.
\label{eq:p eigenvalue}
\end{equation}
A $p$-Dirichlet eigenfunction for $\lambda_p(K,N,r)$ is a solution $\phi \in {\sf AC}_{loc}(0,r)$ with $|\phi|,|\phi'|\in L^p(h_{K,N})$ of
\begin{equation}
   -\Delta_p \phi = \lambda_p(K,N,r)\phi|\phi|^{p-2}, \qquad \text{in the sense of distributions on }(0,r),
\label{eq:model p eigenfunction}
\end{equation}
where $\Delta_p \phi \in L^1_{loc}(0,r)$ is the $p$-Laplacian defined via integration by parts 
\[
-\int_0^r \Delta_p \phi g h_{K,N}\,\d t = \int_0^r |\phi'|^{p-2}\phi' g' h_{K,N} \, \d t,\qquad \forall\,g \in C^\infty_c(0,r).
\]
We collect basic properties of Dirichlet eigenfunctions on weighted intervals in Appendix \ref{apdx:eigenfunctions}. 
\subsection{Proof of Cheng's comparison principle}
We prove our last main result. Recall the variational definition of $\lambda_p(\Omega)$ for $\Omega\subset\X$ open, given in \eqref{eq:variational_eigenvalue}.
\begin{proof}[Proof of Theorem \ref{thm:cheng}]
    Set for brevity $B\coloneqq B_r(x)$ and consider the disintegration given by \cref{thm:localization} relative to the $1$-Lipschitz function $\sfd_x$. Then, we can write
    \begin{equation}
    \mm= \int \mm_q \, \d \mathfrak q,\label{eq:reintegration proof}
    \end{equation}
    and, for $\mathfrak q$-a.e.\ $q\in Q$, $\mm_q = g(q,\cdot)_\sharp(h_q \Leb^1)$, where $h_q:[0,r_q]\to [0,\infty)$ is a ${\sf CD}( k\circ g(q,\cdot) ,N)$ density on ${\rm Dom}(g(q,\cdot))$. Let us consider $\phi\in {\sf AC}_{loc}(0,r)$ to be a non-negative $p$-Dirichlet eigenfunction realizing $\lambda_p(K,N,r)$. We can suppose, by scaling, that $\phi(0)=1$ so that it holds $0\le\phi\le 1$ (recall that $\phi$ is non-increasing by Lemma \ref{lem:phi non-increasing}). The function $u \coloneqq \phi \circ \sfd_x$ is a competitor in the definition of $ \lambda_p(B)$ by the chain rule of Lemma \ref{lem:chain rule eigenfunction}, and we get
   \begin{equation} \label{eq:mwug_ineq}
        \int_B|u|^p\,\d\mm \cdot \lambda_p(B)  \le \int_B |Du|_p^p \,\d \mm  \overset{\eqref{eq:chain rule eigenf}}{\le} \int_B|\phi'|^p\circ \sfd_x \,\d \mm,
   \end{equation}
    where, here and after, we set $|\phi'| = 0$ on $\{t \in \R \colon \nexists |\phi'|(t)\}$. We set $\tilde r_q:=r_q\wedge r$ for every $q\in Q$ and we estimate the right-hand side using the localization as follows
    {\allowdisplaybreaks
    \begin{align*}
        \int_B |\phi'|^p\circ \sfd_x \,\d \mm &=\iint_0^{ \tilde r_q}|\phi'|^ph_q\, \d t \d \mathfrak q  = \iint_0^{ \tilde r_q} \phi' (|\phi'|^{p-2}\phi' h_q)\, \d t \d \mathfrak q  \\
        &= \int \left ( \phi |\phi'|^{p-2}\phi' h_q\Big|_0^{ \tilde r_q} - \int_0^{ \tilde r_q} \phi \Big( ( |\phi'|^{p-2}\phi' )'+ |\phi'|^{p-2}\phi' (\log h_q)'\Big) h_q\,\d t\right)\d \mathfrak q\\
        &\le  -\iint_0^{ \tilde r_q} \phi \Big( ( |\phi'|^{p-2}\phi' )'+ |\phi'|^{p-2}\phi' (\log h_q)'\Big) h_q\,\d t\d \mathfrak q \\
        &\overset{\eqref{eq:p ODE}}{=} \lambda_p(K,N,r)\iint_0^{ \tilde r_q}|\phi|^ph_q \,\d t\d \mathfrak q - \iint_0^{ \tilde r_q} \phi |\phi'|^{p-2}\phi'\big( (\log h_q)'- H_{K,N}\big) h_q \,\d t\d \mathfrak q \\
         &\le  \lambda_p(K,N,r)\int_B|u|^p\,\d \mm  - \iint_0^{ \tilde r_q}  |\phi'|^{p-2} \phi ' \psi_q h_q\,\d t\d \mathfrak q,
    \end{align*}}
    \!\!having repeatedly used that $\phi'\leq 0$ and that $|\phi'|^{p-2}\phi'(0^+)=0$. Combining everything, we get
    \begin{align}
    \label{eq:ineq_with_model}
         \Lambda\coloneqq\frac{\int_B|\phi'|^p\circ \sfd_x \,\d \mm}{\int_B |u|^p\,\d \mm}  \le \lambda_p(K,N,r) + \frac{\iint_0^{ \tilde r_q}  |\phi'|^{p-1} \psi_q h_q\,\d t\d \mathfrak q}{\int_B |u|^p \,\d \mm}.
    \end{align}
    Let us now handle the second term of the right-hand side. Let $\bar s =\bar s(K,N,r) >0$ be the first $s \in (0,r)$ such that $\phi(s) \ge 1/2$ for all $s \in (0,\bar s)$ (recall $\phi(0)=1$ and it is non-increasing), then
    \begin{equation*}
        \int_B |u|^p\,\d \mm \ge  \int_{B_{\bar s}(x)} |\phi|^p\circ \sfd_x\,\d \mm \ge \frac 1
{2^p}\mm(B_{\bar s}(x)).
    \end{equation*}
    Using the above and the H\"older inequality, we estimate
    \begin{equation}
    \begin{split}
    \label{eq:second_term}
        \frac{\iint_0^{ \tilde r_q}  |\phi'|^{p-1} \psi_q h_q\,\d t\d \mathfrak q}{\int_B |u|^p \,\d \mm} &\le \left( \frac{ \iint_0^ {\tilde r_q} \psi_q ^p h_q\,\d t\d \mathfrak q  }{\int_B |u|^p\,\d\mm } \right)^{\frac 1p} \left(\frac{ \iint_0^{ \tilde r_q}|\phi '|^p h_q\,\d t\d \mathfrak q }{\int_B |u|^p\,\d\mm} \right)^{\frac {p-1}p} \le 2\left( \frac{ \iint_0^{ \tilde r_q} \psi_q ^p h_q\,\d t\d \mathfrak q  }{\mm(B_{\bar s}(x))} \right)^{\frac 1p} \Lambda^{\frac {p-1}p}.
    \end{split}
    \end{equation}
    In addition, if $\varepsilon \coloneqq \varepsilon_{K,N,r,p_0, p} \le A_{K,N,\bar p,r}$ as given by \cref{prop:doubling uniform at x}, the assumption \eqref{eq:p cheng small deficit} guarantees
    \begin{equation}
        \frac{1}{\mm(B_{\bar s}(x))} \le 2\frac{v_{K,N}(r)}{v_{K,N}(\bar s)} \frac{1}{\mm(B_r(x))}.
    \label{eq:BG uniform eps}
    \end{equation}
    Thus, we continue from \eqref{eq:second_term} and obtain
    \begin{equation*}
        \frac{\iint_0^{ \tilde r_q}  |\phi'|^{p-1} \psi_q h_q\,\d t\d \mathfrak q}{\int_B |u|^p \,\d \mm} \le \frac{C_{K,N,r}}{\mm(B_r(x))^{\frac1p}}\left(  \iint_0^{ \tilde r_q} \psi_q ^p h_q\,\d t\d \mathfrak q  \right)^{\frac 1p} \Lambda^{\frac {p-1}p},
    \end{equation*}
    for a suitable constant $C_{K,N,r}>0$. Putting together the above inequality with \eqref{eq:ineq_with_model}, we deduce
    \begin{equation}
    \Lambda \le  \lambda_p(K,N,r) + \frac{ C_{K,N,r}}{\mm(B_r(x))^{\frac 1p }}\left(\iint_0^{\tilde r_q} \psi_q ^{p} h_q\,\d t\d  \mathfrak q  \right)^\frac 1p \Lambda^{\frac {p-1}p}.    \label{eq:main cheng estim}
    \end{equation}
    We can now estimate
    \begin{align*}
       &\mm(B_r(x))^{ \frac{1}{2\bar p-1}-\frac 1p} \left( \iint_0^{\tilde r_q} \psi_q ^{p} h_q\,\d t\d \mathfrak q\right)^{\frac 1p}\le   \left( \iint_0^{\tilde r_q} \psi_q ^{2\bar p-1} h_q\,\d t\d \mathfrak q \right)^{\frac 1{2\bar p-1}}\\
        &\qquad = \left( \iint_0^{ \frac{\pi}{2}\sqrt{\frac{N-1}{K^+}}\wedge \tilde r_q} \psi_q ^{2\bar p-1} h_q\,\d t\d \mathfrak q  + \iint_{ \frac{\pi}{2}\sqrt{\frac{N-1}{K^+}}\wedge \tilde r_q}^{\tilde r_q} \left( \frac{h_{K^+,N}}{h_{K^+,N}}\right)^{\frac{4\bar p-N-1}{N-1}}\psi_q^{2\bar p-1} h_q\,\d t\d \mathfrak q  \right)^{\frac 1{2\bar p-1}}\\
        &\qquad \le \left( \iint_0^{ \frac{\pi}{2}\sqrt{\frac{N-1}{K^+}}\wedge \tilde r_q} \psi_q ^{2\bar p-1} h_q\,\d t\d \mathfrak q  \right)^{\frac 1{2\bar p-1}}\\
        &\qquad\,\qquad\qquad\qquad\qquad+ h_{K^+,N}(r)^{\frac{N+1-4\bar p}{(N-1)(2\bar p-1)}}\left( \iint_{ \frac{\pi}{2}\sqrt{\frac{N-1}{K^+}}\wedge \tilde r_q}^{\tilde r_q} h_{K^+,N}^{\frac{4\bar p-N-1}{N-1}}\psi_q^{2\bar p-1} h_q\,\d t\d \mathfrak q  \right)^{\frac 1{2\bar p-1}},
    \end{align*}
    where, for the first inequality, we used H\"older, and for the second one, the subbaditivity of the concave function $t^{\frac1{2\bar p-1}}$ and the monotonicity of $h_{K^+,N}$ on $\left(  \frac{\pi}{2}\sqrt{\frac{N-1}{K^+}},r\right)$, with the usual convention, here and after, that any term containing $K^+$ is not present if $K\le 0$.  We can thus invoke \cref{prop:mean vs deficit 1D} (which we can apply by Lemma \ref{lem:average integral tend zero}) and, by the disintegration formula \eqref{eq:reintegration proof}, we get
    \begin{align*} 
     \left( \iint_0^{\tilde r_q} \psi_q ^{p} h_q\,\d t\d \mathfrak q\right)^{\frac 1p} &\le \alpha_{N,\bar p}\left( 1 +  h_{K^+,N}(r)^{\frac{N+1-4\bar p}{(N-1)(2\bar p-1)}} \right) r^{\frac{1}{2\bar p-1}}\mm(B_r(x))^{\frac 1p} \left( \frac{\rho_{\bar p}^k(B_r(x),K)}{\mm(B_r(x)}\right)^{\frac{1}{2\bar p-1}}.
    \end{align*}
   Thanks to the above, we can continue estimating in \eqref{eq:main cheng estim} to deduce
    \begin{equation}\label{eq:cheng intermediate}
        \Lambda \le  \lambda_p(K,N,r) + C_{K,N,r, p,p_0}  \left( \frac{\rho_{\bar p}^k(B_r(x),K)}{\mm(B_r(x))}\right)^{\frac{1}{2\bar p-1}}\Lambda^{\frac {p-1}p} ,
    \end{equation}
    for a suitable constant $C_{K,N,r, p,p_0} >0$. In particular, we get
    \[
        \lambda_p(B_r(x)) \overset{\eqref{eq:mwug_ineq},\eqref{eq:cheng intermediate}}{\le}   \lambda_p(K,N,r) +  C_{K,N,r, p,p_0}  \left( \frac{\rho_{\bar p}^k(B_r(x),K)}{\mm(B_r(x)}\right)^{\frac{1}{2\bar p-1}}\Lambda^{\frac {p-1}p}.
    \]   
    From here, we see that the proof will be concluded if we show that $\Lambda$ is bounded above by a positive constant depending only on $K,N,r, p,p_0$. We claim that this is the case. Indeed, using Young inequality  $ab \le a^p/p + b^{p'}/{p'}$ in  \eqref{eq:cheng intermediate} for $p' = p/(p-1)$, we get after manipulations
    \[
       \Lambda \le p \lambda_p(K,N,r) +C_{K,N,r, p,p_0}\left(\frac{ \rho^k_{\bar p}(B_r(x),K)}{\mm(B_r(x))}\right)^\frac{p}{2\bar p-1},
    \]
     for a suitable constant $C_{K,N,r, p,p_0}>0$. Finally, possibly decreasing $\eps$ so that $\frac{ \rho^k_{\bar p}(B_r(x),K)}{\mm(B_r(x))} \le 1$, we have that $\Lambda$ is bounded above by a constant depending only on $K,N,r, p,p_0$, concluding the proof.
\end{proof}
\appendix
\section{One-dimensional \texorpdfstring{${\sf CD}$}{CD} densities}\label{apdx:1D density}
In this appendix we study some useful estimates and regularization properties of one-dimensional densities admitting a variable Ricci curvature lower bound. We shall always consider closed intervals $[a,b]$ with $-\infty\le a<b\le +\infty$, with the standard convention if $a,b$ are not finite. 
\begin{definition}\label{def:1D CD}
    Let $h \colon [a,b]\rightarrow [0,+\infty)$ be a function, let $\kappa \colon [a,b] \to \R$ be admissible and $N>1$. We say that $h$ is a ${\sf CD}(\kappa,N)$ density on $[a,b]$, provided for every $x_0,x_1\in [a,b]$ and $t \in [0,1]$ it holds
    \begin{equation}
        h(tx_1 + (1-t)x_0)^{\frac{1}{N-1}} \ge \sigma^{(1-t)}_{k^-_{\gamma},N-1}(|x_1-x_0|) h(x_0)^{\frac1{N-1}} + \sigma_{k^+_{\gamma},N-1}^{(t)}(|x_1-x_0|) h(x_1)^{\frac{1}{N-1}},
    \label{eq:curvature dimension inequality}
    \end{equation}
    where $\gamma_t = (1-t)x_0+tx_1$. 
\end{definition}
The above estimate is a curvature dimension inequality that will naturally appear when analyzing the regularity of the disintegration in the proof of Theorem \ref{thm:localization}. Next, we reconcile this notion with usual curvature dimension conditions as defined in Definition \ref{def:CDkN} (compare to \cite[Lemma A.2]{CavallettiMilman21}).
\begin{lemma}\label{lem:cd density implica cd space}
    Let  $\kappa \colon [a,b] \to \R$ be admissible and fix $N>1$. Let $h \in L^1(a,b)$ be a non-negative function. The following are equivalent:
    \begin{itemize}
        \item[{\rm i)}] there is a representative of $h$ that is a ${\sf CD}(\kappa,N)$ density on $[a,b]$;
        \item[{\rm ii)}] $([a,b],|\cdot|,h \Leb^1)$ is a  ${\sf CD}(\kappa,N)$ space as in Definition \ref{def:CDkN};
        \item[{\rm iii)}] there is a locally Lipschitz continuous representative of $h$ satisfying 
        \[
            (\log h)'' + \frac{|(\log h)'|^2}{N-1} \le - \kappa\qquad\text{in the sense of distributions on } (a,b).
         \]
    \end{itemize}
    In particular, if any of the above holds, then $\log h$ is locally semi-concave and $h$ is locally Lipschitz.
\end{lemma}
\begin{proof}
    The equivalence between items i) and ii) can be proven following the same strategy of \cite[Theorem A.2]{CavallettiMilman21}, taking into account \cite[Prop.\ 7.2]{Ket17}. Under either one of the assumptions, since $\kappa$ is locally bounded from below, for every compact interval $I\subset [a,b]$, $(I,|\cdot|,h \Leb^1)$ is ${\sf CD}(K_0,N)$ for some $K_0 \in \R$. This implies that there exists a representative so that $\log h$ is semiconcave and $h$ is locally Lipschitz.
    The equivalence between i) and iii) follows by showing \cite[Eq. (A.1)]{CavallettiMilman21} in the sense of distributions. This holds as all the terms appearing in the calculus rules are functions. The conclusion follow by \cite[Cor.\ 3.13]{Ket17}.  
\end{proof}
We next study regularization properties of ${\sf CD}$ densities (cf.\ \cite[Prop.\ A.10]{CavallettiMilman21}).
\begin{proposition}\label{prop:convolution 1D density}
    Let $\kappa \colon [a,b]\to \R$ be admissible with $|a|,|b|<\infty$, $N>1$ and let $h \colon [a,b]\to[0,+\infty)$ be a ${\sf CD}(\kappa,N)$ density on $(a,b)$. Consider, for every $\eps \in \left(0, \frac{b-a}{2}\right)$, a smooth mollifier $\eta_\eps$ supported on $[-\eps,\eps]$ with $\int \eta_\eps =1$ defined as $\eta_\eps\coloneqq \eps^{-1}\eta(\cdot /\eps)$ for some $\eta \in C_c^\infty(-1,1)$. Define
    \[
        h_\eps \coloneqq \exp\left( (\log h)\ast \eta_\eps\right),\qquad
        \kappa_\eps \coloneqq \kappa \ast \eta_\eps, \qquad
        I_\eps \coloneqq (a+\eps,b-\eps).
    \]
    Then, $h_\eps$ is a smooth ${\sf CD}(\kappa_\eps,N)$ density on $I_\eps$. Furthermore, it holds:
    \begin{itemize}
        \item[{\rm i)}] $h_\eps \to h$ locally uniformly on $(a,b)$ as $\eps\downarrow 0$;
        \item[{\rm ii)}] $(\log h_\eps)' \to (\log h)' $ a.e.\ on $(a,b)$ as $\eps\downarrow 0$;
        \item[{\rm iii)}] if also $\limi_{r\downarrow a}\dashint_a^r h(r)\,\d r =0$, then we have $\limi_{\eps\downarrow 0}h_\eps(a+\eps)=0$.
    \end{itemize}
\end{proposition}
\begin{proof}
   First, observe that by Jensen inequality (applied to $\eta_\eps \mathcal L^1$) we deduce 
     \begin{equation}
     \label{eq:Jensen_CDkn_density}
         \left[(\log h_\eps)'(t)\right]^2 = \left[((\log h)' \ast \eta_\eps)(t)\right]^2 \le (\left[(\log h)'\right]^2 \ast \eta_\eps)(t)\qquad\text{for a.e. }t \in I_\eps.
     \end{equation}
     Recall now that, thanks to Lemma \ref{lem:cd density implica cd space}, we have
     \begin{equation*}
         (\log h)'' + \frac{[(\log h)']^2}{N-1} \le - \kappa \qquad\text{in the sense of distributions on } I_\eps
     \end{equation*}
     After taking the convolution for $\eps>0$ on both sides and using \eqref{eq:Jensen_CDkn_density}, we get
     \begin{equation*}
         (\log h_\eps)''(t) + \frac{[(\log h_\eps)'(t)]^2}{N-1} \le - \kappa \ast \eta_\eps(t),\qquad\text{a.e. on }t \in I_\eps
     \end{equation*}
     In particular, $h_\eps$ is a smooth ${\sf CD}(\kappa\ast \eta_\eps, N)$ density on $I_\eps$. 

     Conclusion i) simply follows from the assumed continuity properties of $h,\kappa$. Conclusion ii) follows similarly since $\log h$ is locally Lipschitz. Finally, conclusion iii) instead follows estimating 
     \begin{equation}
    \label{eq:estimate_heps}
        0 \le h_\eps(a+\eps) \le \exp \left( \int \log h(t)\eta_\eps(a+\eps-t)\,\d t \right) \le \int h(t) \eta_\eps (a+\eps - t)\,\d t,
    \end{equation}
    where, in the last inequality, we used Jensen's inequality applied to the probability measure $\eta_\eps(\eps - \cdot) \mathcal{L}^1$.  Since $\eta \le L$ for every $L >0$, we have  by definition $\eta_\eps \le L \eps^{-1}$. Hence, $\int h(t) \eta_\eps (a+\eps - t)\,\d t \le L\eps^{-1}\int_{a}^{a+2\eps} h(t)\,\d t$ that, combined with \eqref{eq:estimate_heps} and the assumption on $h$, gives $\limi_{\eps\downarrow 0} h_\eps(a+\eps) =0$.
\end{proof}
\section{Dirichlet \texorpdfstring{$p$}{p}-eigenfunctions on weighted intervals}\label{apdx:eigenfunctions}
We collect here basic properties of $p$-Dirichlet eigenfunctions on $(K,N)$-model intervals. These are well known, but we include them to be self-consistent. 
\begin{lemma} \label{lem:regularity phi}
    Let $K \in \R$ and $N >1$ and $r \in \left(0,\pi\sqrt{\frac{N-1}{K^+}}\right)$. There is a non-negative minimizer $\phi \in {\sf AC}_{loc}(0,r) $ of  \eqref{eq:p eigenvalue}. Furthermore, we have $ |\phi'|^{p-2}\phi'  \in W^{1,1}_{loc}(0,r)$ and it holds
    \begin{equation}
    \left(|\phi'|^{p-2}\phi'\right)'+|\phi'|^{p-2}\phi' (\log h_{K,N})'=-\lambda_p(K,N,r)  |\phi|^{p-2}\phi, \qquad \text{a.e.\ in $(0,r).$}
    \label{eq:p ODE}
    \end{equation}
\end{lemma}
\begin{proof}
    The existence of a minimizer $\phi$ follows by a standard compactness argument. The fact that this can be taken non-negative follows instead by the invariance of the optimization problem \eqref{eq:p eigenvalue} for taking the absolute value. The fact that $\phi$ satisfies \eqref{eq:model p eigenfunction} in a distributional sense is also a standard variational computation. By definition of $\Delta_p \phi$, we have for all $g \in C^\infty_c(0,r)$ that
    \begin{align*}
       - \int_0^r \Delta_p \phi g h_{K,N} \,\d t  &= - |\phi'|^{p-2}\phi' g h_{K,N}\Big|_0^r + \int_0^r (|\phi'|^{p-2}\phi') g' h_{K,N} \, \d t  + \int_0^r |\phi'|^{p-2}\phi' g (\log h_{K,N})'h_{K,N} \, \d t 
    \end{align*}
    Since $g$ is of compact support  on $(0,r)$ we get $- |\phi'|^{p-2}\phi' g h_{K,N}\Big|_0^r =0$. Hence, by the Euler-Lagrange equations \eqref{eq:model p eigenfunction}, we deduce, for all  $g \in C^\infty_c(0,r)$,
    \[
        \lambda_p(K,N,r)\int_0^r \phi|\phi|^{p-2}g h_{K,N}\,\d t =  \int_0^r \big( |\phi'|^{p-2}\phi' g' h_{K,N} + |\phi'|^{p-2}\phi' g (\log h_{K,N})' \big) h_{K,N} \, \d t.
    \]
    Since  $w \coloneqq |\phi'|^{p-2}\phi' h_{K,N} \in L^1_{loc}(0,r)$ and $ \phi|\phi|^{p-2 }h_{K,N}$ as well as that $|\phi'|^{p-2}\phi' (\log h_{K,N})' h_{K,N} \in L^1_{loc}(0,r) $, we thus deduce that $w \in W^{1,1}_{loc}(0,r)$. However, $h_{K,N} \in C^\infty(0,r)$ and $h_{K,N},h_{K,N}^{-1},(\log\,h_{K,N})'\in L^\infty_{loc}(0,r)$, hence a chain rule argument gives also $|\phi'|^{p-2}\phi'\in W^{1,1}_{loc}(0,r)  $ and \eqref{eq:p ODE} follows. 
\end{proof}
   Knowing that $\phi$ satisfies \eqref{eq:p ODE} gives access to standard elliptic regularity theory to deduce that $\phi$ is also Lipschitz continuous. The important fact is that $h_{K,N}>0$ on $(0,r]$, and it is smooth on $[0,r]$, and vanishes at zero. For instance, this can be deduced by showing that $\phi$ is sufficiently integrable by a Moser iteration argument and then recasting \eqref{eq:p ODE} as a degenerate non-homogeneous elliptic equation. The Lipschitz regularity then follows by appealing to \cite[Thm.\ 2]{DiBenedetto83} (and subsequent remark there). We are now going to show that then $\phi$ is also non-increasing (if non-negative).
\begin{lemma}\label{lem:phi non-increasing}
     Let $K \in \R$ and $N >1$ and  $r \in \left(0,\pi\sqrt{\frac{N-1}{K^+}}\right)$. Let $\phi \in {\sf AC}_{loc}(0,r) $ be a non-negative minimizer of \eqref{eq:p eigenvalue}. Then, there exists $\lim_{t\downarrow 0} ( |\phi'|^{p-2}\phi')(t) =0$ and $\phi$ is non-increasing.
\end{lemma}
\begin{proof}
    Thanks to the previous lemma, we know that $w=h_{K,N} |\phi'|^{p-2}\phi' \in W^{1,1}_{loc}(0,r)$ and 
    \[
    - (h_{K,N} |\phi'|^{p-2}\phi')' = \lambda_p(K,N,r) \phi|\phi|^{p-2}h_{K,N},\qquad \text{a.e.\ on }(0,r).
    \]
    For any $t <r$ we can thus integrate the above from zero to $t$ using that $h_{K,N}(0)=0$ to obtain
    \[
  - h_{K,N} (t)\left(|\phi'|^{p-2}\phi'\right)(t)  = - \int_0^t  ( h_{K,N} |\phi'|^{p-2}\phi')'\,\d s = \lambda_p(K,N,r) \int_0^t \phi|\phi|^{p-2}h_{K,N}\,\d s \ge 0,
    \]
   having used, for the latter, the assumption that the $p$-Dirichlet eigenfunction $\phi$ is non-negative and that $h_{K,N}>0$ on $(0,r)$. This implies that $\phi'\le 0$ a.e.\ on $(0,r)$, hence it is non-increasing.  To conclude the proof, it is enough to observe that
   \begin{align*}
        0 \le - \limi_{t\downarrow 0 } |\phi'|^{p-2}\phi'(t) \le - \lims_{t\downarrow 0 } |\phi'(t)|^{p-2}\phi'(t) &=\lims_{t\downarrow 0 } \frac{\lambda_p(K,N,r) \int_0^t \phi|\phi|^{p-2}h_{K,N}\,\d s}{ h_{K,N} (t)} \\
        &= \lambda_p(K,N,r) \lims_{t\downarrow 0 } \frac{\phi(t)|\phi|^{p-2}(t)h_{K,N}(t)}{h_{K,N}'(t)} =0,
   \end{align*}
   having used the De l'H\^opital's rule and then that $\lim_{t\downarrow 0}\frac{h_{K,N}(t)}{h_{K,N}'(t)} =0$. 
\end{proof}
We conclude with the following chain rule technical lemma. 
\begin{lemma}\label{lem:chain rule eigenfunction}
    Let $K\in\R,N >1, p \in (1,\infty)$. Let $(\X,\sfd,\mm)$ be a metric measure space, let $B\subset \X$ be a ball of radius $r \in \left(0,\pi\sqrt{\frac{N-1}{K^+}}\right)$ centred at $x \in \X$.  Let $\phi$ be a non-negative minimizer of \eqref{eq:p eigenvalue}. Then, $\phi \circ \sfd_x \in W^{1,p}_0(B)$ and it holds
    \begin{equation}
    |D(\phi\circ\sfd_x)|_p \le |\phi'|\circ\sfd_x,\quad \mm\text{-a.e.\ on }B.\label{eq:chain rule eigenf}
    \end{equation}
    where $|\phi'|\circ \sfd_x $ is arbitrarily non-negatively defined on $\sfd_x^{-1}( \{t\in  \R \colon \nexists |\phi'|(t) \})$.
\end{lemma}
\begin{proof}
By locality of the minimal $p$-weak upper gradient, the conclusion \eqref{eq:chain rule eigenf} is well-posed. Then, notice that $u\coloneqq \phi\circ \sfd_x$ is a composition of Lipschitz functions with $u,|Du|_p \in L^p(B)$. Thus $u \in W^{1,p}(B)$ and \eqref{eq:chain rule eigenf} holds by the chain rule for Sobolev functions (\cite{Gigli12}) (for this claim, local absolutely continuity of $\phi$ would be enough, see \cite[Lemma 4.5]{NobiliViolo24_PZ}).  We now prove that $u \in W^{1,p}_{0}(B)$.  For every $\varepsilon>0$, we define $\phi_\varepsilon\coloneqq  (\phi-\varepsilon )\vee 0$. By a direct computation, we have 
    \begin{equation*}
    |\phi_\varepsilon(t)-\phi(t)|=\varepsilon \nchi_{\{ \phi \ge \varepsilon \}}(t) + |\phi(t)|\nchi_{\{ \phi < \varepsilon \}}(t) \le \varepsilon
    \end{equation*}
    for all $t \in (0,r)$ and 
    \begin{equation*}
    |\phi_\varepsilon'(t)-\phi'(t)|=|\phi'(t)|\nchi_{\{ 0 < \phi <\varepsilon \}},
    \end{equation*}
    for a.e.\ $t \in (0,r)$. We define $u_\varepsilon\coloneqq  \phi_\eps \circ \sfd_x$ on $B$ that satisfies $u_\varepsilon \in  {\rm Lip}_{bs}(B)$ by construction for every $\eps>0$ small enough. Moreover, we have that $\|u-u_\varepsilon\|^p_{L^p(B)} \le \int_B \left| \phi-\phi_\varepsilon \right|^p \circ \sfd_x \,\d \mm \le \varepsilon^p \mm(B)$ that tends to $0$ as $\varepsilon \to 0$. Next, observe that  
    \begin{align*}
    \| |D (u - u_\varepsilon)|_p \|^p_{L^p(B)} &\le \int_B |D((\phi-\phi_\varepsilon)\circ \sfd_x)|_p^p\,\d\mm \overset{\eqref{eq:chain rule eigenf}}{\le} \int_B |\phi'-\phi_\varepsilon'|^p\circ \sfd_x\,\d \mm \\
    &=\int_{\{ 0 < \phi \circ \sfd_x < \varepsilon \}\cap B} |\phi'|^p \circ \sfd_x\,\d \mm \le {\rm Lip}(\phi)^p\,\mm(\{ 0 < \phi < \varepsilon \} \cap B),
    \end{align*}
    that tends to $0$ as $\varepsilon\to 0$. Thus also $u \in W^{1,p}_0(B)$ concluding the proof.
\end{proof}
\noindent\textbf{Acknowledgments}.
E.C. acknowledges the support by the European Union's Horizon 2020 research and innovation programme (Grant agreement No. 948021). F.N is a member of INDAM-GNAMPA and acknowledges support by the European Union (ERC, ConFine, 101078057) and the MIUR Excellence Department Project
awarded to the Department of Mathematics, University of Pisa, CUP I57G22000700001. We thank C. Ketterer and I. Y. Violo for fruitful discussions around the topic of this paper. Part of this work was carried out while all the authors were visitors at the Centro di Ricerca Matematica Ennio De Giorgi under the 2025 Research in Pairs program. The warm hospitality and stimulating atmosphere are gratefully acknowledged.


\end{document}